\documentclass[11pt]{amsart}
\usepackage[english]{babel}

    \newcommand{\PARENS}[1]{\left(#1\right)}
    \newcommand{\ccases}[1]{\begin{cases}#1\end{cases}}
  \usepackage{amsmath}
  \usepackage{mathrsfs}
  \usepackage{upgreek}

\usepackage{bigstrut}
\usepackage{array}
\usepackage{mathtools,arydshln}\mathtoolsset{centercolon}
\usepackage{marvosym}

\usepackage[total={18cm,24.5cm},top=2cm, left=1.5cm]{geometry}
\usepackage[toc,page]{appendix}
\usepackage{tocvsec2}

\DeclareFontFamily{U}{MnSymbolC}{}
\DeclareSymbolFont{MnSyC}{U}{MnSymbolC}{m}{n}
\DeclareFontShape{U}{MnSymbolC}{m}{n}{
    <-6>  MnSymbolC5
   <6-7>  MnSymbolC6
   <7-8>  MnSymbolC7
   <8-9>  MnSymbolC8
   <9-10> MnSymbolC9
  <10-12> MnSymbolC10
  <12->   MnSymbolC12}{}
\DeclareMathSymbol{\intprod}{\mathbin}{MnSyC}{'267}

\newcommand{\coloneq}{:=}

\newcommand{\ee}{\boldsymbol{e}}
\newcommand{\n}{\boldsymbol{n}}
\newcommand{\m}{\boldsymbol{m}}
\newcommand{\x}{\boldsymbol{x}}
\newcommand{\y}{\boldsymbol{y}}
\newcommand{\z}{\boldsymbol{z}}
\newcommand{\q}{{\boldsymbol{\upalpha}}}
\newcommand{\kk}{\boldsymbol{k}}
\newcommand{\ele}{\boldsymbol{\ell}}

\newcommand{\SC}{\operatorname{SC}}
\newcommand{\dd}{\operatorname{d}}
\newcommand{\Exp}[1]{\operatorname{e}^{#1}}

\newcommand{\diag}{\operatorname{diag}}

\newcommand{\Ds}{\mathscr D}

\newcommand{\N}{\mathbb{N}}
\newcommand{\Z}{\mathbb{Z}}
\newcommand{\R}{\mathbb{R}}
\newcommand{\C}{\mathbb{C}}
\newcommand{\I}{\mathbb{I}}
\newcommand{\T}{\mathbb{T}}

\newtheorem{pro}{Proposition}[subsection]
\newtheorem{lemma}{Lemma}[subsection]
\newtheorem{definition}{Definition}[subsection]
\newtheorem{theorem}{Theorem}[subsection]
\newtheorem{cor}{Corollary}[subsection]
\numberwithin{equation}{subsection}

\begin{document}
\title[Multivariate Orthogonality and Toda]{Multivariate orthogonal polynomials\\ and integrable systems}
\author{Gerardo Ariznabarreta}\address{Departamento de F\'{\i}sica Te\'{o}rica II (M\'{e}todos Matem\'{a}ticos de la F\'{\i}sica), Universidad Complutense de Madrid, 28040-Madrid, Spain}
\email{gariznab@ucm.es}
\thanks{GA thanks economical support from the Universidad Complutense de Madrid  Program ``Ayudas para Becas y Contratos Complutenses Predoctorales en Espa\~{n}a 2011"}
\author{Manuel Ma\~{n}as}
\email{manuel.manas@ucm.es}
\thanks{MM thanks economical support from the Spanish ``Ministerio de Econom\'{\i}a y Competitividad" research project MTM2012-36732-C03-01,  \emph{Ortogonalidad y aproximaci\'{o}n; teor\'{\i}a y aplicaciones}}
\keywords{Multivariate orthogonal polynomials, Borel--Gauss factorization, quasi-determinants, Christoffel--Darboux kernels, Darboux transformations,  Christoffel formula, quasi tau matrices, kernel polynomials, integrable hierarchies, Toda equations, KP equations}
\subjclass{15A23,33C45,37K10,37L60,42C05,46L55}

\begin{abstract}
Multivariate orthogonal polynomials in $D$ real dimensions are   considered from the perspective of the Cholesky factorization of a moment matrix.   The approach allows for the construction of corresponding multivariate orthogonal polynomials, associated second kind functions, Jacobi type matrices and associated three term relations and also  Christoffel--Darboux formul{\ae}. The multivariate orthogonal polynomials, its second kind functions and the corresponding Christoffel--Darboux kernels are shown to be quasi-determinants --as well as Schur complements--  of bordered truncations of the moment matrix; quasi-tau functions are introduced. It is proven that  the second kind functions are  multivariate Cauchy transforms of the multivariate orthogonal polynomials. Discrete and continuous deformations of the measure lead to Toda type integrable hierarchy, being the corresponding flows  described through Lax and Zakharov--Shabat equations; bilinear equations are found. Varying size matrix nonlinear partial difference and differential equations of the 2D Toda lattice type are shown to be solved by matrix coefficients of the multivariate orthogonal polynomials. The discrete flows, which are shown to be connected with a Gauss--Borel factorization of the Jacobi type matrices and its quasi-determinants,  lead to  expressions for the multivariate orthogonal polynomials and its second kind functions in terms of shifted quasi-tau matrices, which generalize to the multidimensional realm those that relate the Baker and adjoint Baker functions with  ratios of Miwa shifted $\tau$-functions in the 1D scenario.
In this context, the multivariate extension of the elementary Darboux transformation is given in terms of quasi-determinants of matrices built up by the evaluation, at a poised set of nodes lying in an appropriate hyperplane in $\R^D$, of the multivariate orthogonal polynomials. The multivariate Christoffel formula for the iteration of $m$ elementary Darboux transformations is given as a quasi-determinant.
 It is shown, using congruences in the space of semi-infinite matrices, that the discrete and continuous flows are intimately connected and determine nonlinear partial difference-differential equations that involves only one site in the integrable lattice behaving as a Kadomstev--Petviashvili type system. Finally,  a brief discussion of measures with a particular linear isometry invariance and some of its consequences for the corresponding multivariate polynomials are given. In particular, it is shown that the Toda times that preserve the invariance condition lay in a secant variety of the Veronese variety of the fixed point set of the linear isometry.

\end{abstract}

\maketitle
\tableofcontents

\section{Introduction}

This paper is devoted to the study of the interrelation between the theory of  Multivariate Orthogonal Polynomials, or orthogonal polynomials on several variables, and the theory of Integrable Systems of Toda type. We perform this analysis with the aid of the Gauss--Borel factorization of the moment matrix, that  in this case reduces to a Cholesky factorization. To understand better the situation we now proceed to give a brief description on the state of the art for multivariate orthogonal polynomials, then we recall some facts regarding Toda equations and integrable systems. As we use quasi-determinants in a number of places we have also included some comments regarding this subject. Finally, we describe the aims, results and the layout of the paper.

\subsection{On multivariate orthogonal polynomials}
Multivariate orthogonal polynomials has been a subject of study for many years, we refer the reader to the book by Charles F. Dunkl and Yuan Xu \cite{Dunkl} in where the authors of this paper enjoyed learning diverse aspects of multivariate orthogonality.  The authors presents in that book the  general theory and emphasizes the classical types of orthogonal polynomials whose weight functions are supported on standard domains such as the cube, the simplex, the sphere and the ball. It also focuses on those of Gaussian type, for which fairly explicit formul{\ae} exist. Another general source could be the lecture notes \cite{xu4} which provide an introduction to orthogonal polynomials of several  variables. It  covers the basic theory but deal mostly with examples, paying special attention to those orthogonal polynomials associated with classical type weight functions supported on the standard domains, for which fairly explicit formul{\ae} exist.

The  recurrence relation for orthogonal polynomials in several variables was studied by Xu in \cite{xu0}, while in \cite{xu1} he linked multivariate orthogonal polynomials with a commutative family of self-adjoint operators and the spectral theorem was used to show the existence of a three term relation for the orthogonal polynomials. He discusses in \cite{xu2}  how the three term relation leads to the construction of multivariate orthogonal polynomials and cubature formul{\ae}.  Xu considers in  \cite{xu8}  polynomial subspaces that contain discrete multivariate orthogonal polynomials  with respect to the bilinear form are identified and shows that the discrete orthogonal polynomials still satisfy a three-term relation and that Favard's theorem holds. Explicit three term recurrence relations for the determination of multivariate orthogonal polynomials, which allow for the derivation of evaluation algorithms of finite series of these polynomials, were obtained \cite{barrio}.  Recursive three-term recurrence for the multivariate Jacobi polynomials on a simplex are explicitly given in \cite{waldron}. In \cite{rodal}  several relations linking differences of bivariate discrete orthogonal polynomials and polynomials are given. We should also mention the work \cite{delgado} in where bivariate real valued polynomials orthogonal with respect to a positive linear functional are considered; interestingly  the authors discuss orthogonal polynomials associated with positive definite block Hankel matrices whose entries are also Hankel and  develop methods for constructing such matrices.

Multivariate Pad\'{e} approximants cubature formul{\ae} were considered in  \cite{beno}. The analysis of orthogonal polynomials and cubature formul{\ae} on the unit ball, the standard simplex, and the unit sphere  \cite{xu6 } lead to conclude  the strong connection of orthogonal structures and cubature formul{\ae} for these three regions.  In \cite{li-xu} Tchebychev polynomials were obtained using symmetric and antisymmetric sums of exponentials and Gaussian cubatures were found, which exist very rarely in higher dimension.
The paper \cite{xu3} presents a systematic study of the common zeros of polynomials in several variables which are related to higher dimensional quadrature.
 In \cite{knese}  a description of polynomials orthogonal on the bicircle and polycircle and their relation to bounded analytic functions on the polydisk is given. Important in this work is a Christoffel--Darboux like formula which in the bivariate case can be related to stable polynomials, Bernstein--Szeg\H{o} measures and gives a new proof of Ando  theorem in operator theory.

Karlin and McGregor  \cite{karlin}  and Milch \cite{milch} discussed interesting examples of multivariate Hahn and Krawt\-chouk polynomials related to growth birth and death processes.
A  study  of two-variable orthogonal polynomials associated with a moment functional satisfying the two-variable analogue of the Pearson differential equation and an extension of some of the usual characterizations of the classical orthogonal polynomials in one variable was found \cite{fernandez}. In \cite{1} semiclassical orthogonal polynomials in two variables are defined as the orthogonal polynomials associated with a quasi definite linear functional satisfying a matrix Pearson-type differential equation, semiclassical functionals
are characterized by means of the analogue of the structure relation in one variable and non trivial examples of semiclassical orthogonal polynomials in two variables where given. Xu and Ilieva gave in \cite{xu7} a characterization of all second order difference operators of several variables that have discrete orthogonal polynomials as eigenfunctions is given and under some mild assumptions, they give a complete solution of the problem.

 In \cite{fernandez2} the authors analyze a bilinear form obtained by adding a Dirac mass to a positive definite moment functional defined in the linear space of polynomials in several variables.
 A new proof of Gasper theorem on the positivity of sums of triple products on Jacobi polynomials  was given in \cite{geronimo};
this theorem plays an important role in setting up a convolution structure for Jacobi polynomials, the correlation operator is an operator on the $N$-sphere looking for its eigenfunction expansion in various angular momentum sectors leads to Gasper's theorem and to the Koornwinder--Schwartz product formul{\ae} for the biangle which constitutes an extension of Gasper's theorem to the bivariate case.
Xu discusses in \cite{xu5} monomial orthogonal polynomials with respect to the weight function  on the unit sphere as well as for the related weight functions on the unit ball and on the standard simplex getting explicit formul{\ae} for the $L^2$ norm and explicit expansions in terms of known orthonormal basis.

Let us mention that even there are Maple libraries --MOPS-- to treat with multivariate orthogonal polynomials, in particular  computes Jack, Hermite, Laguerre, and Jacobi multivariate polynomials, as well as eigenvalue statistics for the Hermite, Laguerre, and Jacobi ensembles
of random matrix theory \cite{dumitriua}.

\subsection{On the Toda equations}

Some times the name given to equations or theorems do not correspond exactly to the original discoverers of the result. This is one of those cases.

The Toda equations can be traced back to the classical {\oe}uvre \emph{Le\c{c}ons sur la Th\'{e}orie G\'{e}n\'{e}rale des Surfaces} published in 1915,  by the French Mathematician Jean Gaston Darboux \cite{darboux-toda}; when he studies the Laplace method on reduction and invariance properties associated with the canonical hyperbolic
equation $\Delta r=0$  where $\Delta$  is a second order real  hyperbolic operator. In the \emph{Deuxi\`{e}me Partie. Livre IV. Chapitre II. La m\'{e}thode de Laplace} if we go to number 336  we discover recursion (27) (page 30 of \cite{darboux-toda}) for the invariants $h_k$ and $h_{k-1}$ --of equations $E_k$ in number 335--:
\begin{align*}
h_{k+1}+h_{k-1}  =2h_k-\frac{\partial^2\log h_k}{\partial x\partial y},
\end{align*}
that for the new dependent variable $q_k$ given by
\begin{align*}
  h_k=\Exp{q_{k-1}-q_k}
\end{align*}
reads as the 2D Toda equation
\begin{align*}
  \frac{\partial^2 q_k}{\partial x\partial y}=\Exp{q_k-q_{k+1}}-\Exp{q_{k-1}-q_k},
\end{align*}
that for the dimensional reduction $x=\pm y=t$ simplifies to the  Toda equation
\begin{align*}
  \frac{\partial^2 q_k}{\partial t^2}=\Exp{q_k-q_{k+1}}-\Exp{q_{k-1}-q_k}.
\end{align*}

Then, more than a half a century later the Japanese Physicist Morikazu Toda,
introduced  \cite{toda} a simple model, that he named as exponential lattice, for a one-dimensional crystal in solid state physics with a nearest neighbor interaction, with potential  $\phi(r)= \frac{a}{ b}\Exp{-r}+a r+c$, $a, b>0$ , such that the particles are subject to
 \begin{align*}
\frac{\dd p_k}{\dd t}(t) &= \Exp{-(q_k(t) - q_{k-1}(t))} - \Exp{-(q_{k+1}(t) - q_k(t))}, \\
\frac{\dd q_k}{\dd t} (t) &= p_k(t),
\end{align*}
where $q_k$ and $p_k$ are the displacement of the $k$-th particle from its equilibrium position, and  its momentum (here the mass is set equal to the unity).
In \cite{toda} exact solutions where obtained in terms of the Jacobian elliptic functions, it was also shown that the system has $N$ normal modes and the expansion due to vibration  of the chain was discussed. Later on \cite{toda1} relations between this nonlinear exponential lattice, the Boussinesq equation and the Korteweg--de Vries equation showed up and therefrom  two-soliton solutions were given in each case for both the head-on and the overtaking collisions.

The Toda lattice is a completely integrable system \emph{\`{a} la Liouville}  as it was shown in 1974 first by Michel Hen\'{o}n \cite{henon} and then by Hermann Flaschka \cite{flaschka2} in terms  Flaschka's variables:
\begin{align*}
  a_k(t) =& \frac{1}{2} \Exp{-\frac{q_{k+1}(t) - q_{k}(t)}{2}}, &b_k(t) = -\frac{1}{2} p_k(t),
\end{align*}
so that 1D Toda equations are written as follows
 \begin{align}
\dot a_k(t) &= a_k(t) \big(b_{k+1}(t)-b_k(t)\big), \\
\dot b_k(t) &= 2 \big(a_k(t)^2-a_{k-1}(t)^2\big).
\end{align}
These equations can be reformulated as  the  Lax equation
$\dot L(t) = [P(t), L(t)]$;  the  \emph{Lax pair},  $L$ and $P$, are linear operators in the space $\ell^2(\mathbb{Z})$  of square summable sequences given by
 \begin{align}
(L(t) f)_k &= a_k(t) f_{k+1} + a_{k-1}(t) f_{k-1} + b_k(t) f_k,  & L(t)=\PARENS{\begin{matrix*}[r]
  b_0(t) & a_0(t)&0 &0 &\dots\\
  a_0(t)& b_1(t)& a_1(t)&0&\dots\\
  0& a_1(t)&b_2(t)& a_2(t)&\dots\\
  &\ddots &\ddots&\ddots
\end{matrix*}},\\
(P(t) f)_k &= a_k(t) f_{k+1} - a_{k-1}(t) f_{k-1}, & P(t)=\PARENS{\begin{matrix*}[c]
 0 & a_0(t)&0 &0 &\dots\\
  a_0(t)& 0& a_1(t)&0&\dots\\
  0& a_1(t)& 0& a_2(t)&\dots\\
  &\ddots &\ddots&\ddots
\end{matrix*}}.
\end{align}
Observe that $L$ is a Jacobi operator, with only the superdiagonal, diagonal and subdiagonal non zero.
The spectrum of $L(t)$ do not depend on time. These eigenvalues  gives a set of independent integrals of motion: the Toda lattice is completely integrable.
 In particular, the Toda lattice can be solved by virtue of the inverse scattering transform for the Jacobi operator $L$. For arbitrary and sufficiently fast decaying initial conditions asymptotically for large $t$ the solution split into a sum of solitons and a decaying dispersive part.
The inverse scattering transform for this system was applied to find solutions in \cite{manakov,flaschka1}. Also in 1975 Mark Kac and Pierre van Moerbeke published two articles in PNAS regarding the Toda Lattice. In \cite{kac-van moerbeke1} a discrete version of Floquet's theory was applied to a system of non-linear differential equations related to the periodic Toda lattice and some solutions found by Toda where shown to fir in the inverse scattering formalism, but more important was \cite{kac-van moerbeke2} in where the motion of the periodic Toda lattice was
explicitly determined in terms of Abelian integrals.

\subsection{Gauss--Borel factorization in integrable systems and orthogonal polynomials}
The seminal paper of Mikio Sato \cite{sato0,sato}, and further developments performed  by the Kyoto school through the use of the bilinear equation and the $\tau$-function formalism \cite{date1,date2,date3}, settled the basis for the Lie group theoretical description of integrable hierarchies, in this direction we have the relevant contribution by Motohico Mulase \cite{mulase} in which the factorization problems, dressing procedure, and linear systems were the key for
integrability. In this dressing setting the multicomponent integrable hierarchies of Toda type were analyzed in depth by Kimio Ueno and Kanehisa Takasaki \cite{ueno-takasaki0,ueno-takasaki1,ueno-takasaki}. See also the papers \cite{BtK,BtK2} and \cite{kac} on the multi-component KP hierarchy and \cite{mma} on the multi-component Toda lattice hierarchy. In a series of papers Mark Adler and Pierre van Moerbeke showed how the Gauss--Borel factorization problem appears in the theory
of the 2D Toda hierarchy and what they called the discrete KP hierarchy \cite{adler,adler-van moerbeke, adler-vanmoerbeke 0,adler-van moerbeke 1,adler-van moerbeke 1.1,adler-van moerbeke 2}. These papers clearly established  --from a group-theoretical setup-- why standard orthogonality of polynomials and integrability of nonlinear equations of Toda type where so close. In fact, the Gauss--Borel factorization of the moment matrix may be understood as the Gauss--Borel factorization of the initial condition for the integrable hierarchy. To see the connection between the work of Mulase and that of Adler and van Moerbeke see \cite{felipe}. Later on, in the recent paper \cite{adler-vanmoerbeke 5}, it is shown that the multiple orthogonal construction described in previous paragraphs was linked with the multi-component KP hierarchy.

 In the Madrid group, based on the Gauss--Borel factorization, we have been searching further the deep links between the Theory of Orthogonal Polynomials and the Theory of Integrable Systems. In \cite{cum1} we studied the generalized orthogonal polynomials \cite{adler} and its matrix extensions from the Gauss--Borel view point. In \cite{cum2} we gave a complete study in terms of factorization for multiple orthogonal polynomials of mixed type and characterized the integrable systems associated to them. Then, we studied Laurent orthogonal polynomials in the unit circle trough the CMV approach in \cite{carlos} and find in \cite{carlos2} the Christoffel--Darboux formula for generalized orthogonal matrix polynomials. These methods where further extended, for example we gave an alternative Christoffel--Darboux formula for mixed multiple orthogonal polynomials \cite{gerardo1} or developed the corresponding  theory of matrix  Laurent orthogonal polynomials in the unit circle and its associated Toda type hierarchy \cite{MOPUC}.

\subsection{On quasi-determinants} We will like to make some comments on Schur complements and quasi-determinants. Besides its name observe that the Schur complement was  not introduced by Issai Schur but by Emilie Haynsworth in 1968 in \cite{schur1,schur2}. In fact, Haynsworth coined that named because the Schur determinant formula given in what today is known as Schur lemma in \cite{schur}.  In the book \cite{zhang} one can find an ample overview on Schur complement and many of its applications.  The most easy examples of quasi-determinants are Schur complements. Israel Gel'fand and collaborators have made many important contributions to the subject and the survey article \cite{gelfand} is an excellent reference. In addition, we also recommend Peter Olver's paper on multivariate interpolation where in \S 3 the reader will find an alternative interesting  approach to the subject.
  In the late 1920  Archibald Richardson \cite{quasideterminant1,quasideterminant2}, one of the two responsible of Littlewood--Richardson rule,  and the famous logician Arend Heyting \cite{quasideterminant3}, founder of intuitionist logic, studied possible extensions of the determinant notion to division rings. Heyting defined the \emph{designant} of a matrix with noncommutative entries, which for $2\times 2$ matrices was the Schur complement, and generalized to larger dimensions by induction. Let us stress that both Richardson's and Heyting's \emph{ quasi-determinants }
were generically rational functions of the matrix coefficients. Soon, in 1931, Oystein Ore \cite{quasideterminant0} manifested his disgust with the rational character of the just introduced quasi-determinant and gave a polynomial proposal, the Ore's determinant.
A definitive impulse to the modern theory was given by the Gel'fand's school \cite{quasidetermiant6,quasideterminant4,quasideterminant5,quasidetermiant7,quasidetermiant8,quasidetermiant9}.
Quasi-determinants where defined over free division rings and was early noticed that is not an analog of the commutative determinant but rather of a ratio determinants. A cornerstone for  quasi-determinants is  the  \emph{heredity principle}, quasi-determinants of quasi-determinants are quasi-determinants; there is no analog of such a principle for determinants.
However, many of the properties of determinants extend to this case, see the cited papers and also \cite{minor} for quasi-minors expansions. Let us mention that in the early 1990 the Gelf'and school \cite{quasidetermiant7} already noticed the role quasi-determinants for some integrable systems, see also \cite{rekhtah} for some recent work in this direction regarding non-Abelian Toda and Painlev\'{e} II equations. Jon Nimmo and his collaborators,  the Glasgow school, have  studied the relation of quasi-determinants and integrable systems, in particular we can mention the papers \cite{nimmo0,nimmo1,nimmo2,gilson,nimmo3};  in this direction see also \cite{qd-otros,qd-otros2,qd-otros3}. All this paved the route, using the connection with orthogonal polynomials \emph{\`{a} la Cholesky}, to the appearance of quasi-determinants in the multivariate orthogonality context. Later, in 2006 Peter Olver applied quasi-determinants to multivariate interpolation \cite{olver}. This is the approach we apply in this paper. As in \cite{olver} the blocks have different sizes, and so multiplication of blocks is only allowed if they are \emph{compatible}. In general, the (non-commutative) multiplication makes sense if the number of columns and rows of the blocks involved fit well. Moreover, we are only permitted to invert diagonal entries that in general makes the minors expansions by columns or rows not applicable \cite{minor} but allows for other result, like the Sylvester's theorem,  to hold in this wider scenario.

\subsection{Aims, results and structure of the paper}
The question was possed to us by Jeff Geronimo: What about the integrable systems associated with multivariate orthogonal polynomials? To answer  this question we  consulted \cite{Dunkl} and we readily  noticed  the ubiquity of the Gauss--Borel factorization in the subject, and therefore the opportunity to link it with the theory of integrable systems. Once this fact was realize we applied the factorization technology of the moment matrix to reproduce the general theory presented in \cite{Dunkl}.

The main difference with the case of orthogonal polynomials in the real line (OPRL) is that now the moment matrix is a block matrix, with its elements being rectangular matrices of varying size. We had come across with matrix blocks before when we studied matrix orthogonal polynomials, but there the size of each block was fixed, now is variable.  This intrinsic fact, leagued with the multivariate character, lead in the one hand to the appearance of Schur complements and  quasi-determinants and, on the other hand, to multivariate Cauchy integrals and integrals along the Shilov border of poly-disks --that is, to be faced to some basic facts of complex analysis in several variables.
The Schur complement already appeared in the study of matrix orthogonal polynomials, see for example \cite{MOPUC,giovanni1,giovanni2}, but we did not understood yet in \cite{MOPUC} the important role played int the theory by quasi-determinants; now we do.  We adjacently  get across symmetric algebra \cite{federer,knapp}, being  isomorphic to the set of multivariate polynomials it some times allow for simple derivation of some result or illuminate some structure.
All the necessary material regarding these issues can be found in the Appendices.

\subsubsection{Results} In the first place we recover a number of classical results from the multivariate orthogonality general theory, see for example \cite{Dunkl},  using  a Cholesky factorization\footnote{A Gauss--Borel factorization for the symmetric case.} of a symmetric moment matrix. We  got the multivariate orthogonal polynomials associated with a given Borel measure and the corresponding second kind functions, that happen to be multivariate  Cauchy transforms of the polynomials. All these objects have \emph{quasi-determinantal expressions in terms of bordered truncated moment matrices}. Then, the shift matrices allow to get the three term relation and also Jacobi type matrices and Christoffel--Darboux formul{\ae}.

Once we have been able to reproduce, with a Cholesky flavour, classical results for multivariate orthogonal polynomials,  we begun the quest of discrete and continuous deformations of the measure which lead to equations of the Toda type. We found both partial difference and partial differential nonlinear equations  for the varying size  matrices. Moreover, we introduce  \emph{quasi-tau matrices} and find the analogous, in this multi-variable scenario, to the 1D expressions of the orthogonal polynomials and its second kind functions as ratios of Miwa shifted tau functions. Besides these achievements we noticed that the discrete flows allow for the finding of the multivariate extension of the elementary Darboux transformations via what we named as the \emph{sample matrix trick}. These allow not only to express the kernel polynomials, but also its second kind functions, the quasi-tau matrices and some other important coefficients as quasi-determinants of the original data.  The sample matrix trick allows also for the study of iterated Darboux transformations and the finding of the \emph{multivariate version of the Christoffel formula}. Many relevant elements of Toda integrable theory, as linear systems, Lax equations, Zakharov--Shabat equations and bilinear equations, are found.  An \emph{asymptotic module} or \emph{asymptotic congruence} arguments  permit for another perspective of the hierarchy, and we find KP type equations for this multivariate case. Finally,  a \emph{linear isometry invariance} of the measure is assumed and  we get, through the Cholesky factorization, the consequences for the multivariate orthogonality and the corresponding integrable systems.

\subsubsection{The layout of the paper}   After this introduction we discuss in \S \ref{2} the general theory of multivariate orthogonal polynomials by using the Cholesky factorization of a moment matrix. We describe the monomials and order them, according to the reserve lexicographic order, so that  we can analyze the conditions for the Cholesky factorization to hold and find the multivariate orthogonal polynomials and its associated second kind functions and its integral representation. The shift matrices  are introduced and the three term relations are recovered. The Christoffel--Darboux  formul{\ae} is deduced in this context.

In \S \ref{3} we introduce discrete Toda deformations of the measure, we find the corresponding integrable discrete flows, wave matrices, lattice resolvents and  Lax (or Jacobi type matrices) pairs are given; a quasi-determinantal expression in terms of the Jacobi matrix for the lattice resolvent is found. Discrete Lax and Zakharov--Shabat equations and corresponding discrete Toda type equations for the varying size quasi-tau functions are described. Then, we find some interesting expressions for the multivariate orthogonal polynomials and its second kind functions in terms of quasi-tau functions and its shifts. For the orthogonal polynomials we need to use the Moore--Penrose pseudo-inverse of a matrix given in terms of the shift matrices and for the second kind functions we need to use a composed, or total, translation. In the 1D scenario these formul{\ae} are the well known expressions for these objects in terms of quotients of tau functions and its Miwa shifts (which happen to be discrete flows). We observe that these discrete transformations are elementary Darboux (or Christoffel) transformations and we are able, introducing the sample matrix trick, to give an explicit expression for the transformed polynomials in terms quasi-determinants of the original ones. The $n$-th iteration of these multivariate elementary Darboux lead to a multivariate Christoffel formul{\ae} expressing the new orthogonal polynomials  $\tilde P_{[k]}(\x)$ in terms of quasi-determinants of the original ones $P_{[k]}(\x),\dots,P_{[k+n]}(\x)$ evaluated at some appropriate nodes. This approach leads to the finding of quasi-determinantal expressions for the kernel polynomials in terms of the evaluation of the Christoffel--Darboux kernels.

Continuous Toda deformations of the measure are discussed in \S \ref{4}. We introduce
Baker and adjoint Baker functions in terms of multivariate orthogonal polynomials and its multivariate Cauchy transforms, we find the corresponding  Lax and Zakharov-Shabat equations and write a continuous Toda type equations for the quasi-tau matrices. The discrete flows are identified with Miwa shifts and the bilinear equations, with integrals along tori --Shilov borders of appropriate  polydisks-- are given. Next, in \S \ref{5} we apply the congruence technique  to find KP type equations, nonlinear equations that relate through nonlinear partial differential-difference equations coefficients of the polynomials but for the same $k$, not involving, as it do happen in the Toda scenario, near neighbours $k+1$ and $k-1$. We connect using this method  discrete and continuous flows. Then, we present linear equations and corresponding nonlinear partial differential equations for the second order flows. We end the section by exploring the linear equations for the third order flows and giving some hints for higher order flows. Finally, in \S \ref{6} we study some linear isometry type symmetries of the measure and its consequences on the multivariate orthogonal polynomials; we discuss also what discrete or continuous flows preserve this symmetry.

In the Appendices we present some necessary material for  reading of the paper. In particular,   compositions, multisets and symmetric algebras  are briefly  treated in  Appendix \ref{symmetric}. Then, in Appendix \ref{qd} we recall some aspects of pseudo-inverses,  Schur complements and quasi-determinants and, in Appendix \ref{scv}, we give some notations and results that appear in the  analysis in several complex variables. For the sake of clarity some of the proofs of Propositions and Theorems have been collected in Appendix \ref{proofss}.

\section{Multivariate orthogonality \`{a} la Cholesky}\label{2}
We study  multivariate  orthogonal polynomials  in a  $D$-dimensional real space (MVOPR) in terms of a Cholesky factorization of a semi-infinite moment matrix. We consider $D$ independent real variables $\x=\left(x_1,x_2,\dots,x_D \right)^\top\in \Omega\subseteq\mathbb{R}^D$ varying in the domain $\Omega$ together with a Borel measure $\dd\mu(x) \in \mathcal{B}(\Omega)$.
The inner product  of two real valued functions $f(\x)$ and $g(\x)$ is defined by
\begin{align*}
 \langle f,g\rangle&\coloneq\int_{\Omega} f(\x)\dd\mu(\x) g(\x).
\end{align*}

\subsection{Ordering the monomials}
Given a multi-index $\q=(\alpha_1,\dots,\alpha_D)^\top \in\Z_+^{D}$ of non-negative integers we write $\x^{\q}=x_1^{\alpha_1}\cdots x_D^{\alpha_D}$; the length\footnote{Also known as absolute value, order or norm. } of $\q$ is $|\q|\coloneq  \sum_{a=1}^{D} \alpha_a$. This length induces the total ordering of monomials, $\x^{\q}<\x^{\q'}\Leftrightarrow|\q|<|\q'|$, that we will use to arrange the monomials. For each non-negative integer $k\in\Z_+$ we introduce the set
\begin{align*}
[k]\coloneq \{\q\in \Z_+^{D}: |\q|=k\},
\end{align*}
 built up with those vectors  in the lattice $\Z_+^D$ with a given length $k$.

We will use the graded reversed lexicographic order; i.e.,  for $\q_1,\q_2\in [k]$
\begin{align*}
 \q_1>\q_2 \Leftrightarrow \exists p\in \Z_+ \text{ with } p<D \text{ such that } \alpha_{1,1}=\alpha_{2,1},\dots,\alpha_{1,p}=\alpha_{2,p} \text{ and } \alpha_{1,p+1}<\alpha_{2,p+1},
\end{align*}
and if $\q^{(k)}\in[k]$ and $\q^{(\ell)}\in[\ell]$, with $k<\ell$ then $\q^{(k)}<\q^{(\ell)}$.
Given the set of integer vectors of length $k$ we use the reversed lexicographic order and write
\begin{align*}
  [k]=\big\{\q_1^{(k)},\q_2^{(k)},\dots,\q^{(k)}_{|[k]|}\big\} \text{ with } \q_a^{(k)}<\q_{a+1}^{(k)}.
\end{align*}
Here $|[k]|$ is the cardinality of the set $[k]$, i.e., the number of elements in the set.
Observe that $|[0]|=1$,  $|[1]|=D$ and $|[2]|=\frac{(D+1)D}{2}$.

  We introduce the  vectors of monomials
 \begin{align*}
  \chi&\coloneq \PARENS{\begin{matrix}\chi_{[0]} \\ \chi_{[1]} \\ \vdots \\ \chi_{[k]} \\ \vdots \end{matrix}}
  & \mbox{where} & &
  \chi_{[k]}&\coloneq  \PARENS{\begin{matrix} \x^{\q_1} \\  \x^{\q_2} \\\vdots \\ \x^{\q_{|[k]|}} \end{matrix}},\\
  \chi^*&\coloneq  \Big(\prod_{a=1}^D x_a^{-1}\Big)\chi(x_1^{-1},\dots,x_D^{-1});
  \end{align*}
 for example $   \chi_{[0]}=1$ , $\chi_{[0]}^*=\prod_{a=1}^D x_a^{-1}$ and
 \begin{align*}
 \chi_{[1]}&=\PARENS{\begin{matrix}
     x_1\\x_2\\\vdots \\x_D
   \end{matrix}}, & \chi^*_{[1]}&=\Big(\prod_{a=1}^D x_a^{-1}\Big) \PARENS{\begin{matrix}
     x_1^{-1}\\x_2^{-1}\\\vdots \\x_D^{-1}
   \end{matrix}}&
   \chi_{[2]}&=\PARENS{\begin{matrix}
     x_1^2\\x_1x_2\\\vdots\\x_1x_D\\x_2^2\\x_2x_3\\\vdots\\x_2x_D\\x_3^2\\\vdots\\x_D^2
   \end{matrix}}, &
    \chi_{[2]}^*&=\Big(\prod_{a=1}^D x_a^{-1}\Big) \PARENS{\begin{matrix}
     x_1^{-2}\\x_1^{-1}x_2^{-1}\\\vdots\\x_1^{-1}x_D^{-1}\\x_2^{-2}\\x_2^{-1}x_3^{-1}\\\vdots\\x_2^{-1}x_D^{-1}\\x_3^{-2}\\\vdots\\x_D^{-2}
   \end{matrix}}.
 \end{align*}
 Observe that for $k=1$ we have that the vectors $\q^{(1)}_a=\ee_a$ for $a\in\{1,\dots,D\}$ forms  the canonical basis of $\R^D$, and for any $\q_j\in[k]$ we have $\q_j=\sum_{a=1}^D \alpha_{j}^a\ee_a$ .
 For the sake of simplicity unless needed we will drop off the super-index and write $\q_j$ instead of $\q^{(k)}_j$, as is understood that $|\q_j|=k$.
 Notice that
 \begin{align}\label{eq:chi-chi*}
   \chi^\top(\y)\chi^*(\x)&=\prod_{a=1}^D\frac{1}{x_a-y_a},&
   \forall \x,\y\in\C^D \text{ such that }|x_a|>|y_a|.
 \end{align}
 \subsection{Monomials and  symmetric tensor powers}

  The dual space of the symmetric tensor powers, see Appendix \ref{symmetric}, happens to be isomorphic to the set of symmetric multilinear functionals on $\R^D$, $\big(\text{Sym}^k(\R^D)\big)^*\cong S((\R^D)^k,\R)$. Hence,
homogeneous polynomials of a given total degree  can be identified with symmetric tensor powers.
 Each multi-index $\q\in[k]$\footnote{Observe that in \cite{federer} we have diverse notation $[k]\equiv \Xi(D,k)$.} can be thought as a weak $D$-composition of $k$ (or weak composition in  $D$ parts), $k=\alpha_{1}+\dots+\alpha_{D}$.
Notice that these weak compositions may be considered as multisets and that, given a linear basis $\{\ee_a\}_{a=1}^D$ of $\R^D$,  as we know from Appendix \ref{symmetric}, we have the linear basis $\{\ee_{a_1}\odot\cdots\odot \ee_{a_k}\}_{\substack{1\leq a_1\leq\cdots\leq a_k\leq D\\ k\in\Z_+}}$ for the symmetric power $\operatorname{S}^k(\R^D)$, where we are using multisets $1\leq a_1\leq\cdots\leq a_k\leq D$. In particular, see Appendix \ref{spowers}, the vectors of this basis $\ee_{a_1}^{\odot M(a_1)}\odot\cdots\odot \ee_{a_p}^{\odot M(a_p)}$, or better its duals $(\ee_{a_1}^*)^{\odot M(a_1)}\odot\cdots\odot (\ee_{a_p}^*)^{\odot M(a_p)}$ are in bijection with monomials of the form $x_{a_1}^{M(a_1)}\cdots x_{a_p} ^{M(a_p)}$. Therefore, either counting weak compositions or multisets we are lead to the following conclusion:  the cardinality of $[k]$ is $|[k]|= \big(\!{D\choose k}\!\big) = {D+k-1 \choose k} $.

The monomials can be nicely expressed in terms of  symmetric products and the multinomial matrix, see Appendix \ref{symmetric}. The reverse lexicographic order can be applied  to $\big(\R^D\big)^{\odot k}\cong \R^{|[k]|}$, we then  take a linear basis of $\operatorname{S}^k(\R^D)$ as the ordered set $B_c=\{\ee^{\q_1},\dots,\ee^{\q_{|[k]|}}\}$ with $\ee^{\q_j}\coloneq \ee_1^{\odot \alpha_{j}^1}\odot\dots\odot \ee_{D}^{\odot \alpha_{j}^D}$ so that
$\chi_{[k]}(\x)=\sum_{i=1}^{|[k]|}\x^{\q_j}\ee^{\q_j}$. This means that in this canonical basis the column matrix representing $\chi_{[k]}$ is
$[\chi_{[k]}]_{B_c}=\begin{psmallmatrix}
  \x^{\q_1}\\\vdots\\\x^{\q_{|[k]|}}
\end{psmallmatrix}$. We will identify $\chi_{[k]}$ with $[\chi_{[k]}]_{B_c}$.
\begin{pro}\label{chi-symmetric}
  If  $[\x^{\odot k}]_{B_c}$ is the column matrix representing $\x^{\odot k}$ in the canonical basis $B_c$ we have
  \begin{align}\label{chi-symmetric power}
  \chi_{[k]}(\x)=\big(\mathcal M_{[k]}\big)^{-1}[\x^{\odot k}]_{B_c}.
  \end{align}
\end{pro}
\begin{proof}
It is a consequence of the multinomial theorem for symmetric powers
\begin{align*}
  \x^{\odot k}&=(x_1\ee_1+\dots+x_D\ee_D)^{\odot k}\\
  =&\sum_{j=1}^{|[k]|}  {k\choose \q_j}\x^{\q_j}\ee_{1}^{\odot \alpha_{j}^1}\odot\dots\odot \ee_{D}^{\odot \alpha_{j}^D}\\
  =&\mathcal M_{[k]}\chi_{[k]}(\x).
\end{align*}
\end{proof}

\subsection{Cholesky factorization of the moment matrix}
In this paper we will consider semi-infinite matrices $A$ with a block or partitioned structure induced by the graded reversed lexicographic order
\begin{align*}
A&=\PARENS{\begin{matrix}
   A_{[0],[0]} & A_{[0],[1]} &  \cdots  \\
   A_{[1],[0]} & A_{[1],[1]} &  \cdots \\
   \vdots                &                 \vdots         &  \\
  \end{matrix}}, &
A_{[k],[\ell]}&=\PARENS{\begin{matrix}
  A_{\q^{(k)}_1,\q^{(\ell)}_1} &   \dots & A_{\q^{(k)}_1,\q^{(\ell)}_{|[\ell]|} }\\
  \vdots & & \vdots\\
  A_{\q^{(k)}_{|[k]|},\q^{(\ell)}_1} &  \dots & A_{\q^{(k)}_{|[k]|},\q^{(\ell)}_{|[\ell]|} }
  \end{matrix}} \in\R^{|[k]|\times |[\ell]|}.
\end{align*}
We use the notation $0_{[k],[\ell]}\in\R^{|[k]|\times|[\ell]|}$ for the rectangular zero matrix, $0_{[k]}\in\R^{|[k]|}$ for the zero vector, and $\I_{[k]}\in\R^{|[k]|\times|[k]|}$ for the identity matrix. For the sake of simplicity we normally  just write $0$ or $\I$ for the zero or identity matrices, and we implicitly assume that the sizes of these matrices are the ones indicated by its position in the partitioned matrix.
\begin{definition}\label{moment}
Associated with the measure $\dd\mu$ we have the following moment matrix
 \begin{align*}
  G&\coloneq \int_{\Omega} \chi(\x)\dd\mu(\x) \chi(\x)^\top.
 \end{align*}
We write the moment matrix  in block form
  \begin{align*}
  G=  \PARENS{\begin{matrix}
   G_{[0],[0]} & G_{[0],[1]} &  \dots \\
   G_{[1],[0]} & G_{[1],[1]} &  \dots \\
   \vdots                &   \vdots              &
  \end{matrix}}
   \end{align*}
with each entry being a rectangular  matrix  with real coefficients
 \begin{align}
 \label{eq:Gkl}
  G_{[k],[\ell]}\coloneq &\int_{\Omega} \chi_{[k]}(\x)\dd\mu(\x) (\chi_{[\ell]}(\x))^\top,
  &k,\ell&=0,1,\dots, \\
 =&\PARENS{\begin{matrix}
  G_{\q^{(k)}_1,\q^{(\ell)}_1} &   \dots & G_{\q^{(k)}_1,\q^{(\ell)}_{|[\ell]|} }\\
  \vdots & & \vdots\\
  G_{\q^{(k)}_{|[k]|},\q^{(\ell)}_1} &  \dots & G_{\q^{(k)}_{|[k]|},\q^{(\ell)}_{|[\ell]|} }
  \end{matrix}} \in  \R^{|[k]|\times |[\ell]|}, &
    G_{\q_i^{(k)},\q^{(\ell)}_j}&\coloneq \int_{\Omega} \x^{\q^{(k)}_i+\q^{(\ell)}_j}\dd\mu(\x)
     \in \mathbb{R}. 
     \notag
 \end{align}
Truncated  moment matrices are given by
 \begin{align*}
  G^{[\ell]}&\coloneq
  \PARENS{\begin{matrix}
   G_{[0],[0]} &  \cdots & G_{[0],[\ell-1]} \\
   \vdots                        &   & \vdots \\
   G_{[\ell-1],[0]}  &  \cdots & G_{[\ell-1],[\ell-1]}
  \end{matrix}},
 \end{align*}
 and for $k\geq \ell$ we will also use the following bordered truncated moment matrix
\begin{align*}
   G^{[\ell+1]}_k&\coloneq
  \PARENS{\begin{array}{ccc}
   G_{[0],[0]} &  \cdots & G_{[0],[\ell-1]} \\
   \vdots                        &   & \vdots \\
   G_{[\ell-2],[0]}  &  \cdots & G_{[\ell-2],[\ell-1]}\\[1pt]
   \hline
   G_{[k],[0]}& \dots & G_{[k],[\ell-1]}
  \end{array}}
 \end{align*}
 in where we have replaced the last row of blocks, $\PARENS{\begin{matrix}
  G_{[\ell-1],[0]}& \dots & G_{[\ell-1],[\ell-1]} \end{matrix}}$, of the truncated moment matrix $G^{[\ell+1]}$ by the row of blocks $\PARENS{\begin{matrix}
  G_{[k],[0]}& \dots & G_{[k],[\ell-1]} \end{matrix}}$.
\end{definition}
Notice that from the above definition we know that the moment matrix is a symmetric matrix, $G=G^\top$,
which implies that a Gauss--Borel factorization of it, in terms of unitriangular lower\footnote{Lower triangular with the block diagonal populated by identity matrices.} and upper triangular matrices, is a Cholesky factorization.
We describe now when and how the Cholesky factorization of the moment can be performed. The result and its proof uses Schur complements, see Appendix \ref{schur}.
\begin{pro}\label{pro:cholesky}\begin{enumerate}
  \item If $\det G^{[\ell]}\neq 0$ for all $\ell=0,1,\dots$ then $G$ admits the following Cholesky type factorization
 \begin{align}\label{cholesky}
  G&=S^{-1} H \left(S^{-1}\right)^{\top},
  \end{align}
  with
\begin{align*}
  S^{-1}&=\PARENS{\begin{matrix}
  \I    &             0                &  0                      &  \cdots            \\
  (S^{-1})_{[1],[0]}        & \I&     0                      &   \cdots         \\
  (S^{-1})_{[2],[0]}        & (S^{-1})_{[2],[1]} & \I&      \\
           \vdots                       &        \vdots                  &                               &\ddots
  \end{matrix}}, \\
  H&=\PARENS{\begin{matrix}
H_{[0]}           &   0         &     0         \\
0                 & H_{[1]} &   0             &    \cdots       \\
0                  &    0            & H_{[2]} &                        \\
\vdots   &   \vdots &              &     \ddots       \\
  \end{matrix}},
  \end{align*}

\item When $\det G^{[\ell]}\neq 0$ the Cholesky type factorization holds and
\begin{align*}
\det G^{[\ell]}&=\prod_{k=0}^{\ell-1} \det H_{[k]}\neq 0
\end{align*}
so that all $H_{[k]}$ are invertible, $k=0,1,\dots$.
\end{enumerate}
\end{pro}
\begin{proof}
See Appendix \ref{proof1}
\end{proof}

A quasi-determinant version, see Appendix \ref{qd}, of the above  result can be given
\begin{pro}\label{qd1}
If the quasi-determinants of the truncated moment matrices are invertible
\begin{align*}
\det  \Theta_*(G^{[k+1]})\neq& 0, & k=0,1,\dots
\end{align*}
the Cholesky factorization \eqref{cholesky} can be performed where
\begin{align*}
  H_{[k]}&=\Theta_*(G^{[k+1]}), &
  (S^{-1})_{[k],[\ell]}&=\Theta_*(G^{[\ell+1]}_k)\Theta_*(G^{[\ell+1]})^{-1}.
\end{align*}
\end{pro}
\begin{proof}
  It is just a consequence of Theorem 3 of \cite{olver}, see Appendix \ref{qd}.
\end{proof}
\subsubsection{On quasi-tau functions}
In the 1D scenario the tau functions can be introduced as the determinant of a truncated moment matrix
\begin{align*}
   \tau_k\coloneq &\det G^{[k]}, & k=&1,2,\dots
\end{align*}
and $\tau_0=1$, so that
\begin{align}\label{quasitau1d}
  H_k=&\frac{\tau_{k+1}}{\tau_k}, & k=&0,1,2,\dots.
\end{align}
and
\begin{align*}
  \tau_{k+1}=H_kH_{k-1}\cdots H_0.
\end{align*}
Moreover, observe that \eqref{quasitau1d} can be written as a quasi-determinant
\begin{align*}
  H_{k}=\det (G^{[k+1]}/G^{[k]})=\Theta_*(G^{[k+1]}).
\end{align*}
Thus, in the 1D scenario the described analogy  suggests that the squared norms $H_k$ can be considered as quasi-tau functions, being the tau functions $\tau_k=\det G^{[k]}$ determinants of the truncated moment matrix and
the quasi-tau functions $H_k=\Theta_*(G^{[k+1]})$ quasi-determinants of the  truncated moment matrix. This extends to the multivariant setting and now we have $H_{[k]}=\Theta_*(G^{[k+1]})$,  motivating us to refer to these matrices as quasi-tau matrices. Let us mention that other authors have introduced similar concepts before, for example in \cite{miranian} a matrix valued tau function is considered for the case of matrix orthogonal polynomials. However, the motivation of the author did not come from the quasi-determinant expressions in terms of the moment matrix but from formul{\ae} from integrable systems.
\subsection{MVOPR}
With the aid of the Cholesky factorization we are ready to introduce the MVOPR
\begin{definition}
 The MVOPR associated to the measure $\dd \mu$  are
 \begin{align}\label{eq:polynomials}
  P&=S\chi =\PARENS{\begin{matrix}
    P_{[0]}\\
    P_{[1]}\\
    \vdots
  \end{matrix}}, & P_{[k]}(\x)&=\sum_{\ell=0}^k S_{[k],[\ell]} \chi_{[\ell]}(\x) =\PARENS{\begin{matrix}
    P_{\q^{(k)}_1}\\
    \vdots\\
    P_{\q^{(k)}_{|[k]|}}
  \end{matrix}},&
  P_{\q^{(k)}_i}&=\sum_{\ell=0}^k\sum_{j=1}^{|[\ell]|} S_{\q^{(k)}_i,\q^{(\ell)}_j} \x^{\q^{(\ell)}_j}.
 \end{align}
We introduce  the coefficients
\begin{align*}
  \beta_{[k]}&\coloneq S_{[k],[k-1]}, & k\geq 1,
\end{align*}
which take values in the linear space of rectangular matrices $\R^{|[k]|\times|[k-1]|}$ and also define the subdiagonal matrix
\begin{align*}
  \beta=\PARENS{\begin{matrix}
    0 & 0 & 0 &0&\cdots\\
    \beta_{[1]}& 0 & 0&0&\cdots\\
    0&\beta_{[2]}&0&0&\cdots\\
    0&0&\beta_{[3]}&0&\cdots\\
    \vdots&\vdots&\ddots&\ddots&\ddots
  \end{matrix}}.
\end{align*}\end{definition}

An immediate consequence of Proposition \ref{qd1} is
\begin{pro}
  The following quasi-determinantal expression holds true
  \begin{align*}
    \beta_{[k]}=-\Theta_*(G^{[k]}_k)\Theta_*(G^{[k]})^{-1}.
  \end{align*}
\end{pro}

Observe that $P_{[k]}=\chi_{[k]}(\x)+\beta_{[k]}\chi_{[k-1]}(\x)+\cdots$ is a vector constructed with the polynomials $P_{\q_i}(\x)$ of degree  $k$, each of which has only one monomial of degree $k$; i. e., we can write $P_{\q_i}(\x)=\x^{\q_i}+Q_{\q_i}(\x)$, with $\deg Q_{\q_i}<k$.

\begin{pro}
 The MVOPR  satisfy
 \begin{align}
  \int_{\Omega} P_{[k]}(\x) \dd \mu(\x) (P_{[\ell]}(\x))^{\top}&=\int_{\Omega} P_{[k]} (\x)\dd \mu(\x) (\chi_{[\ell]}(\x))^{\top}=0, &
  \ell&=0,1,\dots,k-1,\label{orth}\\
  \int_{\Omega} P_{[k]} (\x)\dd \mu(\x)(P_{[k]}(\x))^{\top}&=\int_{\Omega} P_{[k]}(\x) \dd\mu(\x)(\chi_{[k]}(\x))^{\top}=H_{[k]}.\label{H}
 \end{align}
\end{pro}
Therefore, we have  the following orthogonality conditions
 \begin{align*}
\int_{\Omega} P_{\q^{(k)}_i} (\x)P_{\q^{(\ell)}_j}(\x)\dd \mu(\x)&=\int_{\Omega} P_{\q^{(k)}_i}  (\x)\x^{\q^{(\ell)}_j}\dd \mu(\x)=0. &
  \ell&=0,1,\dots,k-1,& i&=1,\dots,|[k]|,& j&=1,\dots,|[\ell]|,
  \end{align*}
  with the normalization conditions
  \begin{align*}
  \int_{\Omega} P_{\q_i} (\x)P_{\q_j}(\x)\dd \mu(\x)&=\int_{\Omega} P_{\q_i}  (\x)\x^{\q_j}\dd \mu(\x)=H_{\q_i,\q_j}, &  i,j&=1,\dots,|[k]|.
   \end{align*}
      Here we use Dunkl and Xu's notation, see \cite{Dunkl}. Despite the MVOPR are orthogonal for different $k$ and $\ell$, $P_{\q^{(k)}_i}\bot P_{\q^{(\ell)}_j}$, for $k=\ell$ the value of $\langle P_{\q_i},P_{\q_j}\rangle=H_{\q_i,\q_j}$ is not zero in general, and the set of polynomials given by coefficients of the vector $P_{[k]}$ are not orthogonal among them. Observe that \eqref{H} imply that the matrices $H_{[k]}$ are Grammian matrices and that, being the measure  positive definite,  we can write $H_{[k]}=M_{[k]}^\top h_{[k]} M_{[k]}$, for some orthogonal matrix $M_{[k]}\in \text{O}(\R^{|[k]|})$ and  diagonal matrix $h_{[k]}=\diag(h_{[k],1},\dots,h_{[k],|[k]|})$ with $h_{[k],j}>0$ for $j\in\{1,\dots,|[k]|\}$. With the new vector polynomials $\tilde P_{[k]}=M_{[k]}P_{[k]}$ we do have
  \begin{align*}
  \int_{\Omega} \tilde P_{\q_i} (\x)\tilde P_{\q_j}(\x)\dd \mu(\x)=&\delta_{i,j}h_{[k],j}, &  i,j&=1,\dots,|[k]|,
   \end{align*}
   being the $h_{[k],j}$ the squared norms of the polynomials. Now, instead of a block Cholesky factorization  we have a standard Cholesky factorization $G=\tilde S^{-1} h (\tilde S^{-1})^\top$, with
   $\tilde S=\diag( M_{[0]}^\top,M_{[1]}^\top,\dots) S$ and $h=\diag(h_{[0]},h_{[1]},\dots)$.  However, this scalar Cholesky factorization does not help much in the understanding of MOVPR, the reason will become clear in \S \ref{los Lambda}, where the three term relations or the Christoffel--Darboux formul{\ae} are deduced from the block Cholesky factorization. The clue is that the shift matrices, for which we have the symmetry \eqref{eq:symmetry} of the moment matrix,  are naturally written in block form.

   Also notice that $H_{[0]}=\int_\Omega\dd\mu(\x)$ is just the measure of the support.

\begin{pro}
 The MVOPR can be expressed as Schur complements of bordered truncated moment matrices
 \begin{align*}
 P_{[\ell]}(\x)=\SC\PARENS{\begin{array}{ccc|c}
       G_{[0],[0]}  & \cdots & G_{[0],[\ell-1]} & \chi_{[0]}(\x)\\
           \vdots      &       &            \vdots   &  \vdots\\
       G_{[\ell-1],[0]} & \cdots & G_{[\ell-1],[\ell-1]} & \chi_{[\ell-1]}(\x)\\[1pt]
\hline
             G_{[\ell],[0]} & \cdots & G_{[\ell],[\ell-1]} & \chi_{[\ell]}(\x)
      \end{array}},
 \end{align*}
 or as quasi-determinants
  \begin{align*}
 P_{[\ell]}=\Theta_*\PARENS{\begin{matrix}
       G_{[0],[0]}  & \cdots & G_{[0],[\ell-1]} & \chi_{[0]}(\x)\\
           \vdots      &       &            \vdots   &  \vdots\\
                   G_{[\ell],[0]} & \cdots & G_{[\ell],[\ell-1]} & \chi_{[\ell]}(\x)
      \end{matrix}}.
 \end{align*}
\end{pro}
\begin{proof}
Any semi-infinite matrix can be written in block form $M=\begin{psmallmatrix}
  M^{[\ell]}& M^{[\ell],[\geq ell]}\\
  M^{[\geq \ell] [,\ell]}& M^{[\geq \ell]}
\end{psmallmatrix}$, where $M^{[\ell]}$ is the truncation with the first $\ell$ block rows and block columns, $M^{[\ell],[\geq \ell]}$ is the truncation with the $\ell$ first rows and the columns after the $\ell$-th, $M^{[\geq \ell] ,[\ell]}$ the truncation as the previous permuting the role of rows and columns $M^{[\geq \ell]}$ the truncation formed by all rows and columns after the $\ell$-th one.
From the factorization of the moment matrix
 \begin{align*}
  SG&=H(S^{-1})^{\top} & \Longrightarrow& & 0&=S^{[\geq \ell], [\ell]}G^{[\ell]}+S^{[\geq \ell]} G^{[\geq \ell], [\ell]}
  & \Longrightarrow & &  S^{[\geq \ell], [\ell]}&=-S^{[\geq \ell]} G^{[\geq \ell], [\ell]}\big(G^{[\ell]}\big)^{-1}.
 \end{align*}
With this we can rewrite $P_{[\ell]}=\sum_{k=0}^\ell S_{[\ell],[k]} \chi_{[k]}$ as
\begin{align*}
P_{[\ell]}&=\chi_{[\ell]}-
\PARENS{\begin{matrix}G_{[\ell],[0]} & G_{[\ell],[1]} & \cdots & G_{[\ell],[\ell-1]}\end{matrix}}
\left(G^{[\ell]} \right)^{-1} \chi^{[\ell]}.
\end{align*}
To get the stated result we need only to fix our attention in any of the rows of this matrix.
\end{proof}
In terms of ratios of determinant we get for the components
 \begin{align*}
 P_{\q^{(\ell)}_j}=\SC\PARENS{\setlength{\extrarowheight}{3.5pt}\begin{array}{ccc|c}
       G_{[0],[0]}  & \cdots & G_{[0],[\ell-1]} & \chi_{[0]}\\
       \vdots            & &  \vdots             &  \vdots\\
       G_{[\ell-1],[0]}  & \cdots & G_{[\ell-1],[\ell-1]} & \chi_{[\ell-1]} \\[1pt]
\hline
 G_{\q^{(\ell)}_j,[0]}  & \cdots & G_{\q^{(\ell)}_j,[\ell-1]} & \x^{\q^{(\ell)}_j}  \bigstrut[t]
      \end{array}}=
      \frac{\begin{vmatrix}
                   G_{[0],[0]}  & \cdots & G_{[0],[\ell-1]} & \chi_{[0]}\\
       \vdots            & &  \vdots             &  \vdots\\
       G_{[\ell-1],[0]}  & \cdots & G_{[\ell-1],[\ell-1]} & \chi_{[\ell-1]}\\
    G_{\q^{(\ell)}_j,[0]}  & \cdots & G_{\q^{(\ell)}_j,[\ell-1]} & \x^{\q^{(\ell)}_j}
                 \end{vmatrix}
}{ \begin{vmatrix}
 G_{[0],[0]}        & \cdots & G_{[0],[\ell-1]}\\
\vdots            &          &  \vdots     \\
  G_{[\ell-1],[0]}  & \cdots & G_{[\ell-1],[\ell-1]}
                 \end{vmatrix}}.
\end{align*}
\subsection{Functions of the second kind}
In this subsection we need some material regarding several complex variables analysis and we refer the reader to Appendix \ref{scv}.
Complementary to the vector $P$ of multivariate polynomials we introduce
\begin{definition}
Second kind functions are given by the coefficients of
  \begin{align}\label{def:C}
  C\coloneq &H(S^{-1})^\top\chi^*=\PARENS{\begin{matrix}
    C_{[0]}\\
    C_{[1]}\\
\vdots
  \end{matrix}}, & C_{[k]}&\coloneq \PARENS{\begin{matrix}
    C_{\q_1}\\
    \vdots\\
    C_{\q_{|[k]|}}
  \end{matrix}}. &
\end{align}
\end{definition}
Observe that for $\z=(z_1,\dots,z_D)^\top\in\C^D$ we have
\begin{align*}
 C_{[\ell]}(\z)=H_{[\ell]}\sum_{k=\ell}^\infty \big((S^{-1})_{[k],[\ell]}\big)^\top\chi^*_{[k]}(\z),
\end{align*}
which is a vector with each of its  components $C_{\ele_a}(\z)$, $a=1,\dots, |[\ell]|$, a $D$-fold Laurent series. This is just not the case for the definition of $P$, see \eqref{eq:polynomials}, where we had  finite sums instead of infinite series.
In the case of $C_{\q_i}(\z)$, which has  domain of convergence $\Ds_{\q_i}$, we can introduce $\boldsymbol w=\z^{-1}\coloneq (z_1^{-1},\dots,z_D^{-1})^\top$ --i.e., $\z=\boldsymbol w^{-1}$-- and notice that $C_{\q_i}(\z(\boldsymbol w))=\sum_{\boldsymbol\beta\in\Z_+^D} c_{\boldsymbol\beta}\boldsymbol{w}^{\boldsymbol\beta}$ is a power series in $\boldsymbol w$ and consequently converges in a complete Reinhardt domain $\mathscr D$. Therefore, its domain of convergency is the union of polydisks, and in each of them the convergence is absolute and uniform. In particular, the polydisk of convergence $\Delta(\boldsymbol r)\subset\mathscr D $ satisfies the extended Cauchy--Hadamard formula $\limsup_{|\boldsymbol\beta|\to\infty}\sqrt[|\boldsymbol\beta|]{|c_{\boldsymbol\beta}|\boldsymbol{r}^{\boldsymbol\beta}}=1$. The domain of
convergence of $C_{\q_i}$ contains a polyannulus of convergence with polyradii given by $\boldsymbol r=\boldsymbol 0$ and $\boldsymbol R=(r_1^{-1},\dots,r_D^{-1})$.

Let us show that  the second kind functions can be expressed   as multivariate Cauchy transforms of the MVOPR.
\begin{pro}\label{cauchy}
  The second kind functions satisfy
  \begin{align*}
       C_{\q_i}(\z)&=\int_\Omega \frac{P_{\q_i}(\y)}{(z_1-y_1)\cdots(z_D-y_D)}\dd\mu(\y), &
     \forall \z&\in\Ds_{\q_i}\setminus\operatorname{supp}(\dd\mu),  & i&=1,\dots,|[k]|.
  \end{align*}
\end{pro}
\begin{proof}
  See Appendix \ref{proof2}.
  \end{proof}
  We introduce
\begin{align*}
  \Gamma&\coloneq G \chi^*=\PARENS{\begin{matrix}
    \Gamma_{[0]}\\
    \Gamma_{[1]}\\
\vdots
  \end{matrix}}, & \Gamma_{[k]}&\coloneq \PARENS{\begin{matrix}
    \Gamma_{\q_1}\\
    \vdots\\
    \Gamma_{\q_{|[k]|}}
  \end{matrix}}.
\end{align*}
\begin{pro}
  The coefficients $\Gamma_{\q_i}$ are the multivariate Cauchy transform of the monomials $\x^{\q_i}$
  \begin{align*}
   \Gamma_{\q_i}(\x)=\int_\Omega\frac{\y^{\q_i}}{(x_1-y_1)\cdots(x_D-y_D)}\dd\mu(\y).
  \end{align*}
\end{pro}
\begin{proof}
  Is a byproduct of the proof of Proposition \ref{cauchy}.
\end{proof}
\begin{pro}
  We have $C=S\Gamma$.
\end{pro}

\begin{pro}
 In terms of Schur complements or quasi-determinants of bordered truncated moment matrices the functions of the second kind are
 \begin{align*}
 C_{[\ell]}=\SC\PARENS{\begin{array}{ccc|c}
       G_{[0],[0]}  & \cdots & G_{[0],[\ell-1]} & \Gamma_{[0]}\\
           \vdots      &       &            \vdots   &  \vdots\\
       G_{[\ell-1],[0]} & \cdots & G_{[\ell-1],[\ell-1]} & \Gamma_{[\ell-1]}\\[1pt]
\hline
             G_{[\ell],[0]} & \cdots & G_{[\ell],[\ell-1]} & \Gamma_{[\ell]}
      \end{array}}=\Theta_*\PARENS{\begin{matrix}
       G_{[0],[0]}  & \cdots & G_{[0],[\ell-1]} & \Gamma_{[0]}\\
           \vdots      &       &            \vdots   &  \vdots\\
       G_{[\ell-1],[0]} & \cdots & G_{[\ell-1],[\ell-1]} & \Gamma_{[\ell-1]}\\
             G_{[\ell],[0]} & \cdots & G_{[\ell],[\ell-1]} & \Gamma_{[\ell]}
      \end{matrix}}.
 \end{align*}
\end{pro}
In terms of determinant ratios  the components are
 \begin{align*}
 C_{\q^{(\ell)}_j}=\SC\PARENS{\begin{array}{ccc|c}
       G_{[0],[0]}  & \cdots & G_{[0],[\ell-1]} & \Gamma_{[0]}\\
       \vdots            & &  \vdots             &  \vdots\\
       G_{[\ell-1],[0]}  & \cdots & G_{[\ell-1],[\ell-1]} & \Gamma_{[\ell-1]}\\[1pt]
\hline
    G_{\q^{(\ell)}_j,[0]}  & \cdots & G_{\q^{(\ell)}_j,[\ell-1]} & \Gamma_{\q^{(\ell)}_j}
      \end{array}}=
      \frac{\begin{vmatrix}
                   G_{[0],[0]}  & \cdots & G_{[0],[\ell-1]} & \Gamma_{[0]}\\
       \vdots            & &  \vdots             &  \vdots\\
       G_{[\ell-1],[0]}  & \cdots & G_{[\ell-1],[\ell-1]} & \Gamma_{[\ell-1]}\\
    G_{\q^{(\ell)}_j,[0]}  & \cdots & G_{\q^{(\ell)}_j,[\ell-1]} & \Gamma_{\q^{(\ell)}_j}
                 \end{vmatrix}
}{ \begin{vmatrix}
 G_{[0],[0]}        & \cdots & G_{[0],[\ell-1]}\\
\vdots            &          &  \vdots     \\
  G_{[\ell-1],[0]}  & \cdots & G_{[\ell-1],[\ell-1]}
                 \end{vmatrix}}.
\end{align*}
\begin{definition}
Given $k$ distinct labels $a_1,\dots,a_k$ in $\{1,\dots, D\}$ we introduce the reduced second kind functions
\begin{align}\label{hatsecond}
     \widehat C_{a_1,\dots,a_k}=& \lim_{x_{a_1}\to\infty}\cdots\lim_{ x_{a_k}\to\infty}\Big(\Big[\prod_{i=1}^kx_{a_i}\Big]C\Big).
\end{align}
\end{definition}
Observe that the labels in the reduced second kind functions  indicate precisely those independent variables in
which they do not depend; therefore, when $k=D$ we have a constant.
For  the reduced second kind functions we find
\begin{pro}
  When $\operatorname{supp}(\dd\mu)$ is a bounded set the reduced second kind functions fulfill
\begin{align}\label{cauchyintegral}
 \widehat C_{\q_i,a_1,\dots,a_k}&=\int_\Omega\frac{P_{\q_i}(\y)}{\prod_{i=1}^{D-k}(z_{b_i}-y_{b_i})}\dd\mu(\y), &a &=1,\dots,|[\ell]|,&    \z\in\Ds_{\q_i}\setminus\operatorname{supp}(\dd\mu)
\end{align}
where $\{a_1,\dots,a_k\}\cup\{b_1,\dots,b_{D-k}\}=\{1,\dots,D\}$.
\end{pro}
\begin{proof}
The Lebesgue dominated convergence theorem  ensures that we can interchange the limit with the integral\footnote{The control or dominating function can be taken to be $g_{\z}(\y)=\frac{P_{\q_i}(\y)}{\prod_{i=1}^{D-k}(z_{b_i}-y_{b_i})}$}
\begin{align*}
 \lim_{z_{a_1}\to\infty}\cdots\lim_{ z_{a_k}\to\infty}\Big(\Big[\prod_{i=1}^kz_{a_i}\Big]C\Big)&=\int_\Omega
 \lim_{z_{a_1}\to\infty}\cdots\lim_{ z_{a_k}\to\infty}\Big(\Big[\prod_{i=1}^kz_{a_i}\Big)
\frac{P(\y)}{\prod_{a=1}^D(z_a-y_a)}\dd\mu(\y)\\
 &=\int_\Omega\Big[\prod_{i=1}^k
\lim_{z_{a_i}\to\infty}\Big(1-\frac{y_{a_i}}{z_{a_i}}\Big)^{-1}\Big]
\frac{P(\y)}{\prod_{i=1}^{D-k}(z_{b_i}-y_{b_i})}\dd\mu(\y).
\end{align*}
\end{proof}
From this result we infer that
\begin{align*}
  \widehat C_{[k],1,\dots,D}=H_{[0]}\delta_{0,k}.
\end{align*}

\subsection{The shift matrices}\label{los Lambda}
In this section we are going to discuss three term relations that extend the recursion relations existing in $D=1$ to the multivariate case. For this aim we need to introduce a set of $D$ shift matrices $\{\Lambda_1,\dots,\Lambda_D\}$ that play a very important role, they model the action of increasing by one the degree of the monomials.
\begin{definition}
The shift matrices are  given by
  \begin{align*}
\Lambda_a&=\PARENS{\begin{matrix}
                0   & (\Lambda_a)_{[0],[1]} & 0 & 0 &\cdots\\
0        & 0 &(\Lambda_a)_{[1],[2]} & 0 &\cdots\\
   0                &  0    &          0       & (\Lambda_a)_{[2],[3]}  &  \\
         0                &  0    &          0             &               0         &\ddots  \\
                         \vdots        &  \vdots    &  \vdots         &\vdots
                 \end{matrix}}
\end{align*}
where the entries in the non zero blocks are given by
\begin{align*}
      (\Lambda_a)_{\q^{(k)}_i,\q^{(k+1)}_j}&=\delta_{\q^{(k)}_i+\ee_a,\q^{(k+1)}_j},&
      a&=1,\dots, D, &
      i&=1,\dots,|[k]|,&
      j&=1,\dots,|[k+1]|.
\end{align*}
\end{definition}

Related to these   shift matrices we further introduce
\begin{definition}
\begin{enumerate}
    \item Given any $\kk=\sum_{a=1}^D k_{a}\ee_a\in\Z_+^D$ we define
 \begin{align*}
   \Lambda_{\kk}&\coloneq \Lambda_1^{k_1}\cdots\Lambda_D^{k_D} &
   \Pi_{\kk}&\coloneq \left(\Lambda_1^\top\right)^{k_1}\left(\Lambda_1\right)^{k_1}\cdots\left(\Lambda_D^\top\right)^{k_D}\left(\Lambda_D\right)^{k_D}
 \end{align*}
 \item For $\kk=n\ee_a$,  $a\in\{1,\dots, D\}$ and $n\in\Z_+$ we use the  notations
 \begin{align*}
   \Pi_{a,n}\coloneq &\Pi_{n\ee_a}=\left(\Lambda_a^\top\right)^n\left(\Lambda_a\right)^n.
 \end{align*}
 If $n=1$ we use $\Pi_{a,1}=\Pi_a=\Lambda_a^\top\Lambda_a$.
 \end{enumerate}
\end{definition}
Notice that $\kk_b+\ee_a\in[k+1]$.

\begin{pro}\label{pro:Lambda}
 \begin{enumerate}
 \item  The matrices $\Pi_{a,n}$ are projections, $\left(\Pi_{a,n}\right)^2=\Pi_{a,n}$ and $\Pi_{a,n}=\Pi_{a,n}^\top$. Moreover they are
 diagonal matrices whose non vanishing coefficients  are the unity.
 The ones are located precisely in the entries of the diagonal corresponding to the entries (monomials) in $\chi$ which contain
 $(x_a)^m$ with $m\geq n$ among its factors. We can write
     \begin{align*}
       \I=\Pi_{a,n}+\Pi_{a,n}^\perp,
     \end{align*}
     where $\Pi_{a,n}^\perp$ is a diagonal matrix with its non vanishing coefficients,
     which are equal to the unity, located in those entries of $\chi$ which contain $(x_a)^m$ with $m<n$ among its factors.
We also have
 \begin{align}\label{piachi}
\Pi_a^\perp\chi(\x)&=\chi(\x)|_{x_a=0}, &
  \Pi_a^\perp\chi^*(\x)&=x_a^{-1}\lim_{x_a\to\infty}\big(x_a\chi^*(\x)\big).
\end{align}
  \item  The shift matrices fulfill the  following properties  for all  $\kk,\ele\in\Z_+^D$
  \begin{align*}
\Lambda_{\kk}\Lambda_{\ele}&=\Lambda_{\kk+\ele}=\Lambda_{\ele} \Lambda_{\kk},  &
\Lambda_{\kk}(\Lambda_{\kk})^\top&=\mathbb{I}.
\end{align*}
\item When $a\neq b$ we have the commutation relations
\begin{align*}
\Lambda_{b}^\top\Lambda_{a}&=\Lambda_{a}\Lambda_{b}^\top,
&
  \Pi_{a}\Pi_{b}&=\Pi_{b}\Pi_{a}=\Pi_{\ee_a+\ee_b}
\end{align*}
\item We also have the ``eigenvalue'' type properties
\begin{align}\label{eigen}
\Lambda_{\kk}\chi(\x)&= \x^{\kk} \chi(\x),& \Lambda_{\kk}\chi^{*}(\x)&= \x^{-\kk} \chi^{*}(\x),\\\label{lambdaTchi}
\Lambda_{\kk}^\top\chi(\x)&= \x^{-\kk}\Pi_{\kk}\chi,  & \Lambda_{\kk}^\top\chi^{*}(\x)&=\x^{\kk}\Pi_{\kk}\chi^{*}(\x).
  \end{align}
   \end{enumerate}
\end{pro}
\begin{pro}\label{pro3}
  For $k$ distinct labels $a_1,\dots, a_k\in\{1,\dots,D\}$,  $a_i\neq a_j$ for $i\neq j$, and $k$ complex numbers  $q_{a_1},\dots,q_{a_k}\in\C$ we have
      \begin{multline}\label{lambdaTq}
       \Big[    \prod_{i=1}^k (\Lambda_{a_i}^\top-q_{a_i})\Big]\chi^*=\Big[\prod_{i=1}^k(x_{a_i}-q_{a_i})\Big]\chi^*
    +(-1)^k\lim_{x_{a_1}\to\infty}\cdots\lim_{ x_{a_{k}}\to\infty}\Big(\Big[\prod_{i=1}^kx_{a_i}\Big]\chi^*\Big),
    \\+\sum_{j=1}^{k-1}\frac{(-1)^j}{(k-j)!j!}\sum_{\sigma\in \mathfrak S_k}\Big(\Big[\prod_{i=j+1}^k\big(x_{a_\sigma(i)}-q_{a_{\sigma(i)}}\big)\Big]
    \lim_{x_{a_{\sigma(1)}}\to\infty}\cdots\lim_{ x_{a_{\sigma(j)}}\to\infty}\Big(\Big[\prod_{i=1}^jx_{a_\sigma(i)}\Big]\chi^*\Big)\Big),
  \end{multline}
  where $\mathfrak S_k$ denotes the symmetric group   of $k$ letters.
\end{pro}
\begin{proof}
See Appendix \ref{proof3}.
\end{proof}
\begin{pro}The moment matrix $G$ satisfies
  \begin{align}\label{eq:symmetry}
 \Lambda_{\kk} G&= G \big(\Lambda_{\kk}\big)^\top,&  \forall &\kk\in\Z_+^D.
  \end{align}
\end{pro}
\begin{proof}
It follows from Definition \ref{moment} of the moment matrix  $G$ and the eigen-value property  in Proposition \ref{pro:Lambda} for  $\Lambda_{\kk}$.
\end{proof}
\subsection{Jacobi matrices and three terms relations}
Once the shift matrices have been introduced we are ready to discuss its \emph{dressing}, that leads to the Jacobi matrices which are extremely important not only for the general theory of MVOPR but also for exploring its connection with the Toda theory.

\begin{definition}
We introduce the following Jacobi type matrices
 \begin{align}\label{def:J}
  J_{\kk}&\coloneq S \Lambda_{\kk} S^{-1}, &  \forall &\kk\in\Z_+^D,
 \end{align}
 and the basic Jacobi matrices
 \begin{align*}
 J_a&\coloneq J_{\ee_a}.
 \end{align*}
\end{definition}
\begin{pro}
The  Jacobi type  matrices satisfy
  \begin{align}\label{pro9}
J_{\kk}^{\top} &=H^{-1}J_{\kk}H, & J_{\kk}P&=\x^{\kk}P, & \forall &\kk\in\Z_+^D.
 \end{align}
\end{pro}
\begin{proof}
  Using the Cholesky factorization of the moment matrix $G$ we get
\begin{align*}
S \Lambda_{\kk} S^{-1}= H \left(S \Lambda_{\kk} S^{-1}\right)^{\top} H^{-1}.
\end{align*}
The eigen-value property is obvious from \eqref{eigen}.
\end{proof}
For the second kind functions we have

\begin{pro}\label{pro4}
 For $k$ distinct labels $a_1,\dots, a_k\in\{1,\dots,D\}$,  $a_i\neq a_j$ for $i\neq j$, and $k$ complex numbers  $q_{a_1},\dots,q_{a_k}\in\C$ we have
  \begin{align*}
   \Big[    \prod_{i=1}^k (J_{a_i}-q_{a_i})\Big]C\coloneq & \Big[\prod_{i=1}^k(x_{a_i}-q_{a_i})\Big]C+(-1)^k\widehat C_{a_{1},\dots,a_{k}}\\
    &+\sum_{j=1}^{k-1}\frac{(-1)^j}{(k-j)!j!}\sum_{\sigma\in \mathfrak S_k}\Big(\Big[\prod_{i=j+1}^k\big(x_{a_\sigma(i)}-q_{a_{\sigma(i)}}\big)\Big]
    \widehat C_{a_{\sigma(1)},\dots,a_{\sigma(j)}}\Big),
  \end{align*}
\end{pro}
\begin{proof}
See Appendix \ref{proof4}.
\end{proof}
From these properties we easily conclude that
\begin{pro}\label{explicit jacobi}
  The explicit form of the basic Jacobi matrices is
  \begin{align*}
J_{a}&= \PARENS{\begin{matrix}
 (J_a)_{[0],[0]}  & (J_a)_{[0],[1]}   &            0              &         0                  &          0&  \cdots         &\\
(J_a)_{[1],[0]}    & (J_a)_{[1],[1]}        &  (J_a)_{[1],[2]}    &   0                 &  0& \cdots\\
              0       &   (J_a)_{[2],[1]} &   (J_a)_{[2],[2]}      &   (J_a)_{[2],[3]}       &      0              &\cdots\\
               0      &     0                &  (J_a)_{[3],[2]}       &   (J_a)_{[3],[3]}       & (J_a)_{[3],[4]}  &\\
                 \vdots    &  \vdots                   &             \quad\quad  \quad \ddots          &      \quad\quad  \quad\ddots      &   \quad\quad  \quad\ddots            &
 \end{matrix}},
\end{align*}
where
\begin{align*}
 (J_a)_{[0],[0]}&=-(\Lambda_a)_{[0],[1]} \beta_{[1]} &
 (J_a)_{[0],[1]}&=(\Lambda_a)_{[0],[1]} \\
\end{align*}
and
\begin{align*}
 (J_a)_{[k],[k-1]}&=H_{[k]}\left[(\Lambda_a)_{[k-1],[k]} \right]^\top \left(H_{[k-1]}\right)^{-1},  \\
 (J_a)_{[k],[k]}&=\beta_{[k]} (\Lambda_a)_{[k-1],[k]}-(\Lambda_a)_{[k],[k+1]}\beta_{[k+1]} ,\\
 (J_a)_{[k],[k+1]}&=(\Lambda_a)_{[k],[k+1]}.
\end{align*}
\end{pro}

From \eqref{def:J} and \eqref{pro9} it is easy to see that the $J_{\kk}$ only have $(2|\kk|+1)$ block diagonals that do not vanish. These
are the $|\kk|$ first block superdiagonals, the diagonal itself and $|\kk|$ first block subdiagonals.

\begin{definition}\label{vectores}
We introduce the following objects
\begin{align*}
 \boldsymbol \Lambda&\coloneq (\Lambda_1,\dots,\Lambda_D)^\top,& \boldsymbol J&\coloneq (J_1,\dots,J_D)^\top,&
 \boldsymbol{\widehat C}&\coloneq (\widehat C_{1},\dots,\widehat C_{D})^\top.
\end{align*}
and given any vector $\boldsymbol n=(n_1,\dots,n_D)^\top\in\R^D$ we define the following \emph{dot products}
\begin{align*}
  \boldsymbol n\cdot\boldsymbol\Lambda&\coloneq \sum_{a=1}^Dn_a\Lambda_a, &
    \boldsymbol n\cdot\boldsymbol J&\coloneq \sum_{a=1}^Dn_a J_a.
\end{align*}
\end{definition}
The celebrated three term relations \cite{Dunkl}  in the multivariate context are
\begin{pro}
The MVOPR satisfy  the following three term  relations\footnote{Observe that for $k=0$ we get
  \begin{align*}
   \boldsymbol n\cdot\boldsymbol x+(\boldsymbol n\cdot\boldsymbol \Lambda)_{[0],[1]} \beta_{[1]}&=(\boldsymbol n\cdot\boldsymbol \Lambda)_{[0],[1]}P_{[1]}(\x).
 \end{align*}
which agrees with the definition of $P$ in terms of the factorization matrix, this is, $P_{[1]}(\x)^\top=(x_1,\dots,x_D)+\beta_{[1]}^\top$.}
 \begin{multline}\label{eq: recursion MVORP}
 ( \boldsymbol n\cdot\boldsymbol  x)P_{[k]}=H_{[k]}(\boldsymbol n\cdot\boldsymbol \Lambda)_{[k-1],[k]}^\top H_{[k-1]}^{-1} P_{[k-1]}+\big(\beta_{[k]} (\boldsymbol n\cdot\boldsymbol \Lambda)_{[k-1],[k]}-(\Lambda_a)_{[k],[k+1]}\beta_{[k+1]}\big)P_{[k]}+
(\boldsymbol n\cdot\boldsymbol \Lambda)_{[k],[k+1]}P_{[k+1]},
 \end{multline}
for $k=1,2,\dots$.
 The second kind functions satisfy\footnote{For $k=0$ we get
 \begin{align*}
   (x_a+(\Lambda_a)_{[0],[1]} \beta_{[1]})C_{[0]}(\x)&=(\Lambda_a)_{[0],[1]}C_{[1]}(\x)+\hat C_{[0],a}(\x).
 \end{align*}}
  \begin{multline}\label{eq: recursion Cauchy}
  (\boldsymbol n\cdot\boldsymbol x)C_{[k]}=H_{[k]}(\boldsymbol n\cdot\boldsymbol \Lambda)_{[k-1],[k]}^\top H_{[k-1]}^{-1} C_{[k-1]}+\big(\beta_{[k]} (\boldsymbol n\cdot\boldsymbol \Lambda)_{[k-1],[k]}-(\boldsymbol n\cdot\boldsymbol \Lambda)_{[k],[k+1]}\beta_{[k+1]}\big)C_{[k]}\\
  +(\boldsymbol n\cdot\boldsymbol \Lambda)_{[k],[k+1]}C_{[k+1]}-\boldsymbol n\cdot\boldsymbol {\widehat C}_{[k]},
 \end{multline}
 for $k=1,2,\dots$.
\end{pro}
\begin{proof}
  They are an immediate consequence of \eqref{eigen} and Propositions \ref{pro4}, for $k=1$,  and \ref{explicit jacobi}.
\end{proof}
\subsection{Christoffel--Darboux formul{\ae}}
\begin{definition}\label{DEFCD}
 The Christoffel--Darboux  kernels are
 \begin{align*}
  K^{(\ell)}(\x,\y)&\coloneq \left(\chi^{[\ell]}(\x)\right)^{\top} \left(G^{[\ell]} \right)^{-1}\chi^{[\ell]}(\y)=
  \sum_{k=0}^{\ell-1}\left(P_{[k]}(\x)\right)^\top \left(H_{[k]} \right)^{-1}P_{[k]}(\y),
  \end{align*}
while the second kind Christoffel--Darboux kernels are given by
\begin{align*}
 Q^{(\ell)}(\x,\y)&\coloneq \left(\big(\chi^*\big)^{[\ell]}(\x)\right)^{\top} \left(G^{[\ell]} \right)\big(\chi^*\big)^{[\ell]}(\y)=
  \sum_{k=0}^{\ell-1}\left(C_{[k]}(\x)\right)^\top \left(H_{[k]} \right)^{-1}C_{[k]}(\y).
 \end{align*}
\end{definition}
The Christoffel--Darboux kernel $K^{(\ell+1)}(\x,\y)$ gives the projection, $S_\ell:L^2(\R^D,\mu)\to L^2(\R^D,\mu)$, on the set of MOVPR of degree $\ell$ or less, given any vector in the real Hilbert space $f\in L^2(\R^D,\mu)$ its orthogonal projection is given by the truncated Fourier series
$S_{\ell}(f)(\x)=\int_\Omega K^{(\ell+1)}(\x,\y) f(\y)\dd\mu(\y)$. In fact, if $P$ is a MOVPR of degree $\ell$ or less then $P=S_{\ell}(P)$. and  $S_{\ell}\circ S_\ell=S_\ell$;   these kernels are subject to the \emph{reproducing property} $K^{(\ell+1)}(\x,\y)=\int_\Omega K^{(\ell+1)}(\x,\boldsymbol z)\dd\mu(\z)K^{(\ell+1)}(\boldsymbol z,\y)$. This projection is the best approximation to $f$ with MOVPR of degree $\ell$ or less, in the sense that the mean square distance,  $\int_\Omega (f(\x)-S_\ell(f)(\x))^2\dd\mu(\x)$, is minimized and all the other polynomials of  degree $\ell$ or less have a bigger mean square distance to $f$. When the space of MVOPR is dense in $L^2(\R^D,\mu)$ then the Fourier series converge in the mean square distance to $f$, $\lim_{\ell\to\infty}\int_\Omega (f(\x)-S_{\ell}(f)(\x))^2\dd\mu(\x)=0$. For more information see \S 3.5 of \cite{Dunkl}.

From  Definition \ref{DEFCD} it directly follows that
\begin{pro}
  In terms of Schur complements or quasi-determinants of bordered truncated moment matrices the Christoffel--Darboux kernel is expressed as
  \begin{align*}
 K^{[\ell]}(\x,\y)&=\operatorname{SC}\PARENS{\begin{array}{ccc|c}
       G_{[0],[0]}  & \cdots & G_{[0],[\ell-1]} & \chi_{[0]}(\y)\\
           \vdots      &       &            \vdots   &  \vdots\\
       G_{[\ell-1],[0]} & \cdots & G_{[\ell-1],[\ell-1]} & \chi_{[\ell-1]}(\y)\\[2pt]
\hline
           \bigstrut[t]  \chi_{[0]}^{\top}(\x) & \cdots & \chi_{[\ell-1]}^{\top}(\x) & 0
      \end{array}}
      =\Theta_*\PARENS{\begin{matrix}
       G_{[0],[0]}  & \cdots & G_{[0],[\ell-1]} & \chi_{[0]}(\y)\\
           \vdots      &       &            \vdots   &  \vdots\\
       G_{[\ell-1],[0]} & \cdots & G_{[\ell-1],[\ell-1]} & \chi_{[\ell-1]}(\y)\\ \bigstrut[t]
            \chi_{[0]}^{\top}(\x) & \cdots & \chi_{[\ell-1]}^{\top}(\x) & 0
      \end{matrix}}.
\end{align*}
while the second kind Christoffel--Darboux kernel have the following one
\begin{align*}
 Q^{[\ell]}(\x,\y)&=\operatorname{SC}\PARENS{\begin{array}{ccc|c}
       (G^{-1})_{[0],[0]}  & \cdots & (G^{-1})_{[0],[\ell-1]} & \chi^*_{[0]}(\y)\\
           \vdots      &       &            \vdots   &  \vdots\\
       (G^{-1})_{[\ell-1],[0]} & \cdots & (G^{-1})_{[\ell-1],[\ell-1]} & \chi^*_{[\ell-1]}(\y)\\[1pt]
\hline
 \bigstrut[t]    \bigstrut[t]  (\chi^*_{[0]})^{\top}(\x) & \cdots & (\chi^*_{[\ell-1]})^{\top}(\x) & 0
      \end{array}}\\&=\Theta_*\PARENS{\begin{matrix}
       (G^{-1})_{[0],[0]}  & \cdots & (G^{-1})_{[0],[\ell-1]} & \chi^*_{[0]}(\y)\\
           \vdots      &       &            \vdots   &  \vdots\\
       (G^{-1})_{[\ell-1],[0]} & \cdots & (G^{-1})_{[\ell-1],[\ell-1]} & \chi^*_{[\ell-1]}(\y)\\
            (\chi^*_{[0]})^{\top}(\x) & \cdots & (\chi^*_{[\ell-1]})^{\top}(\x) & 0
      \end{matrix}}.
\end{align*}
\end{pro}
In terms of determinants we have
 \begin{align*}
 K^{[\ell]}(\x,\y)=\frac{\begin{vmatrix}
                   G_{[0],[0]}  & \cdots & G_{[0],[\ell-1]} & \chi_{[0]}(\y)\\
       \vdots            & &  \vdots             &  \vdots\\
       G_{[\ell-1],[0]}  & \cdots & G_{[\ell-1],[\ell-1]} & \chi_{[\ell-1]}(\y)\\ \bigstrut[t]
    \chi_{[0]}^{\top}(\x)  & \cdots & \chi_{[\ell-1]}^{\top}(\x) & 0
                 \end{vmatrix}
}{ \begin{vmatrix}
 G_{[0],[0]}        & \cdots & G_{[0],[\ell-1]}\\
\vdots            &          &  \vdots     \\
  G_{[\ell-1],[0]}  & \cdots & G_{[\ell-1],[\ell-1]}
                 \end{vmatrix}}.
\end{align*}
\begin{align*}
 Q^{[\ell]}(\x,\y)=\begin{vmatrix}
                   (G^{-1})_{[0],[0]}  & \cdots & (G^{-1})_{[0],[\ell-1]} & \chi^*_{[0]}(\y)\\
       \vdots            & &  \vdots             &  \vdots\\
       (G^{-1})_{[\ell-1],[0]}  & \cdots & (G^{-1})_{[\ell-1],[\ell-1]} & \chi^*_{[\ell-1]}(\y)\\ \bigstrut[t]
    (\chi^*_{[0]})^{\top}(\x)  & \cdots & (\chi^*_{[\ell-1]})^{\top}(\x) & 0
                 \end{vmatrix}
\cdot \begin{vmatrix}
 G_{[0],[0]}        & \cdots & G_{[0],[\ell-1]}\\
\vdots            &          &  \vdots     \\
  G_{[\ell-1],[0]}  & \cdots & G_{[\ell-1],[\ell-1]}
                 \end{vmatrix}.
\end{align*}

\begin{theorem}
  The following  Christoffel--Darboux formula holds
\begin{align}    \label{kerr}
  K^{(\ell)}(\x,\y) =&\frac{
  \big((\boldsymbol n\cdot\boldsymbol \Lambda)_{[\ell-1],[\ell]}P_{[\ell]}(\x)\big)^\top
  \big(H_{[\ell-1]} \big)^{-1}
  P_{[\ell-1]}(\y)-  P_{[\ell-1]}(\x)^\top \big(H_{[\ell-1]} \big)^{-1}(\boldsymbol n\cdot\boldsymbol \Lambda)_{[\ell-1],[\ell]}P_{[\ell]}(\y)
  }{\boldsymbol n\cdot(\boldsymbol {x}-\boldsymbol{y})},
  \end{align}
  while for the second kind kernel we have
  \begin{align}
                 \label{kerQQ}
  Q^{(\ell)}(\x,\y) =&\frac{\left[(\boldsymbol n\cdot\boldsymbol \Lambda)_{[\ell-1],[\ell]}C_{[\ell]}(\x)\right]^\top\left(H_{[\ell-1]} \right)^{-1}\left[C_{[\ell-1]}(\y)\right]-  \left[C_{[\ell-1]}(\x)\right]^\top \left(H_{[\ell-1]} \right)^{-1}\left[(\boldsymbol n\cdot\boldsymbol \Lambda)_{[\ell-1],[\ell]}C_{[\ell]}(\y) \right]}{\boldsymbol n\cdot(\boldsymbol {x}-\boldsymbol{y})}\\
               \nonumber     &+\frac{\big[\boldsymbol n\cdot\boldsymbol {\widehat C}^{[\ell]}(\x) \big]^{\top}\left(H^{[\ell]}\right)^{-1}C^{[\ell]}(\y)-\left[C^{[\ell]}(\x)\right]^{\top}\left(H^{[\ell]}\right)^{-1}\boldsymbol n\cdot\boldsymbol {\widehat C}^{[\ell]}(\y)}{\boldsymbol n\cdot(\boldsymbol {x}-\boldsymbol{y})}.
\end{align}
\end{theorem}

\begin{proof}

In the first place, for the polynomials,  in the one hand we have
\begin{align*}
 (P^{[\ell]}(\x))^\top (H^{-1}\n\cdot\boldsymbol J)^{[\ell]} P^{[\ell]}(\y)&=
 (\n\cdot\y) (P^{[\ell]}(\x))^\top (H^{[\ell]})^{-1} P^{[\ell]}(\y)-
 P_{[l-1]}(\boldsymbol{x})^{\top}\left(H_{[\ell-1]}\right)^{-1}(\n\cdot\boldsymbol J)_{[\ell-1],[\ell]}P_{[l]}(\boldsymbol{y}),
 \end{align*}
 and in the other hand
 \begin{align*}
 (P^{[\ell]}(\x))^\top (H^{-1}\n\cdot\boldsymbol J)^{[\ell]} P^{[\ell]}(\y)&=
 (\n\cdot\x) (P^{[\ell]}(\x))^\top (H^{[\ell]})^{-1} P^{[\ell]}(\y)-
\big((\n\cdot\boldsymbol J)_{[\ell-1][\ell]}P_{[\ell]}(\x))^{\top}(H_{[\ell-1]})^{-1}P_{[\ell-1]}(\y).
\end{align*}
For the second kind functions we  proceed similarly.  However, we must take care of the appearing of
 the reduced second kind functions. In this case the two possibilities are
\begin{align*}
 (C^{[\ell]}(\x))^\top (H^{-1}\n\cdot\boldsymbol J )^{[\ell]} C^{[\ell]}(y)&=\begin{multlined}[t]
(\n\cdot\y)(C^{[\ell]}(\x))^\top(H^{[\ell]})^{-1}C^{[\ell]}(\y)-
 (C^{[\ell]}(\x))^\top(H^{[l]})^{-1}(\n\cdot \boldsymbol{\hat{C}}^{[\ell]}(\y))\\-
 C_{[\ell-1]}(\x)^{\top}(H_{[\ell-1]})^{-1}(\n\cdot\boldsymbol J)_{[\ell-1],[\ell]}
 C_{[\ell]}(\y)
 \end{multlined}
 \\
 (C^{[\ell]}(\x))^\top (H^{-1}\n\cdot\boldsymbol J )^{[\ell]} C^{[\ell]}(y)&=\begin{multlined}[t]
(\n\cdot\x)(C^{[\ell]}(\x))^\top(H^{[\ell]})^{-1}C^{[\ell]}(\y)-
( \n\cdot\boldsymbol{\hat{C}}^{[\ell]}(\x))^{\top}(H^{[\ell]})^{-1} C^{[\ell]}(\y)\\-
((\n\cdot\boldsymbol J)_{[\ell-1],[\ell]} C_{[\ell]}(\x))^\top(H_{[l-1]})^{-1}C_{[\ell-1]}(\y).
 \end{multlined}
\end{align*}
\end{proof}
 Observe also that \eqref{kerQQ} is not a standard Christoffel--Darboux formula because the last term involves all the reduced second kind functions. These terms are absent in the scalar case but in this multivariant scenario show up.

\section{On discrete Toda  and MOVPR}\label{3}

In this section we discuss the connection between MVOPR associated with
different measures; these collection of measures can be considered as a lattice of measures.
 The set of transformations which we are about to introduce  is not the more general one but capture the essential facts.
\subsection{The discrete flows}
  For the construction of $D$ discrete flows we consider an invertible matrix \begin{align*}
    N=(n_{a,b})_{a,b=1,\dots, D}\in\operatorname{GL}(\R^D),
  \end{align*}
   and therefore  $D$ linearly independent vectors $\n_a=(n_{a,1},\dots,n_{a,D})^\top$, and a vector $\boldsymbol q=(q_1,\dots,q_D)^\top\in\R^D$, where $q_a\neq 0$, $ a=\{1,\dots,D\}$.
 For a given measure $\dd\mu$ and each multi-index $\m=(m_1,\dots,m_D)^\top\in\Z^D$ we consider the measure
 \begin{align*}
 \dd\mu_{\m}(\x)=\Big[\prod_{a=1}^D\big(\n_a\cdot\x-q_a\big)^{m_a}\Big]\dd\mu(\x).
\end{align*}

Associated with this deformed measure we  introduce the  set
\begin{align}\label{def:R}
  R\coloneq \{\x\in\R^D: |\n_1\cdot\x|<|q_1|,\dots,|\n_D\cdot\x|<|q_D| \},
\end{align}
and the related sets $R_a\coloneq \{ \x\in\R^D: -|q_a|<\n_a\cdot\x<|q_a|\}$, $a\in\{1,\dots,D\}$. Observe that  $R=\cap_{a=1}^D R_a$ is a bounded  open convex polytope included in the ball centered at the origin of radius $\max_{a\in\{1,\dots, D\}}|q_a|$. We see that the border of $R_a$ is $\partial R_a=\pi_a^+\cup\pi_a^-$ in terms of the hyperplanes $\pi_a^\pm\coloneq \{\x\in\R^D: \n_a\cdot\x=\pm q_a\}$.
The measure $\dd\mu_{\m}$ has a definite sign in $R\cap\operatorname{supp}(\mu)$ since the hyperplane $\pi_a^+:\,\n_a\cdot\x=q_a$ belong to the border and therefore is unreachable in $R$.

As we will see later on, \S \ref{s.darboux},  these discrete flows are built up in terms of Darboux transformations. Sometimes the flows described by $m_a\to m_a+1$ are known as Christoffel transformations and those associated  with $m_a\to m_a-1$ as Geronimus transformations.

A natural question arises here: Which are the corresponding moment matrices? and the answer, given in terms of shift matrices, is fairly nice.
\begin{pro}
For a given  Borel measure $\mu$ let us assume that $\operatorname{supp}\mu\subset R$, with $R$ given in \eqref{def:R},  then the moment matrices $G(\m)$ of the measures $\dd\mu_{\m}$ satisfy
  \begin{align}\label{eq: Miwa-G}
    G(\m)=\bigg(\prod_{a=1}^D( \n_a\cdot\boldsymbol\Lambda-q_a)^{m_a}\bigg)G=G\bigg(\prod_{a=1}^D( \n_a\cdot\boldsymbol\Lambda^\top-q_a)^{m_a}\bigg).
  \end{align}
\end{pro}
\begin{proof}
We need to be specially careful when $m_a$ is a negative integer because, if that case we are dealing with powers of the  inverse matrix of $( \n_a\cdot\boldsymbol\Lambda-q_a)$, $a\in\{1,\dots,D\}$. The request $q_a\neq 0$  ensures that $( \n_a\cdot\boldsymbol\Lambda-q_a)^{-1}$ can be formally given as the following upper triangular matrix
\begin{align*}
  ( \n_a\cdot\boldsymbol\Lambda-q_a)^{-1}=-q_a^{-1}-q_a^{-2} (\n_a\cdot\boldsymbol\Lambda)^2-q_a^{-3} (\n_a\cdot\boldsymbol\Lambda)^2-\cdots.
\end{align*}
This series is a matrix organized by superdiagonals with $ \big((\n_a\cdot\boldsymbol\Lambda-q_a)^{-1}\big)_{[k],[k+j]}=-q_a^{-(j+1)}( (\n_a\cdot\boldsymbol\Lambda)^j)_{[k],[k+j]}$ the $j$-th block in the $k$-th superdiagonal, and no series is involved for a given block; i.e., the expression is not only formal but it is well defined. However, we should also tackle the more subtle problem of the domain of these matrices. In particular, its action on $\chi$ gives $(\n_a\cdot\boldsymbol\Lambda)\chi(\x)=(\n_a\cdot\boldsymbol\x)\chi(\x)$ and corresponding  series is
$  ( \n_a\cdot\boldsymbol\Lambda-q_a)^{-1}\chi=-q_a^{-1}-q_a^{-2}(\n_a\cdot\boldsymbol\x)-q_a^{-3}(\n_a\cdot\boldsymbol\x)^2-\cdots$ which converges for
           $|\n_a\cdot\boldsymbol\x|<|q_a|$.
Recalling Proposition \ref{pro:Lambda} the result follows.
\end{proof}
\begin{definition}
\begin{enumerate}
  \item We introduce
  \begin{align*}
    W_0(\m)\coloneq \prod_{a=1}^D( \n_a\cdot\boldsymbol\Lambda-q_a)^{m_a}.
  \end{align*}
  \item The action of the  translation $T_a$ on any function $f$ on $\Z^D$ is defined by
  \begin{align*}
    (T_af)(m_1,\dots,m_a,\dots,m_D)=f(m_1,\dots,m_a+1,\dots,m_D),
  \end{align*}
and the partial difference operator is given by
  \begin{align*}
    \Delta_a\coloneq T_a-1.
  \end{align*}
  These translations depend on $N,\boldsymbol q$ and when needed we use the notation $T^{(N,\boldsymbol q)}_a$ or $T^{(\boldsymbol q)}_a$ to indicate it.
\item The MVOPR and the corresponding second kind functions associated with $\dd\mu_{\m}(\x)$ will be denoted by $P(\m,\x)$ and $C(\m,\x)$.
\item Assuming  the block Cholesky factorization for\footnote{With $S(\m)$ block lower triangular with the block diagonal populated by identities, and $H(\m)$ block diagonal.} $G(\m)$, $G(\m)=(S(\m))^{-1}H(\m)\big(S(\m)^{-1}\big)^\top$, for each $\m\in\Z^D$ we introduce the following semi-infinite matrices
 \begin{align}\label{def:Ma}
  M_a(\m)&\coloneq S(\m) \big((T_aS)(\m)\big)^{-1}.
   \end{align}
\end{enumerate}
\end{definition}

For the sake of simplicity, from hereon and when not needed we will omit writing the $\m$-dependence and it will be implicitly assumed.
\begin{pro}\label{pro:discrete G}
 The moment matrix satisfies
  \begin{align*}
    T_aG=( \n_a\cdot\boldsymbol\Lambda-q_a)G=G( (\n_a\cdot\boldsymbol\Lambda)^\top-q_a).
  \end{align*}
  \end{pro}
  \begin{proof}
    We observe that \eqref{eq: Miwa-G} could be written as
  \begin{align*}
    G(\m)=W_0(\m)G=G\big(W_0(\m)\big)^\top.
  \end{align*}
  \end{proof}
\begin{pro}
The matrix $M_a=S(T_aS)^{-1}$ fulfills
 \begin{align}\label{eq:Ma+alternative}
M_a&=H\big((T_aS)( \n_a\cdot\boldsymbol\Lambda-q_a)S^{-1}\big)^\top(T_aH)^{-1}.
\end{align}
Moreover, $M_a$ is a  block  lower unitriangular matrix with only the first subdiagonal different from zero; i.e., $M_a=\I+\rho_a$ with
\begin{align}
\rho_a&=H ( \n_a\cdot\boldsymbol\Lambda)^\top (T_aH)^{-1}\label{consequence}\\
&=-\Delta_a\beta. \label{ultima}
 \end{align}
\end{pro}
\begin{proof}
For \eqref{eq:Ma+alternative} introduce the Cholesky factorization in the second equality in \eqref{eq: Miwa-G}, this equation has as direct consequence \eqref{consequence}. The relation \eqref{ultima} follows from  \eqref{def:Ma} and \eqref{eq:Ma+alternative}.
\end{proof}
 Componentwise we have
  \begin{align*}
\rho_a\coloneq &\PARENS{\begin{matrix}
0                                                                                      &   0                     &   0                     &                \cdots        \\
\rho_{a,[1]}   &  0& 0                       &              \cdots         \\
0                                    & \rho_{a,[2]} &0        &         & \\
         \vdots                           &      \ddots                  &            \ddots       &                         \ddots
                 \end{matrix}}
 \end{align*}
 with
 \begin{align}\label{rho1}
 \rho_{a,[k]}&\coloneq H_{[k]}[( \n_a\cdot\boldsymbol\Lambda)_{[k-1],[k]}]^\top \big( T_aH_{[k-1]}\big) ^{-1} 
\\ &=-\Delta_a\beta_{[k]}.\label{rho2}
 \end{align}
\begin{definition}
  We introduce the wave matrices\footnote{These definitions are motivated by the relation $W_1 G= W_2$.}
  \begin{align*}
    W_1\coloneq &SW_0, & W_2\coloneq &H\big(S^{-1}\big)^\top,
    \end{align*}
    and the lattice resolvents
    \begin{align}\label{eq:def.omega}
    \omega_a\coloneq & (T_aH) M_a^\top H^{-1}, & a\in&\{1,\dots, D\}.
  \end{align}
\end{definition}
  \begin{pro}\label{pro:G_and_evolved_waves2}
The evolved wave functions $W_1$ and $W_2$ satisfy the following
\begin{align}\label{toto2}
 G=&W_1(\m)^{-1}W_2(\m).
\end{align}
\end{pro}
\begin{proof}
From the Cholesky factorization we deduce that
  \begin{align}\label{toto}
 G=W_0(\m)^{-1}(S(\m))^{-1}\,H(\m)\big((S(\m))^{-1}\big)^\top=W_1(\m)^{-1}W_2(\m).
\end{align}
\end{proof}
For the lattice resolvent we have
\begin{pro}
The lattice resolvent can be expressed  as
\begin{align} \label{omJ}
  \omega_a=&(T_aS)( \n_a\cdot\boldsymbol\Lambda-q_a)S^{-1}.
  \end{align}
  Moreover, we have the explicit form
  \begin{align*}
    \omega_a= \alpha_a+\n_a\cdot\boldsymbol\Lambda
  \end{align*}
  where the diagonal term have the following alternative expressions
  \begin{align}
\alpha_a =&(T_aH)H^{-1}\label{omJ2}\\
 =& (T_a\beta)( \n_a\cdot\boldsymbol\Lambda)-
 ( \n_a\cdot\boldsymbol\Lambda)\beta- q_a.\label{resolvent beta}
\end{align}
\end{pro}
Componentwise we have
\begin{align}\label{H1}
\alpha_{a,[k]}&= (T_aH_{[k]})H_{[k]} ^{-1}
\\&=(T_a\beta_{[k]})( \n_a\cdot\boldsymbol\Lambda)_{[k-1],[k]}-
 ( \n_a\cdot\boldsymbol\Lambda)_{[k],[k+1]}\beta_{[k+1]}- q_a.\label{H2}
\end{align}
\begin{proof}
The first relation is a consequence of \eqref{eq:Ma+alternative} and  \eqref{eq:def.omega}, then \eqref{omJ} and  \eqref{omJ2} are consequences of \eqref{omJ} and \eqref{eq:def.omega}.
Finally, from \eqref{omJ} and  \eqref{omJ2} we infer \eqref{resolvent beta}.
\end{proof}
As byproduct we get
\begin{pro}\label{todabetaH}
The following equations hold
\begin{align*}
(T_aH_{[k]})H_{[k]} ^{-1}&=(T_a\beta_{[k]})( \n_a\cdot\boldsymbol\Lambda)_{[k-1],[k]}-
 ( \n_a\cdot\boldsymbol\Lambda)_{[k],[k+1]}\beta_{[k+1]}- q_a,\\
 H_{[k]}[( \n_a\cdot\boldsymbol\Lambda)_{[k-1],[k]}]^\top \big( T_aH_{[k-1]}\big) ^{-1}&=-\Delta_a\beta_{[k]}.
\end{align*}
\end{pro}
We are now ready to show what type nonlinear partial difference equations of Toda type are satisfied.

\begin{theorem}\label{toda equations 1}
  The quasi-tau matrices $ H_{[k]}$ are subject to the following equations
  \begin{multline*}
   \Delta_b\big( (\Delta_aH_{[k]})H_{[k]} ^{-1}\big)=
     ( \n_a\cdot\boldsymbol\Lambda)_{[k],[k+1]}H_{[k+1]}[( \n_b\cdot\boldsymbol\Lambda)_{[k],[k+1]}]^\top \big( T_bH_{[k]}\big) ^{-1}\\
    -(T_aH_{[k]})[( \n_b\cdot\boldsymbol\Lambda)_{[k-1],[k]}]^\top \big( T_aT_bH_{[k-1]}\big) ^{-1}( \n_a\cdot\boldsymbol\Lambda)_{[k-1],[k]}.
  \end{multline*}
The  matrices $\beta_{[k]}$ fulfill
  \begin{multline*}
\big((T_b\beta_{[k]})( \n_b\cdot\boldsymbol\Lambda)_{[k-1],[k]}-
 ( \n_b\cdot\boldsymbol\Lambda)_{[k],[k+1]}\beta_{[k+1]}- q_b \big) (\Delta_a\beta_{[k]}) \\=T_b\Delta_a\beta_{[k]}\big((T_aT_b\beta_{[k]})( \n_b\cdot\boldsymbol\Lambda)_{[k-1],[k]}-
 ( \n_b\cdot\boldsymbol\Lambda)_{[k],[k+1]}(T_a\beta_{[k+1]})- q_b\big).
  \end{multline*}
\end{theorem}
\begin{proof}
   An immediate consequence of Proposition \ref{todabetaH}.
\end{proof}
Let us stress  that the matrices in the nonlinear lattice have increasing sizes.

\begin{pro}\label{prepa}
 If $T_aG$ and $G$ admit Cholesky decompositions then for each $a\in\{1,\dots,D\}$ we have the following $LU$ factorization
  \begin{align}\label{eq:Cholesky for J}
\n_a\cdot\boldsymbol J-q_a=M_a \omega_a,
  \end{align}
and the  $UL$ factorization
 \begin{align}\label{eq: UL for TJ}
    T_a(  \n_a\cdot\boldsymbol J)-q_a=&\omega_a M_a.
  \end{align}
\end{pro}
\begin{proof}
From Proposition \ref{pro:discrete G} we get
\begin{flalign*}
  T_aG=&(T_aS)^{-1}(T_aH)\big((T_aS)^{-1}\big)^\top &\text{Cholesky factorization of $(T_aG)$}\\
          =&(\n_a\cdot\boldsymbol J-q_a)G&\text{see Proposition \ref{pro:discrete G}}\\
           =&( \n_a\cdot\boldsymbol J-q_a)S^{-1}H\big(S^{-1}\big)^\top&\text{Cholesky factorization of $G$}
\end{flalign*}
and therefore we have the Cholesky factorization \eqref{eq:Cholesky for J}.
To prove  \eqref{eq: UL for TJ} we observe that the Lax equations \eqref{lax} lead to
\begin{align*}
  T_a(\n_a\cdot\boldsymbol J-q_a)\omega_a=&\omega_a(\n_a\cdot\boldsymbol J-q_a)\\
  =&\omega_a M_a\omega_a
\end{align*}
which imply the result.
\end{proof}
Thus, for each given direction these translations  reproduce the behaviour of the classical elementary Darboux transformations which imply the interchange or intertwining  of
the lower triangular  and upper triangular factors in the Gauss--Borel decomposition. Later on we will discuss the explicit form of these Darboux transformations for  the MVOPR, quasi-tau and $\beta$ matrices.

Before we derived  \eqref{rho1} and \eqref{H1} where we expressed $\rho_{a,[k]}$, \eqref{rho2}, \eqref{H2} and $\alpha_{a,[k]}$ in terms of $H$ and $\beta$ matrices and its discrete time translations. Now, we show an alternative form of writing these functions, with no discrete time translations involved,  in terms of quasi-determinants of truncated Jacobi  matrices.

\begin{theorem}\label{quasideterminant alpha}
In terms of  quasi-determinants of the Jacobi matrices we have the following formul{\ae}
  \begin{align*}
    \rho_{a,[k]}&=(\n_a\cdot\boldsymbol J_{[k],[k-1]})\big(\Theta_*(\n_a\cdot\boldsymbol J^{[k]}-q_a\I^{[k]})\big)^{-1},\\
\alpha_{a,[k]}&=\Theta_*(\n_a\cdot\boldsymbol J^{[k+1]}-q_a\I^{[k+1]}).
  \end{align*}
\end{theorem}
\begin{proof}
  It is a direct consequence of Proposition \ref{prepa} and Theorem 3.4 of \cite{olver}.
\end{proof}

In particular we deduce that
  \begin{align*}
    \rho_{a,[k]}&=H_{[k]}(\n_a\cdot\boldsymbol\Lambda_{[k-1],[k]})^\top H_{[k-1]}\big(\Theta_*(\n_a\cdot\boldsymbol J^{[k]}-q_a\I^{[k]})\big)^{-1}.
  \end{align*}

Next, we are ready to collect the integrable system structure for these discrete flows giving the classical elements: linear systems, Lax and Zakharov--Shabat equations in its discrete version.

\begin{pro}\label{pro5}
  \begin{enumerate}
    \item For each $a\in\{1,\dots,D\}$ the wave matrices $W_1$ and $W_2$  are solutions of the following linear system
    \begin{align}\label{eq: linear W}
      T_aW=&\omega_a W.
    \end{align}
    \item The Jacobi type matrices $\n_a\cdot\boldsymbol J$, $a\in\{1,\dots,D\}$, satisfy the discrete Lax equations
    \begin{align}\label{lax}
      T_b(\n_a\cdot\boldsymbol J)\omega_b=&\omega_b\,(\n_a\cdot\boldsymbol J),& M_b\,T_b(\n_a\cdot\boldsymbol J)&=(\n_a\cdot\boldsymbol J)M_b,& & \forall a,b\in\{1,\dots,D\}.
    \end{align}

    \item The lattice resolvent $\omega_a$ and the $M_a$, $a\in\{1,\dots,D\}$, are subject to the discrete Zakharov--Shabat
    equations
    \begin{align}\label{ZS}
      (T_a\omega_b)\omega_a=&(T_b\omega_a)\omega_b, & M_a(T_aM_b)&=M_b(T_bM_a),& &\forall a,b\in\{1,\dots,D\}.
    \end{align}
       \end{enumerate}
\end{pro}
\begin{proof}
 See Appendix \ref{proof5}.
\end{proof}

\subsection{Quasi-tau functions formul{\ae} for MOVPR}
\subsubsection{Expressing the MVOPR in terms of quasi-tau matrices}\label{qt}
We discuss here certain expressions for the MVOPR, in terms of the  quasi-tau matrices $H_{[k]}$, that extend to the multidimensional situation the $\tau$ type formul{\ae} of the 1-D scenario.
\begin{pro}
 The MVOPR satisfy
  \begin{align}
  \rho_{a,[k]}(T_aP)_{[k-1]}+  (T_aP)_{[k]}&=P_{[k]},\label{nose}\\
  \label{SnS}
\alpha_{a,[k-1]}P_{[k-1]}+ (\n_a\cdot\boldsymbol \Lambda)_{[k-1],[k]} P_{[k]}&=  (\n_a\cdot\x-q_a) (T_aP)_{[k-1]},
 \end{align}
\end{pro}
\begin{proof}
  It is just a consequence of
$ P=M_a(T_aP)$ and $\omega_a P= (\n_a\cdot\x-q_a) (T_aP)$.
\end{proof}
 \begin{pro}\label{importante}
 When $\boldsymbol p\in\pi_a^+$, i.e. $\n_a\cdot\boldsymbol p=q_a$, the following relation holds
 \begin{align}\label{eq:relation}
(\n_a\cdot\boldsymbol \Lambda)_{[k-1],[k]}P_{[k]}(\boldsymbol p)=
 -\alpha_{a,[k-1]} P_{[k-1]}(\boldsymbol p).
 \end{align}
\end{pro}
\begin{proof}
Set $\n_a\cdot\x=q_a$ in \eqref{SnS}.
\end{proof}
\begin{definition}
Together with   $\boldsymbol q\coloneq (q_1,\dots,q_d)$ we consider the two following rectangular matrices
\begin{align*}
 [N\boldsymbol\Lambda]_{k}&\coloneq \PARENS{\begin{matrix}
  (\n_1\cdot\boldsymbol\Lambda)_{[k],[k+1]} \\  \vdots \\(\n_D\cdot\boldsymbol\Lambda)_{[k],[k+1]}
 \end{matrix}} \in \R^{D|[k]|\times |[k+1]|}, &
 [\boldsymbol{ T}^{(\boldsymbol q)} H]_k&\coloneq \PARENS{\begin{matrix}
  T_1^{(\boldsymbol q)}H_{[k]} \\ \vdots \\   T_D^{(\boldsymbol q)}H_{[k]}
                                        \end{matrix}}\in \R^{D|[k]|\times |[k]|}.
 \end{align*}
\end{definition}
Observe that,  putting together as rows the blocks $ (\n_a\cdot\boldsymbol\Lambda)_{[k],[k+1]} $, for $a\in\{1,\dots, D\}$, we get  a full column rank matrix $[N\boldsymbol\Lambda]_{k}$. Hence,  the correlation matrix $[N\boldsymbol\Lambda]_{k}^\top[N\boldsymbol\Lambda]_{k}\in\R^{|[k+1]|\times |[k+1]|}$ is invertible and the pseudo-inverse is
\begin{align*}
[N\boldsymbol\Lambda]_{k}^{+}=\big([N\boldsymbol\Lambda]_{k}^\top[N\boldsymbol\Lambda]_{k}\big)^{-1}[N\boldsymbol\Lambda]_{k}^\top,
\end{align*}
which happens to be  a left inverse,  see Appendix \ref{appendix pseudoinverse}.

We are now ready for
 \begin{theorem}\label{theorem:tauMVOPR}
The MVOPR can be expressed in terms of  quasi-tau matrices $H$  and its discrete time translations as follows
\begin{align*}
 P_{[k]}({\boldsymbol q})=(-1)^k
[N\boldsymbol\Lambda]_{k-1}^{+}[\boldsymbol {T}^{(N\boldsymbol q)} H]_{k-1} (H_{[k-1]})^{-1}[N\boldsymbol\Lambda]_{k-2}^{+}[\boldsymbol {T}^{(N\boldsymbol q)} H]_{k-2} (H_{[k-2]})^{-1}\cdots[N\boldsymbol\Lambda]_{0}^{+}[\boldsymbol {T}^{(N\boldsymbol q)} H]_{0} H_{[0]}^{-1}.
\end{align*}
 \end{theorem}
 \begin{proof}
Proposition \ref{importante} takes an interesting form if we choose $\boldsymbol p\in\cap_{a=1}^D\pi_a^+$; i.e., $\boldsymbol n_a\cdot \boldsymbol p=q_a$, which means that $N\boldsymbol p=\boldsymbol q$ and as $N$ is invertible we have $\boldsymbol p=N^{-1}\boldsymbol q$. In this case  \eqref{eq:relation} takes the form
\begin{align*}
[N\boldsymbol\Lambda]_{k-1}P_{[k]}(N^{-1}\boldsymbol q)=
 -[\boldsymbol {T}^{(\boldsymbol q)}H]_{k-1} H_{[k-1]}^{-1} P_{[k-1]}(N^{-1}\boldsymbol q).
 \end{align*}
 But, since $[N\boldsymbol \Lambda]_{k-1}$ has full column rank $k$ we
 can take its Moore--Penrose pseudoinverse $[N\boldsymbol\Lambda]_{k-1}^{+}$ to get
 \begin{align*}
P_{[k]}(N^{-1}\boldsymbol q)=
 -[N\boldsymbol\Lambda]_{(k-1)}^+[\boldsymbol {T}^{(\boldsymbol q)}H]_{k-1} H_{[k-1]}^{-1} P_{[k-1]}(N^{-1}\boldsymbol q).
 \end{align*}
Iterating this relation we get the desired result.
 \end{proof}

 In the one dimensional case we only have one component and the block matrices are just  numbers and we should replace $(\n_a\cdot\boldsymbol\Lambda)_{[k-1],[k]}$ by 1.
Thus, for $D=1$ the Theorem \eqref{theorem:tauMVOPR} gives
\begin{align*}
 P_{k}(q)=(-1)^kT H_{k-1}(H_{k-1})^{-1}T H_{k-2} (H_{k-2})^{-1}\cdots T H_{0} H_{0}^{-1},
\end{align*}
so that
\begin{align*}
 P_{k}(q)=&(-1)^k\frac{T H_{k-1}T H_{k-2} \cdots T H_{0}}{H_{k-1}H_{k-2}\cdots H_{0}}
 \\=&(-1)^k\frac{T\tau_k}{\tau_k}, & \tau_k\coloneq \det G^{[k-1]}=H_{k-1}H_{k-2}\cdots H_{0}.
\end{align*}
This is  a well known expression in terms of Miwa shifts of $\tau$-functions, see \S \ref{miwa}, where the $\tau$-function is the determinant of the OPRL  moment matrix

\subsubsection{Quasi-tau matrix expressions for the second kind functions}
\begin{pro}\label{pro6}
The second kind functions satisfy
   \begin{align}
  \rho_{a,[k]}(T_aC)_{[k-1]}+  (T_aC)_{[k]}&=(\n_a\cdot\x-q_a)C_{[k]}-\n_a\cdot{\boldsymbol {\widehat C}}_{[k]},\label{eq: MC}\\
\alpha_{a,[k]}C_{[k]} +(\n_a\cdot\boldsymbol\Lambda)_{[k],[k+1]} C_{[k+1]}&=  (T_a C)_{[k]}.\label{second}
 \end{align}
\end{pro}
\begin{proof}
See Appendix \ref{proof6}.
\end{proof}

From now on in this subsection we take $N=\I_D$.
\begin{definition}\label{T}
  Let us introduce the composed or total translation
  \begin{align*}
   T&=\prod_{a=1}^DT_{a}, & M&\coloneq S\big(T S\big)^{-1}.
  \end{align*}
\end{definition}
Observe that $P=M(TP)$. From  Cholesky factorization \eqref{cholesky} we get
\begin{pro}The following relation holds
\begin{align}\label{eq:alternativeMD}
M&= (H)\big(S^{-1}\big)^{\top}\Big[\prod_{a=1}^D(\Lambda_a-q_a)\Big]^\top(TS)^{\top}(TH)^{-1}.
\end{align}
\end{pro}

This lower unitriangular banded (with $D$ subdiagonals) block matrix can be decomposed into the product of $D$  lower unitriangular with only the  diagonal and first subdiagonal different from zero
\begin{pro}
    For each permutation $\sigma$ in the symmetric group $\mathfrak S_D$ the following decomposition holds true
  \begin{align}\label{factorM-}
    M=M_{\sigma(1)}\big(T_{\sigma 1}M_{\sigma 2}\big)\cdots  \big(T_{\sigma 1}\cdots T_{\sigma(D-1)}M_{\sigma D}\big).
  \end{align}
\end{pro}
\begin{proof}
  From Definition \ref{def:Ma} for each factor in the RHS of \eqref{factorM-} we have
    \begin{align*}
    M_{\sigma 1} &=S\big(T_{\sigma 1}S\big)^{-1},\\
    T_{\sigma 1}M_{\sigma 2}&=\big(T_{\sigma 1}S\big)\big(T_{\sigma 1}T_{\sigma 2}S\big)^{-1},\\   \shortvdotswithin{=}
T_{\sigma(1)}\cdots T_{\sigma(D-1)}M_{\sigma(D)}&=
\big(T_{\sigma 1}\cdots T_{\sigma(D-1)}S\big)\big(T_{\sigma 1}\cdots T_{\sigma D}S\big)^{-1},
  \end{align*}
  and the result follows.
\end{proof}
\begin{definition}
  We introduce
\begin{align}\label{rhoij}
    \rho^{(a)}_{[k]}\coloneq &\big(\prod_{j=a}^DT_{j}^{-1}\big)\big(\rho_{a,[k]}\big)\in\R^{k\times(k-1)}.
\end{align}
\end{definition}
Matrices that can be expressed as
\begin{pro}
  We have
  \begin{align}\label{rhoij2}
    \rho^{(a)}_{[k]}\coloneq &\big(\prod_{j=a}^DT_{j}^{-1}H_{[k]}\big)
    (\Lambda_a)_{[k-1],[k]}^\top
    \big( \prod_{j=a+1}^DT_{j}^{-1}H_{[k-1]}\big) ^{-1}.
\end{align}
\end{pro}
\begin{proof}
  It follows immediately when we substitute \eqref{rho1} into \eqref{rhoij}.
\end{proof}

For second kind functions, the analogous result to Theorem \ref{theorem:tauMVOPR} is
\begin{theorem}\label{teo1}
  The second kind functions can be written as
  \begin{align*}
  C_{[k]}(\boldsymbol q)=(-1)^{k+D} \smashoperator{\sum_{1\leq a_1\leq\cdots \leq a_k\leq D }}\rho^{(a_k)}_{[k]}\cdots \rho^{(a_1)}_{[1]}T^{-1}H_{[0]}
\end{align*}
  or in terms of quasi-tau matrices and its inverse translations  as follows
  \begin{multline}\label{C}
C_{[k]}(\boldsymbol q)= (-1)^{k+D} \!\!\!\!\!\!\!\!\!\!\sum_{1\leq a_1\leq\cdots \leq a_k\leq D }
(\prod_{j=a_k}^DT_j^{-1}H_{[k]})(\Lambda_{a_k})_{[k-1],[k]}^\top(\prod_{j=a_k+1}^DT_{j}^{-1}H_{[k-1]})^{-1}
\cdots\\\times
(\prod_{j=a_1}^DT_j^{-1}H_{[1]})(\Lambda_{a_1})_{[0],[1]}^\top(\prod_{j=a_1+1}^DT_j^{-1}H_{[0]})^{-1} T^{-1}H_{[0]}.
  \end{multline}
\end{theorem}
\begin{proof}
See Appendix \ref{proof7}.
\end{proof}
Observe that for $D=1$ the formula \eqref{C} reads
  \begin{align*}
C_{[k]}(q)= &(-1)^{k+1} (T^{-1}H_{[k]})H_{[k-1]})^{-1}
\cdots
(T^{-1}H_{[1]})H_{[0]})^{-1} T^{-1}H_{[0]}\\
=& (-1)^{k+1} \frac{T^{-1}(H_{[k]}\cdots H_{[0]})}{H_{[k-1]}\cdots H_{[0]}}\\
=& (-1)^{k+1} \frac{T^{-1}\tau_k}{\tau_{k-1}},
  \end{align*}
for $ \tau_k=\det G^{[k]}$, which is the well known formula the well-known formula that expresses  the adjoint Baker functions in terms of $\tau$-functions and Miwa shifts.
For $D=2$ we have $T=T_1T_2$ and $(T^{-1}M)^{-1}=(T_2^{-1}M_2)^{-1} (T^{-1}M_1)^{-1}$
  \begin{gather*}
C_{[k]}(\boldsymbol q)= (-1)^{k}
\Big(T_2^{-1}H_{[k]})(\Lambda_{1})_{[k-1],[k]}^\top(H_{[k-1]})^{-1}
\cdots
(T_2^{-1}H_{[1]})(\Lambda_1)_{[0],[1]}^\top(H_{[0]})^{-1} \\+
(T_2^{-1}H_{[k]})(\Lambda_{1})_{[k-1],[k]}^\top(H_{[k-1]})^{-1}
\cdots
(T_2^{-1}H_{[2]})(\Lambda_1)_{[1],[2]}^\top(H_{[1]})^{-1} (T^{-1}H_{[1]})(\Lambda_2)_{[0],[1]}^\top (T_2^{-1}H_{[0]})^{-1}\\+
(T_2^{-1}H_{[k]})(\Lambda_{1})_{[k-1],[k]}^\top(H_{[k-1]})^{-1}
\cdots
(T^{-1}H_{[2]})(\Lambda_2)_{[1],[2]}^\top(T_2^{-1}H_{[1]})^{-1} (T^{-1}H_{[1]})(\Lambda_2)_{[0],[1]}^\top (T_2^{-1}H_{[0]})^{-1}\\+\begin{multlined}[t]
(T^{-1}H_{[k]})(\Lambda_{2})_{[k-1],[k]}^\top(T_2^{-1}H_{[k-1]})^{-1}\\
\cdots
(T^{-1}H_{[2]})(\Lambda_2)_{[1],[2]}^\top(T_2^{-1}H_{[1]})^{-1} (T^{-1}H_{[1]})(\Lambda_2)_{[0],[1]}^\top (T_2^{-1}H_{[0]})^{-1}\Big)T^{-1}H_{[0]}.\end{multlined}
  \end{gather*}

\subsection{Transforming the Christoffel--Darboux kernels and kernel polynomials}

We give here some relations among translated and non translated Christoffel--Darboux kernels; we begin with the following result for the kernels $K^{(\ell)}(\x,\y)$ and the MOVPR.
\begin{theorem}\label{translatedCD}
The translated and non translated Christoffel--Darboux  kernels are connected by
   \begin{align}\label{kerrSnS}
 K^{(\ell)}(\x,\y)&=(\n_a\cdot\x-q_a)(T_aK)^{(\ell-1)}(\x,\y)+
P_{[\ell-1]}(\x)^{\top}(H_{[\ell-1]})^{-1} (T_aP)_{[\ell-1]}(\y).
\end{align}
 \end{theorem}
 \begin{proof}
See Appendix \ref{prooftranslatedCD}.
 \end{proof}

For the second kind kernels $Q^{(\ell)}(\x,\y) $ and second kind functions $C(\x)$we have
  \begin{pro}
The transformed and initial  second kind Christoffel--Darboux kernels are connected by
\begin{align}\label{kerrQSnS}
 (T_{a}Q)^{(\ell)}(\x,\y)-(T_{a}Q)^{(\ell)}(\x,\y_a)&=(y_a-q_a)Q^{(\ell)}(\x,\y)\\
 \nonumber &+\left[ C_{[\ell]}(\x) \right ]^{\top}\left[(\Lambda_a)_{[\ell-1][\ell]} \right]^\top (T_{a}H)_{[\ell-1]}^{-1} \left[ (T_{a}C)_{[\ell-1]}(\y)-(T_{a}C)_{[\ell-1]}(\y_a)\right]
\end{align}
 Moreover, these kernels fulfill
  \begin{align*}
  &Q^{(\ell)}(\x,\y)+(T_{a}Q)^{(\ell)}(\x_a,\y)-(T_{a}Q)^{(\ell)}(\x,\y_a)=\\
  &=\frac{\left[(\Lambda_a)_{[\ell-1],[\ell]}C_{[\ell]}(\x)\right]^\top\left[\left(H_{[\ell-1]} \right)^{-1}[C_{[\ell-1]}(\y)]-\left((T_{a}H)_{[\ell-1]} \right)^{-1}[(T_{a}C)_{[\ell-1]}(\y_a)]\right]}{x_a-y_a}\\
  &-\frac{\left[\left(H_{[\ell-1]} \right)^{-1} C_{[\ell-1]}(\x)-\left((T_{a}H)_{[\ell-1]} \right)^{-1} (T_aC)_{[\ell-1]}(\x_a)\right]^{\top}\left[(\Lambda_a)_{[\ell-1],[\ell]}C_{[\ell]}(\y) \right]}{x_a-y_a}.
 \end{align*}
 \end{pro}
 \begin{proof}
  Apply $T_{a}$ to
  \begin{align}
    \frac{(T_{a}C)(\x)-(T_{a}C)(\x_a)}{x_a-q_a}&=  \big(M_a\big)^{-1}C(\x)
  \end{align} and use \eqref{kerQQ} in
  \begin{align*}
 Q^{(\ell)}(\x,\y) &=\frac{\left[(\Lambda_a)_{[\ell-1],[\ell]}C_{[\ell]}(\x)\right]^\top\left(H_{[\ell-1]} \right)^{-1}\left[C_{[\ell-1]}(\y)\right]-  \left[C_{[\ell-1]}(\x)\right]^\top \left(H_{[\ell-1]} \right)^{-1}\left[(\Lambda_a)_{[\ell-1],[\ell]}C_{[\ell]}(\y) \right]}{x_a-y_a}\\
               \nonumber     &+\frac{\left[-M_a (T_{a}C)^{[\ell]}(\x_a) \right]^{\top}\left(H^{[\ell]}\right)^{-1}C^{[\ell]}(\y)-\left[C^{[\ell]}(\x)\right]^{\top}\left(H^{[\ell]}\right)^{-1}[-M_a (T_aC)^{[\ell]}(\y_a)]}{x_a-y_a}.
  \end{align*}
but now we can let the $M_a$ act on the $C^{[\ell]}$ instead of on the $(T_{a}C)^{[\ell]}$ and this way obtain the result.
 \end{proof}

\subsection{Elementary Darboux transformations and the sample matrix trick}\label{s.darboux}
Darboux transformations were introduced in  \cite{darboux}  in the context of the Sturm--Liouville theory and since them have been applied in several problems. It was in \cite{matveev}, a paper where explicit solutions of the Toda lattice where found, where this covariant transformation was given the name of \emph{Darboux}. It has been used in the 1D realm of orthogonal polynomials quite successfully, see for example \cite{yoon, bueno-marcellan1,bueno-marcellan2, marcellan}.
In Geometry  the theory of transformations of surfaces preserving some given properties conforms a classical subject, in the list of such transformations given in the classical treatise by Einsehart \cite{eisenhart} we find the Levy transformation, which later on was named as elementary Darboux transformation and known in the orthogonal polynomials context as Christoffel transformation \cite{yoon, szego}; in this paper we have denoted it by $T$. The adjoint elementary Darboux or adjoint Levy transformation $T^{-1}$ is also relevant  \cite{matveev,dsm} and is referred some times as a Geronimus transformation \cite{yoon}, and in the notation of this paper corresponds to $T^{-1}$. For further information see \cite{rogers-schief,gu}. In order to  extend it to the multivariate realm let us recall  some basic facts about the 1D case and then extend it to an arbitrary  number of dimensions.

 \subsubsection{The 1-D context. Elementary Darboux transform} For $D=1$ \eqref{eq:relation} reads
  \begin{align*}
    P_k(q)=-\alpha_k P_{k-1}(q)
  \end{align*}
  and as we are dealing with numbers we deduce
\begin{align*}
  \alpha_k=-\frac{P_k(q)}{P_{k-1}(q)}
\end{align*}
that reintroduced in the $D=1$ version of \eqref{SnS} gives, the so called kernel polynomials \cite{yoon},
\begin{align}\label{darboux1D}
TP_{k-1}(x)=\frac{P_k(x)P_{k-1}(q)-P_k(q)P_{k-1}(x)}{x-q}\frac{1}{P_{k-1}(q)}=K^{(k)}(x,q)\frac{H_k}{P_{k-1}(q)}
\end{align}
which is the standard elementary Darboux transformation for the OPRL.  From $\alpha_k=(T H_k)H_k^{-1}$ we get
\begin{align*}
(T_kH_k)P_{k-1}(q)=-P_k(q)H_k,
\end{align*}
Notice that we can recover this relation directly from the $D=1$ version of \eqref{kerrSnS} evaluated at $x=q$
\begin{align*}
 K^{(\ell)}(q,y)=- P_{\ell}(q)(TH_{\ell-1}) ^{-1}   (TP)_{\ell-1}(y).
\end{align*}
That is, according \eqref{darboux1D}, the transformed polynomials are intimately related to the Christoffel--Darboux kernel; this motivates that the polynomials $TP_k$ are sometimes known as kernel polynomials \cite{yoon}.

\subsubsection{The multivariate elementary Darboux transformation}
A nice property of \eqref{darboux1D} is that the Darboux transformed OPRL are expressed explicitly in terms of objects related to the OPRL associated with the original measure. This is apparently lost in the multivariate situation as, despite equation \eqref{SnS} gives new MVOPR associated with the shifted measure $T_a\dd\mu_{\m}$ in terms of the MVOPR for the measure $\dd\mu_{\m}$, now the equivalent relation \eqref{eq:relation}  does not allow to express $\alpha_{a,[k]}$ in terms of MVOPR for the original measure solely. We could use Theorem \ref{quasideterminant alpha} which involves no translations, however it is expressed in terms of quasi-determinants of the Jacobi type matrix and not in terms of the MOVPR. We will show a way to overcome this problem.

We begin with the elementary multivariate Darboux transformation associated with $\n\in\R^D$ and $q\in\R$. For that aim we are going to describe what we call the sample matrix trick.
\begin{definition}\label{def:sample}
Given  the set $\{\boldsymbol{p}_{1},\dots \boldsymbol{p}_{|[k]|}\}\subset\pi^+=\{\x\in\R^D:\x\cdot\n=q\}\subset\mathbb{R}^D$, whose elements are known as \emph{nodes},  we consider
the \emph{sample matrices}
\begin{align*}
    \Sigma^k_{[\ell]}= \PARENS{\begin{matrix}P_{[\ell]}(\boldsymbol p_{1})&\dots&P_{[\ell]}(\boldsymbol p_{|[k]|})\end{matrix}}\in \R^{|[\ell]|\times |[k]|}.
 \end{align*}
The set $\{\boldsymbol{p}_{1},\dots \boldsymbol{p}_{|[k]|}\}$ of nodes is said to be a \emph{poised set}  for the \emph{interpolation polynomials}  $\{P_{\kk_a}\}_{a=1}^{|[k]|}$ if  the sample matrix $\Sigma^k_{[k]}$  is invertible, i.e. $\det \Sigma^k_{[k]}\neq 0$.
\end{definition}
We now consider the transformation generated by the discrete flow $T\dd\mu(x)=(\n\cdot\x-q)\dd\mu(x)$. An important observation is that the matrix $\alpha_{[k]}$ can be expressed in terms of sample matrices of MOVPR, this is the \emph{sample matrix trick},
 \begin{pro}For a poised set $\{\boldsymbol{p}_{1},\dots, \boldsymbol{p}_{|[k]|}\}\subset\pi^+\subset\mathbb{R}^D$ of nodes we can write
   \begin{align}\label{alphaMVOPR}
      \alpha_{[k]}=-(\n\cdot\boldsymbol \Lambda)_{[k],[k+1]}  \Sigma^k_{[k+1]}\big(  \Sigma^k_{[k]}\big)^{-1}.
   \end{align}
 \end{pro}
 \begin{proof}
   From \eqref{eq:relation} we get
 \begin{align*}
\alpha_{[k]} P_{[k]}(\boldsymbol p_{i})&=-(\n\cdot\boldsymbol \Lambda)_{[k],[k+1]}P_{[k+1]}(\boldsymbol p_{i}), & i&=1,\dots,|[k]|,
  \end{align*}
  or
  \begin{align*}
    \alpha_{[k]} \PARENS{\begin{matrix}P_{[k]}(\boldsymbol p_{1}) & \dots &P_{[k]}(\boldsymbol p_{|[k]|})\end{matrix}}=-\PARENS{\begin{matrix}(\n\cdot\boldsymbol \Lambda)_{[k],[k+1]}P_{[k+1]}(\boldsymbol p_{1})&\dots&(\n\cdot\boldsymbol \Lambda)_{[k],[k+1]}P_{[k+1]}(\boldsymbol p_{|[k]|})\end{matrix}},
  \end{align*}
 and as  we are dealing with a poised set of nodes we get the result.
 \end{proof}
Hence, we have a set of nodes $\{\boldsymbol{p}_{1},\dots, \boldsymbol{p}_{|[k]|}\}\subset\pi^+\subset\mathbb{R}^D$ and a set of interpolation data, $-(\n\cdot\boldsymbol \Lambda)_{[k],[k+1]}  \Sigma^k_{[k+1]}$, so that the linear combination  $\phi(\x)=\alpha_{[k]} P_{[k]}(\x)$, the interpolation function,  passes through the interpolation points;  i.e., $\phi(\boldsymbol p_j)=-(\n\cdot\boldsymbol \Lambda)_{[k],[k+1]}P_{[k+1]}(\boldsymbol p_{j})$.

Now, we are ready to give the elementary multivariate Darboux transformations for MVOPR
\begin{theorem}\label{theorem:Darboux}
  Given a poised set $\{\boldsymbol{p}_{1},\dots \boldsymbol{p}_{|[k]|}\}\subset\pi^+\subset\mathbb{R}^D$ of nodes we have the following expressions of the elementary Darboux transformed MVOPR, the kernel polynomials $TP(\x)$ associated with
  $(\n\cdot \x-q)\dd\mu(\x)$, in terms of quasi-determinants of the original MVOPR
  \begin{align}\label{mvopr-darboux}
      (TP)_{[k]}(\x)=(\n\cdot\x-q) ^{-1}(\n\cdot\boldsymbol \Lambda)_{[k],[k+1]}\Theta_*\PARENS{\begin{matrix}
  \Sigma^k_{[k]} & P_{[k]}(\x)\\
 \Sigma^k_{[k+1]} &P_{[k+1]}(\x)\end{matrix}}.
     \end{align}
  For the second kind functions analogous relations hold
    \begin{align}\label{second-darboux}
      (TC)_{[k]}(\x)=(\n\cdot\boldsymbol \Lambda)_{[k],[k+1]}\Theta_*\PARENS{\begin{matrix}
  \Sigma^k_{[k]} & C_{[k]}(\x)\\
 \Sigma^k_{[k+1]} &C_{[k+1]}(\x)\end{matrix}}.
   \end{align}
   \end{theorem}
\begin{proof}
  To prove \eqref{mvopr-darboux} introduce in \eqref{SnS} the expressions given in  \eqref{alphaMVOPR} to get the kernel polynomials
    \begin{align}
    \label{SnS2}
    \begin{aligned}
      (TP)_{[k]}(\x)=&(\n\cdot\x-q) ^{-1}(\n\cdot\boldsymbol \Lambda)_{[k],[k+1]}\big(P_{[k+1]}(\x)- \Sigma^k_{[k+1]}\big(  \Sigma^k_{[k]}\big)^{-1}P_{[k]}(\x)\big)
\\=&(\n\cdot\x-q) ^{-1}(\n\cdot\boldsymbol \Lambda)_{[k],[k+1]}\operatorname{SC}\PARENS{\begin{matrix}
  \Sigma^k_{[k]} & P_{[k]}(\x)\\
 \Sigma^k_{[k+1]} &P_{[k+1]}(\x)\end{matrix}}.
    \end{aligned}
  \end{align}
  from where the results follows.
  For \eqref{second-darboux} we  recall \eqref{second} and use \eqref{alphaMVOPR}.
\end{proof}
We remark that we are using the notation of \cite{olver} for the quasi-determinants, and in fact in this case the lower right corner is a $|[k]|$-th dimensional vector, not even a square matrix, therefore we are dealing with an extended quasi-determinant with some of elements not in a ring.
Notice that this result extends to the multivariate situation the well known 1D situation described by \eqref{darboux1D}.
For the transformed quasi-tau functions $H$'s and the coefficients $\beta$ we have
\begin{pro}When the conditions specified in Theorem \ref{theorem:Darboux} are satisfied the elementary Darboux transformations of the matrices $H_{[k]}$ and $\beta_{[k]}$ are given by
 the following quasi-determinantal formul{\ae}
   \begin{align*}
     (TH)_{[k]}=(\n\cdot\boldsymbol \Lambda)_{[k],[k+1]}\Theta_*\PARENS{\begin{matrix}
  \Sigma^k_{[k]} & H_{[k]}\\
 \Sigma^k_{[k+1]} &0\end{matrix}}.
   \end{align*}
   and we have the relation
   \begin{align*}
    ( T\beta)_{[k]}(\n\cdot\boldsymbol \Lambda)_{[k-1],[k]}= q+(\n\cdot\boldsymbol \Lambda)_{[k],[k+1]}\Theta_*\PARENS{\begin{matrix}
  \Sigma^k_{[k]} & \I_{[k]}\\
 \Sigma^k_{[k+1]} &\beta_{[k+1]}\end{matrix}}.
   \end{align*}
\end{pro}
\begin{proof}
  The first relation is a immediate consequence \eqref{alphaMVOPR} and \eqref{H1} so that
  \begin{align*}
(TH)_{[k]}=-(\n\cdot\boldsymbol \Lambda)_{[k],[k+1]}  \Sigma^k_{[k+1]}\big(  \Sigma^k_{[k]}\big)^{-1}H_{[k]}.
  \end{align*}
  From \eqref{alphaMVOPR} and \eqref{H2}  we get
  \begin{align*}
   (T\beta)_{[k]}( \n\cdot\boldsymbol\Lambda)_{[k-1],[k]}-
 ( \n_a\cdot\boldsymbol\Lambda)_{[k],[k+1]}\beta_{[k+1]}- q=-(\n\cdot\boldsymbol \Lambda)_{[k],[k+1]}  \Sigma^k_{[k+1]}\big(  \Sigma^k_{[k]}\big)^{-1}
  \end{align*}
  that is
    \begin{align*}
   (T\beta)_{[k]}( \n\cdot\boldsymbol\Lambda)_{[k-1],[k]}=q+(\n\cdot\boldsymbol \Lambda)_{[k],[k+1]} \big(\beta_{[k+1]}-( \Sigma^k_{[k+1]}\big(  \Sigma^k_{[k]}\big)^{-1}\big).
  \end{align*}
\end{proof}
Observe that we can  diagrammatically write \emph{\`{a} la Gel'fand}
 \begin{align*}
 (TH)_{[k]}=&\begin{vmatrix}
  \Sigma^k_{[k]} & H_{[k]}\\
 (\n\cdot\boldsymbol \Lambda)_{[k],[k+1]}\Sigma^k_{[k+1]} &\boxed{0}\end{vmatrix},\\
    ( T\beta)_{[k]}(\n\cdot\boldsymbol \Lambda)_{[k-1],[k]}=& q+\begin{vmatrix}
  \Sigma^k_{[k]} & \I_{[k]}\\
 (\n\cdot\boldsymbol \Lambda)_{[k],[k+1]}\Sigma^k_{[k+1]} &\boxed{(\n\cdot\boldsymbol \Lambda)_{[k],[k+1]}\beta_{[k+1]}}\end{vmatrix}.
  \end{align*}

We can give a more explicit expression for $\beta$ using the Moore--Penrose pseudo-inverse. In fact, using Proposition \ref{moorepenroseright} and the multinomial matrix $\mathcal M_{[k]}$ \eqref{multinomialmatrix} we can write
  \begin{align*}
    ( T\beta)_{[k]}
 =&\begin{multlined}[t]q\mathcal M_{[k]}^{-1/2}\big((\n\cdot\boldsymbol \Lambda)_{[k-1],[k]}\mathcal M_{[k]}^{-1/2}\big)^+\\+\begin{vmatrix}
  \Sigma^k_{[k]} &\mathcal M_{[k]}^{-1/2}\big((\n\cdot\boldsymbol \Lambda)_{[k-1],[k]}\mathcal M_{[k]}^{-1/2}\big)^+\\
(\n\cdot\boldsymbol \Lambda)_{[k],[k+1]} \Sigma^k_{[k+1]} &\boxed{(\n\cdot\boldsymbol \Lambda)_{[k],[k+1]}\beta_{[k+1]}\mathcal M_{[k]}^{-1/2}\big((\n\cdot\boldsymbol \Lambda)_{[k-1],[k]}\mathcal M_{[k]}^{-1/2}\big)^+}\end{vmatrix}.
\end{multlined}
   \end{align*}

\subsection{Multivariate Christoffel  formula and quasi-determinants}
For $D=1$ there is a well known formula for the orthogonal polynomials $\{q_n(x)\}$ associated to a measure of the form $c(x-q_1)\cdots(x-q_m)\dd\mu(x)$ in terms of the orthogonal polynomials $\{p_n(x)\}$ of the measure $\dd\mu(x)$, see \S 2.5 of \cite{szego}, as
\begin{align*}
 q_n(x)=\frac{1}{c(x-q_1)\cdots(x-q_m)}\begin{vmatrix}
   p_n(x)&\dots&p_{n+m}(x)\\
    p_n(q_1)&\dots&p_{n+m}(q_1)\\
    \vdots &&\vdots\\
     p_n(q_l)&\dots&p_{n+m}(q_l)
 \end{vmatrix}.
\end{align*}
The Hungarian Mathematician Gabor Szeg\H{o}, who considers the proof very easy, points out that it was proven for $\dd\mu=\dd x$ by Elwin Bruno Christoffel in \cite{christoffel}. This fact was rediscovered in the Toda context, see for example the formula (5.1.11) in \cite{matveev} for $W^+_n(N)$.

In this section we will construct an analogous to the  Christoffel formula  in the multivariate context. We use the sample matrix trick and quasi-determinants.

\subsubsection{Iterating two elementary Darboux transformations}
First, for a better understanding we discuss the iteration of two elementary  Darboux transformations
\begin{align*}
  \dd\mu(\x)\rightarrow (\n^{(1)}\cdot\x-q^{(1)})
  \dd\mu(\x)\rightarrow (\n^{(2)}\cdot\x- q^{(2)}) (\n^{(1)}\cdot\x-q^{(1)})\dd\mu(\x)
\end{align*}
or, equivalently,  $\dd\mu\to T^{(1)}T^{(2)}\dd\mu$.

Given the corresponding lattice resolvents
\begin{align} \label{omegas}
  \omega^{(a)}=&(T^{(a)}S)( \n^{(a)}\cdot\boldsymbol\Lambda-q^{(a)})S^{-1}, & a&\in\{1,2\},
  \end{align}
we introduce
\begin{definition}
The second iterated resolvent matrix is
\begin{align}\label{resolvent2}
  \omega\coloneq \big(T^{(2)}\omega^{(1)}\big)\omega^{(2)}.
\end{align}
\end{definition}
A first result regarding the two step Darboux transformation is
\begin{pro}
The MOVPR satisfy
  \begin{align}\label{mdarboux}
(\n^{(2)}\cdot\x-q^{(2)})(\n^{(1)}\cdot\x-q^{(1)})(T^{(2)}T^{(1)}P)(\x)=\omega P(\x).
\end{align}
\end{pro}
\begin{proof}
It is a consequence of $(\n^{(a)}\cdot\x-q^{(a)})T^{(a)}P=\omega^{(a)} P$ with $a\in\{1,2\}$.
\end{proof}
Regarding the matrix structure of $\omega$ if we define
\begin{align*}
  \n\coloneq q^{(1)}\n^{(2)}+q^{(2)}\n^{(1)}\in\R^D
\end{align*}
  we quickly found that
\begin{pro}
The second iterated  resolvent $\omega$ decomposes in diagonals as follows
\begin{align}\label{eq:omega2}
\begin{aligned}
  \omega=&\underbracket{(\n^{(1)}\cdot\boldsymbol\Lambda)(\n^{(2)}\cdot\boldsymbol\Lambda)}_{\text{second superdiagonal}}\\&+
\underbracket{(T^{(1)}T^{(2)}\beta)(\n^{(1)}\cdot\boldsymbol\Lambda)(\n^{(2)}\cdot\boldsymbol\Lambda)-
(\n^{(1)}\cdot\boldsymbol\Lambda)(\n^{(2)}\cdot\boldsymbol\Lambda)\beta-\n\cdot\boldsymbol\Lambda
}_{\text{first superdiagonal}}\\&+
  \underbracket{(T^{(1)}T^{(2)}H)H^{-1}}_{\text{diagonal}}
  \end{aligned}
\end{align}
\end{pro}
\begin{proof}
From \eqref{omegas} and \eqref{resolvent2} we get
\begin{align*}
  \omega=(T^{(1)}T^{(2)}S)\big((\n^{(1)}\cdot\boldsymbol\Lambda)(\n^{(2)}\cdot\boldsymbol\Lambda)-\n\cdot\boldsymbol\Lambda+q^{(1)}q^{(2)}\big)S^{-1}
\end{align*}
and the two superdiagonal terms follow immediately. Now, from $\omega^{(a)}=\n^{(a)}\cdot\boldsymbol \Lambda+(T^{(a)}H)H^{-1}$ we get the diagonal part
$\big(T^{(2)}((T^{(1)}H)H^{-1})\big)(T^{(2)}H)H^{-1}$.
\end{proof}
Notice that the Zakharov--Shabat or compatibility equations \eqref{ZS}, which can be written as the symmetry condition $\omega=\big(T^{(2)}\omega^{(1)}\big)\omega^{(2)}=\big(T^{(1)}\omega^{(2)}\big)\omega^{(1)}$, are an immediate consequence of the previous result.
\begin{pro}\label{pro:omega2c} Relations \eqref{eq:omega2} can be written  componentwise  as follows
\begin{align*}
 (\omega)_{[k],[k+2]}=&\big((\n^{(1)}\cdot\boldsymbol\Lambda)(\n^{(2)}\cdot\boldsymbol\Lambda)\big)_{[k],[k+2]},\\
 (\omega)_{[k],[k+1]}=&
   (T^{(1)}
 T^{(2)}\beta_{[k]})
 \big((\n^{(1)}\cdot\boldsymbol\Lambda)(\n^{(2)}
 \cdot\boldsymbol\Lambda)\big)_{[k-1],[k+1]}-\big((\n^{(1)}\cdot\boldsymbol\Lambda)(\n^{(2)}\cdot\boldsymbol\Lambda)\big)_{[k],[k+2]}\beta_{[k+2]}
 -(\n\cdot\boldsymbol\Lambda)_{[k],[k+1]}
\\
 (\omega)_{[k],[k]}=&(T^{(1)}T^{(2)}H_{[k]})H_{[k]}^{-1}.
\end{align*}
\end{pro}

Again we need to look at certain   hyperplanes $\pi^{(a,+)}=\{\x\in\R^D: \n^{(a)}\cdot\x=q^{(a)}\}$,
for $a\in\{1,2\}$.
\begin{pro}
  For any $\boldsymbol p\in\pi^{1,+}\cup\pi^{2,+}$; i.e., either $\n^{(1)}\cdot\boldsymbol p=q^{(1)}$ or $\n^{(2)}\cdot\boldsymbol p=q^{(2)}$, we have
\begin{align}\label{eq:darboux.2step}
  \omega_{[k],[k+2]}P_{[k+2]}(\boldsymbol p)+
\omega_{[k],[k+1]}P_{[k+1]}(\boldsymbol p)+
  \omega_{[k],[k]}P_{[k]}(\boldsymbol p)=0.
\end{align}
\end{pro}
\begin{proof}
  It follows from \eqref{mdarboux}.
\end{proof}
We now employ the sample matrix trick  used for the elementary Darboux transformation to characterize $\omega$ in terms of MOVPR evaluated at some particular points. For the sets $\{\boldsymbol p^{(1)}_{j}\}_{j=1}^{|[k]|}, \{\boldsymbol { p}^{(2)}_{j}\}_{j=1}^{|[k+1]|}$ we use the notation introduced in Definition \ref{def:sample}, i.e. we use the matrices $\Sigma^{(1),k}_{[\ell]}$ for the first set of points and $\Sigma^{(2),k}_{[\ell]}$ for the second set.
\begin{pro}\label{pro:chistoffel2}
   Suppose that  $\big\{\boldsymbol p^{(1)}_{j}\big\}_{j=1}^{|[k]|}\cup \big\{\boldsymbol  p^{(2)}_{j}\big\}_{j=1}^{|[k+1]|}\subset \pi^{1,+}\cup\pi^{2,+}$ is a poised set for $\begin{psmallmatrix}
      P_{[k]}\\P_{[k+1]}
    \end{psmallmatrix}$, i.e. $\begin{vsmallmatrix}
    \Sigma^{(1),k}_{[k]} &\Sigma^{(2),k}_{[k+1]}\\
    \Sigma^{(1),k}_{[k+1]} &\Sigma^{(2),k+1}_{[k+1]}
    \end{vsmallmatrix}\neq 0$. Then
      \begin{align*}
    \PARENS{\begin{matrix}
  \omega_{[k],[k]} & \omega_{[k],[k+1]}
    \end{matrix}}=&-\omega_{[k],[k+2]}\PARENS{\begin{matrix}
     \Sigma^{(1),k}_{[k+2]} & \Sigma^{(2),{k+1}}_{[k+2]}
    \end{matrix}}\PARENS{\begin{matrix}
    \Sigma^{(1),k}_{[k]} &\Sigma^{(2),k+1}_{[k]}\\
    \Sigma^{(1),k}_{[k+1]} &\Sigma^{(2),k+1}_{[k+1]}
    \end{matrix}}^{-1}.
  \end{align*}
\end{pro}
\begin{proof}
  As said we proceed as in the proof of Proposition \ref{alphaMVOPR} and evaluate \eqref{eq:darboux.2step} in the set $\{\boldsymbol p^{(1)}_{j}\}_{j=1}^{|[k]|}\cup \{\boldsymbol  p^{(2)}_{j}\}_{j=1}^{|[k+1]|}$ to get
     \begin{align*}
    \PARENS{\begin{matrix}
    \omega_{[k],[k]} & \omega_{[k],[k+1]}
    \end{matrix}}\PARENS{\begin{matrix}
    \Sigma^{(1),k}_{[k]} &\Sigma^{(2),k+1}_{[k]}\\
    \Sigma^{(1),k}_{[k+1]} &\Sigma^{(2),k+1}_{[k+1]}
    \end{matrix}}=&-\omega_{[k],[k+2]}\PARENS{\begin{matrix}
     \Sigma^{(1),k}_{[k+2]} & \Sigma^{(2),k+1}_{[k+2]}
    \end{matrix}},
  \end{align*}
from where the result follows.
\end{proof}
\begin{theorem}
  For the composition of two elementary Darboux transformations, when the conditions required in Proposition \ref{pro:chistoffel2} hold,  we have the following multivariate quasi-determinantal Christoffel formula for the kernel polynomials
  \begin{align*}
  (T^{(2)}T^{(1)}P)_{[k]}(\x)=  \frac{\big((\n^{(1)}\cdot\boldsymbol\Lambda)(\n^{(2)}\cdot\boldsymbol\Lambda)\big)_{[k],[k+2]}}{(\n^{(2)}\cdot\x-q^{(2)})(\n^{(1)}\cdot\x-q^{(1)})}
  \Theta_*\PARENS{\begin{matrix}
    \Sigma^{(1),k}_{[k]} &\Sigma^{(2),k+1}_{[k]}&P_{[k]}(\x)\\
    \Sigma^{(1),k}_{[k+1]} &\Sigma^{(2),k+1}_{[k+1]}&P_{[k+1]}(\x)\\
        \Sigma^{(1),k}_{[k+2]} & \Sigma^{(2),{k+1}}_{[k+2]} &P_{[k+2]}(\x)
  \end{matrix}}.
  \end{align*}
\end{theorem}
\begin{proof}
From \eqref{mdarboux} we get
\begin{align*}
  (\n^{(2)}\cdot\x-q^{(2)})(\n^{(1)}\cdot\x-q^{(1)})(T^{(2)}T^{(1)}P)_{[k]}(\x)=\omega_{[k],[k+2]} P_{[k+2]}(\x)+\omega_{[k],[k+1]]}P_{[k+1]}(\x)+\omega_{[k],[k]}P_{[k]}(\x)\\
  =\omega_{[k],[k+2]} \Big(P_{[k+2]}(\x)-\PARENS{\begin{matrix}
     \Sigma^{(1),k}_{[k+2]} & \Sigma^{(2),{k+1}}_{[k+2]}
    \end{matrix}}\PARENS{\begin{matrix}
    \Sigma^{(1),k}_{[k]} &\Sigma^{(2),k}_{[k]}\\
    \Sigma^{(1),k}_{[k+1]} &\Sigma^{(2),k+1}_{[k+1]}
    \end{matrix}}^{-1}\PARENS{\begin{matrix}
      P_{[k]}(\x)\\
        P_{[k+1]}(\x)
    \end{matrix}}\Big)
\end{align*}
from where the result follows.
\end{proof}
\begin{pro}\label{betaH2step}
The quasi-tau matrices $H_{[k]}$ and the $\beta_{[k]}$ matrices transform for a 2-step elementary Darboux transformation according to the following quasi-determinantal formul{\ae}
  \begin{align*}
(T^{(1)}T^{(2)}H)_{[k]}=&
\big((\n^{(1)}\cdot\boldsymbol\Lambda)(\n^{(2)}\cdot\boldsymbol\Lambda)\big)_{[k],[k+2]}
 \begin{vmatrix}
    \Sigma^{(1),k}_{[k]} &\Sigma^{(2),k+1}_{[k]}&H_{[k]}\\
    \Sigma^{(1),k}_{[k+1]} &\Sigma^{(2),k+1}_{[k+1]}&0_{[k+1],[k]}\\
        \Sigma^{(1),k}_{[k+2]} & \Sigma^{(2),{k+1}}_{[k+2]} &\boxed{0_{[k+2],[k]}}
  \end{vmatrix},\\
 (T^{(1)}
 T^{(2)}\beta_{[k]})
 \big((\n^{(1)}\cdot\boldsymbol\Lambda)(\n^{(2)}
 \cdot\boldsymbol\Lambda)\big)_{[k-1],[k+1]}
=&
\begin{multlined}[t]
(\n\cdot\boldsymbol\Lambda)_{[k],[k+1]}\\+\big((\n^{(1)}\cdot\boldsymbol\Lambda)(\n^{(2)}\cdot\boldsymbol\Lambda)\big)_{[k],[k+2]}
\begin{vmatrix}
    \Sigma^{(1),k}_{[k]} &\Sigma^{(2),k+1}_{[k]}&  0_{[k],[k+1]}\\
   \Sigma^{(1),k}_{[k+1]} &\Sigma^{(2),k+1}_{[k+1]}&     \I_{[k+1]}\\
           \Sigma^{(1),k}_{[k+2]} & \Sigma^{(2),{k+1}}_{[k+2]} &\boxed{\beta_{[k+2]}}
  \end{vmatrix}.
  \end{multlined}
\end{align*}
\end{pro}
\begin{proof}
See Appendix \ref{proof beta H 2 step}.
\end{proof}
Observe that in this case we have used Gel'fand style instead of the Olver's notation for quasi-determinant.

\subsubsection{The general case: $m$ steps Darboux transformations}
We are now ready to consider the general case of $m$ iterated elementary Darboux transformations
\begin{align*}
  \dd\mu(\x)&\rightarrow \mathcal Q(\x) \dd\mu(\x), & \mathcal Q&\coloneq  \prod_{i=1}^m(\n^{(i)}\cdot\x- q^{(i)}),
\end{align*}
i.e.,  $\dd\mu\to Td\mu$, where $T\coloneq T^{(1)}\cdots T^{(m)}$ is the iteration of $m$ elementary Darboux transformations

In terms of the lattice resolvents
\begin{align*}
  \omega^{(i)}=&(T^{(i)}S)( \n^{(i)}\cdot\boldsymbol\Lambda-q^{(i)})S^{-1}, & i&\in\{1,\dots,m\}.
  \end{align*}
we introduce
\begin{definition}\label{def:miteratedresolvent}
  The   $m$-th iterated resolvent is
\begin{align*}
  \omega\coloneq \big(T^{(m)}\cdots T^{(2)}\omega^{(1)}\big)\big(T^{(m)}\cdots T^{(3)}\omega^{(2)}\Big)\cdots\omega^{(m)}.
\end{align*}
\end{definition}
From its definition we see that the
\begin{pro}
  The $m$-th iterated resolvent satisfies
  \begin{align}\label{omegamS}
   \omega= (TS)\Big(\prod_{i=1}^m(\n^{(i)}\cdot\boldsymbol\Lambda-q^{(i)})\Big)S^{-1}.
  \end{align}
\end{pro}

\begin{pro}
The $m$-th iterated resolvent $\omega$ can be expressed in diagonals as follows
\begin{align}\label{eq:omegam}
\begin{aligned}
  \omega=&\underbracket{\Big(\prod_{i=1}^m(\n^{(i)}\cdot\boldsymbol\Lambda)\Big)}_{\text{$m$-th superdiagonal}}\\&+
\underbracket{(T\beta)\Big(\prod_{i=1}^m(\n^{(i)}\cdot\boldsymbol\Lambda)\Big)-
\Big(\prod_{i=1}^m(\n^{(i)}\cdot\boldsymbol\Lambda)\Big)\beta-\sum_{i=1}^mq^{(i)}\prod_{j\neq i}(\n^{(j)}\cdot\boldsymbol\Lambda)
}_{\text{$(m-1)$-th superdiagonal}}\\&\shortvdotswithin{+}&+
  \underbracket{(TH)H^{-1}}_{\text{diagonal}}
\end{aligned}
\end{align}
\end{pro}
\begin{proof}
Observe that $\prod_{i=1}^m(\n^{(i)}\cdot\boldsymbol \Lambda -q^{(i)})$ splits into $m$ block superdiagonals. The $m$-th superdiagonal is $\prod_{i=1}^m(\n^{(i)}\cdot\boldsymbol \Lambda)$ while the $(m-1)$-th superdiagonal is given by $-\sum_{i=1}^mq^{(i)}\prod_{j\neq i}(\n^{(i)}\cdot\boldsymbol\Lambda)$. Then, applying \ref{omegamS}  we get the two higher superdiagonals of the $m$-iterated resolvent.
Now, from $\omega^{(i)}=\n^{(i)}\cdot\boldsymbol \Lambda+(T^{(i)}H)H^{-1}$ we get the diagonal part.
\end{proof}

The   components of the  $m$-th iterated resolvent $\omega$ are
\begin{align}\label{explicitresolvent}
\begin{aligned}
  \omega_{[k],[k+m]}=&\Big(\prod_{i=1}^m\big(\n^{(i)}\cdot\boldsymbol\Lambda\big)\Big)_{[k],[k+m]},\\
  \omega_{[k],[k+m-1]}=&\begin{multlined}[t](T\beta)_{[k]}\Big(\prod_{i=1}^m(\n^{(i)}\cdot\boldsymbol\Lambda)\Big)_{[k-1],[k+m-1]}-
\Big(\prod_{i=1}^m(\n^{(i)}\cdot\boldsymbol\Lambda)\Big)_{[k],[k+m]}\beta_{[k+m]}\\-\sum_{i=1}^mq^{(i)}\Big(\prod_{j\neq i}(\n^{(j)}\cdot\boldsymbol\Lambda)\Big)_{[k],[k+m-1]},
\end{multlined}\\
  \omega_{[k],[k]} =&(TH)_{[k]}H_{[k]}^{-1}.
\end{aligned}
\end{align}

\begin{pro}
The MOVPR and the second kind functions satisfy
  \begin{align}\label{mdarbouxm}
\mathcal Q(\x)TP(\x)=\omega P(\x),\\
  TC(\x)=\omega C(\x),\label{mdarbouxmC}
\end{align}
\end{pro}
\begin{proof}
It follows from $(\n^{(i)}\cdot\x-q^{(i)})T^{(i)}P(\x)=\omega^{(i)} P(\x)$ with $i\in\{1,\dots,m\}$ and $T^{(i)}C(\x)=\omega^{(i)} C(\x)$.
\end{proof}

We consider again  the hyperplanes $\pi^{(i,+)}=\{\x\in\R^D: \n^{(i)}\cdot\x=q^{(i)}\}$
for $i\in\{1,\dots,m\}$ to get
\begin{pro}
  For any $\boldsymbol p\in\cup_{i=1}^m \pi^{i,+}$ we have
\begin{align}\label{eq:darboux.mstep}
  \omega_{[k],[k+m]}P_{[k+m]}(\boldsymbol p)+
\omega_{[k],[k+1]}P_{[k+m-1]}(\boldsymbol p)+\cdots+
  \omega_{[k],[k]}P_{[k]}(\boldsymbol p)=0.
\end{align}
\end{pro}
\begin{proof}
  It follows from \eqref{mdarbouxm}.
\end{proof}
The sample matrix trick is used again  to characterize $\omega$ in terms of MOVPR evaluated at some particular points. For the sets $\big\{\boldsymbol p^{(i)}_{j}\big\}_{j=1}^{|[k+i-1]|}$, $i\in\{1,\dots,m\}$, we use the notation $\Sigma^{(i),k}_{[\ell]}$ introduced in Definition \ref{def:sample}
\begin{pro}\label{pro:chistoffelm}
   Suppose that  $\cup_{i=1}^m\{\boldsymbol p^{(i)}_{j}\}_{j=1}^{|[k-1+i]|}\subset \cup_{i=1}^m\pi^{(i),+}$ is a poised set for $\begin{psmallmatrix}
      P_{[k]}(\x)\\\vdots\\P_{[k+m-1]}(\x)\end{psmallmatrix}$, i.e.
    \begin{align*}
      \begin{vmatrix}
    \Sigma^{(1),k}_{[k]} &\dots &\Sigma^{(m),k+m-1}_{[k]}\\ \vdots&&\vdots\\
    \Sigma^{(1),k}_{[k+m-1]} &\dots&\Sigma^{(m),k+m-1}_{[k+m-1]}
    \end{vmatrix}\neq 0.
    \end{align*} Then
      \begin{align*}
    \PARENS{\begin{matrix}
  \omega_{[k],[k]} & \dots& \omega_{[k],[k+m-1]}
    \end{matrix}}=&-\omega_{[k],[k+m]}\PARENS{\begin{matrix}
     \Sigma^{(1),k}_{[k+m]} & \dots&\Sigma^{(m),{k+m-1}}_{[k+m]}
    \end{matrix}}\PARENS{\begin{matrix}
        \Sigma^{(1),k}_{[k]} &\dots &\Sigma^{(m),k+m-1}_{[k]}\\ \vdots&&\vdots\\
    \Sigma^{(1),k}_{[k+m-1]} &\dots&\Sigma^{(m),k+m-1}_{[k+m-1]}
    \end{matrix}}^{-1}
  \end{align*}
\end{pro}
\begin{proof}
  Evaluate \eqref{eq:darboux.mstep} in the set $\cup_{i=1}^m\{\boldsymbol p^{(i)}_{j}\}_{j=1}^{|[k-1+i]|}$ to get
          \begin{align*}
    \PARENS{\begin{matrix}
  \omega_{[k],[k]} & \dots& \omega_{[k],[k+m-1]}
    \end{matrix}}\PARENS{\begin{matrix}
        \Sigma^{(1),k}_{[k]} &\dots &\Sigma^{(m),k+m-1}_{[k]}\\ \vdots&&\vdots\\
    \Sigma^{(1),k}_{[k+m-1]} &\dots&\Sigma^{(m),k+m-1}_{[k+m-1]}
    \end{matrix}}=&-\omega_{[k],[k+m]}\PARENS{\begin{matrix}
     \Sigma^{(1),k}_{[k+m]} & \dots&\Sigma^{(m),{k+m-1}}_{[k+m]}
    \end{matrix}}
  \end{align*}
and  the desired result follows immediately.
\end{proof}
\begin{theorem}\label{teo2}
 If the conditions specified in Proposition \ref{pro:chistoffelm} are fulfill the following multivariate quasi-determinantal Christoffel formul{\ae} holds true
  \begin{align*}
  TP_{[k]}(\x)=  \frac{\big(\prod_{i=1}^m(\n^{(i)}\cdot\boldsymbol\Lambda)\big)_{[k],[k+m]}}{\mathcal Q(\x)}
  \Theta_*\PARENS{\begin{matrix}
    \Sigma^{(1),k}_{[k]} &\dots&\Sigma^{(m),k+m-1}_{[k]}&P_{[k]}(\x)\\
   \vdots & &\vdots&\vdots\\
        \Sigma^{(1),k}_{[k+m]} & \dots& \Sigma^{(m),{k+m-1}}_{[k+m]} &P_{[k+m]}(\x)
  \end{matrix}},\\
  TC_{[k]}(\x)=\big(\prod_{i=1}^m(\n^{(i)}\cdot\boldsymbol\Lambda)\big)_{[k],[k+m]}
  \Theta_*\PARENS{\begin{matrix}
    \Sigma^{(1),k}_{[k]} &\dots&\Sigma^{(m),k+m-1}_{[k]}&C_{[k]}(\x)\\
   \vdots & &\vdots&\vdots\\
        \Sigma^{(1),k}_{[k+m]} & \dots& \Sigma^{(m),{k+m-1}}_{[k+m]} &C_{[k+m]}(\x)
  \end{matrix}}.
  \end{align*}
\end{theorem}
\begin{proof}
See Appendix \ref{proof8}.
\end{proof}
In the scalar case, $D=1$, this formula in the OPRL context is known as Christoffel formula, see for example \cite{szego}.

\subsubsection{Quasi-determinantal expressions for the resolvent, quasi-tau matrices and $\beta$}
An interesting consequence of Proposition \ref{pro:chistoffelm} is the following
 \begin{pro}
The resolvent coefficients can be expressed as quasi-determinants as follows
  \begin{align*}
  \omega_{[k],[k+i]}=  \big(\prod_{j=1}^m(\n^{(j)}\cdot\boldsymbol\Lambda)\big)_{[k],[k+m]}
 \Theta_*\PARENS{\begin{matrix}
    \Sigma^{(1),k}_{[k]} &\dots&\Sigma^{(m),k+m-1}_{[k]}&0_{[k],[k+i]}\\
    \vdots &&\vdots&\vdots\\
     \Sigma^{(1),k}_{[k+i-1]} &\dots&\Sigma^{(m),k+m-1}_{[k+i-1]}&0_{[k+i-1],[k+i]}\\
       \Sigma^{(1),k}_{[k+i]} &\dots&\Sigma^{(m),k+m-1}_{[k+i]}&\I_{[k+i]}\\
       \Sigma^{(1),k}_{[k+i+1]} &\dots&\Sigma^{(m),k+m-1}_{[k+i+1]}&0_{[k+i+1],[k+i]}\\
       \vdots&&\vdots&\vdots
       \\
        \Sigma^{(1),k}_{[k+m]} & \dots& \Sigma^{(m),{k+m-1}}_{[k+m]} & 0_{[k+m],[k+i]}
  \end{matrix}}.
  \end{align*}
  \end{pro}
 \begin{proof}

 Assuming that  $\cup_{i=1}^m\{\boldsymbol p^{(i)}_{j}\}_{j=1}^{|[k-1+i]|}\subset \cup_{i=1}^m\pi^{(i),+}$ is a poised set for $\left(\begin{smallmatrix}
      P_{[k]}(\x)\\\vdots\\P_{[k+m-1]}(\x)
    \end{smallmatrix}\right)$, then for $i\in\{0,\dots,m-1\}$
         \begin{align*}
   \omega_{[k],[k+i]}=&-\omega_{[k],[k+m]}\PARENS{\begin{matrix}
     \Sigma^{(1),k}_{[k+m]} & \dots&\Sigma^{(m),{k+m-1}}_{[k+m]}
    \end{matrix}}\PARENS{\begin{matrix}
        \Sigma^{(1),k}_{[k]} &\dots &\Sigma^{(m),k+m-1}_{[k]}\\ \vdots&&\vdots\\
    \Sigma^{(1),k}_{[k+m-1]} &\dots&\Sigma^{(m),k+m-1}_{[k+m-1]}
    \end{matrix}}^{-1}\PARENS{\begin{matrix}
      0_{[k],[k+i]}\\
      \vdots\\
      0_{[k+i-1],[k+i]}\\
      \I_{[k+i]}\\
      0_{[k+i+1],[k+i]}\\
      \vdots\\
      0_{[k+m-1],[k+i]}
    \end{matrix}}.
  \end{align*}
 \end{proof}

The previous Proposition together with  \eqref{explicitresolvent} gives

 \begin{pro}\label{ostias}
   The $m$-th iteration of elementary Darboux transformations has the following effects on the quasi-tau matrices $H_{[k]}$
  \begin{align*}
    TH_{[k]}=&\big(\prod_{i=1}^m(\n^{(i)}\cdot\boldsymbol\Lambda)\big)_{[k],[k+m]}
  \Theta_*\PARENS{\begin{matrix}
    \Sigma^{(1),k}_{[k]} &\dots&\Sigma^{(m),k+m-1}_{[k]}&H_{[k]}\\
       \Sigma^{(1),k}_{[k+1]} &\dots&\Sigma^{(m),k+m-1}_{[k+1]}&0_{[k+1],[k]}\\
   \vdots & &\vdots&\vdots\\
        \Sigma^{(1),k}_{[k+m]} & \dots& \Sigma^{(m),{k+m-1}}_{[k+m]} &0_{[k+m],[k]},
  \end{matrix}},
  \end{align*}
  and on the matrices $\beta_{[k]}$
  \begin{multline*}
    (T\beta)_{[k]}\Big(\prod_{i=1}^m(\n^{(i)}\cdot\boldsymbol\Lambda)\Big)_{[k-1],[k+m-1]}=
    \sum_{i=1}^mq^{(i)}\Big(\prod_{j\neq i}(\n^{(j)}\cdot\boldsymbol\Lambda)\Big)_{[k],[k+m-1]}\\+\big(\prod_{i=1}^m(\n^{(i)}\cdot\boldsymbol\Lambda)\big)_{[k],[k+m]}
 \Theta_*\PARENS{\begin{matrix}
    \Sigma^{(1),k}_{[k]} &\dots&\Sigma^{(m),k+m-1}_{[k]}&0_{[k],[k+m-1]}\\
    \vdots &&\vdots&\vdots\\
     \Sigma^{(1),k}_{[k+m-2]} &\dots&\Sigma^{(m),k+m-1}_{[k+m-2]}&0_{[k+m-2],[k+m+1]}\\
       \Sigma^{(1),k}_{[k+m-1]} &\dots&\Sigma^{(m),k+m-1}_{[k+m-1]}&\I_{[k+m+1]}\\
        \Sigma^{(1),k}_{[k+m]} & \dots& \Sigma^{(m),{k+m-1}}_{[k+m]} &\beta_{[k+m]}
  \end{matrix}}.
  \end{multline*}
   \end{pro}

We remark that in Proposition \ref{betaH2step} and Proposition \ref{ostias} we can get explicitly the transformed $\beta$ by multiplying on the right by  right inverse matrices of the product of factors of type $\n\cdot\boldsymbol\Lambda$,  for that aim see Proposition \ref{moorepenroseright}. We obtain that a right inverse of $\Big(\prod_{i=1}^m(\n^{(i)}\cdot\boldsymbol\Lambda)\Big)_{[k-1],[k+m-1]}$
 can be expressed in terms of pseudo-inverses as
\begin{align*}
\prod_{i=0}^{m-1}\Big(\mathcal M^{-1/2}_{[k+i]}\big( (\n^{(i)}\cdot\boldsymbol\Lambda)_{[k-1+i],[k+i]}\mathcal M_{[k+i]}^{-1/2}\big)^+\Big),
\end{align*}
where we use the multinomial matrix \eqref{multinomialmatrix}.
Therefore
\begin{pro}
  After the iteration of $m$ elementary Darboux transformations the transformed coefficient $\beta$ can be expressed as a quasi-determinant as follows
   \begin{multline*}
T\beta_{[k]}=
    \sum_{i=1}^mq^{(i)}\Big(\prod_{j\neq i}(\n^{(j)}\cdot\boldsymbol\Lambda)\Big)_{[k],[k+m-1]}\prod_{i=0}^{m-1}\Big(\mathcal M^{-1/2}_{[k+i]}\big( (\n^{(i)}\cdot\boldsymbol\Lambda)_{[k-1+i],[k+i]}\mathcal M_{[k+i]}^{-1/2}\big)^+\Big)\\+\big(\prod_{i=1}^m(\n^{(i)}\cdot\boldsymbol\Lambda)\big)_{[k],[k+m]}
\begin{vmatrix}
    \Sigma^{(1),k}_{[k]} &\dots&\Sigma^{(m),k+m-1}_{[k]}&0_{[k],[k+m-1]}\\
    \vdots &&\vdots&\vdots\\
     \Sigma^{(1),k}_{[k+m-2]} &\dots&\Sigma^{(m),k+m-1}_{[k+m-2]}&0_{[k+m-2],[k+m+1]}\\
       \Sigma^{(1),k}_{[k+m-1]} &\dots&\Sigma^{(m),k+m-1}_{[k+m-1]}&\prod_{i=0}^{m-1}\Big(\mathcal M^{-1/2}_{[k+i]}\big( (\n^{(i)}\cdot\boldsymbol\Lambda)_{[k-1+i],[k+i]}\mathcal M_{[k+i]}^{-1/2}\big)^+\Big)\\[1pt]
        \Sigma^{(1),k}_{[k+m]} & \dots& \Sigma^{(m),{k+m-1}}_{[k+m]} &\boxed{\beta_{[k+m]}\prod_{i=0}^{m-1}\Big(\mathcal M^{-1/2}_{[k+i]}\big( (\n^{(i)}\cdot\boldsymbol\Lambda)_{[k-1+i],[k+i]}\mathcal M_{[k+i]}^{-1/2}\big)^+\Big)}
  \end{vmatrix}.
  \end{multline*}
\end{pro}
\subsubsection{Christoffel--Darboux kernel and the kernel polynomials}
The transformed polynomials $TP_k$ are known in the 1D case as kernel polynomials because the nice formula \eqref{darboux1D}.
We will show now that a similar relation holds in the multivariate situation, and thus the transformed MVOPR should receive the name of multivariate kernel polynomials in the same footing as it happens in the 1D scenario.

 Now we consider a result similar to Theorem \ref{translatedCD} but for $m$ elementary Darboux transformations.
 In doing so we need
 \begin{definition}
   We introduce the following truncation matrix of the resolvent matrix
   \begin{align*}
     \omega^{[\ell,m]}\coloneq  \PARENS{
  \begin{matrix}
    \omega_{[\ell],[\ell]}&\omega_{[\ell],[\ell+1]}&\dots &\omega_{[\ell],[\ell+m-2]}&\omega_{[\ell],[\ell+m-1]}\\
0_{[\ell+1],[\ell]}&\omega_{[\ell+1],[\ell+1]}&\dots&\omega_{[\ell+1],[\ell+m-2]} &\omega_{[\ell+1],[\ell+m-1]}\\
\vdots & &\ddots&\vdots&\vdots\\
0_{[\ell+m-2],[\ell]}&0_{[\ell+m-2],[\ell+1]}&\dots &\omega_{[\ell+m-2],[\ell+m-2]}&\omega_{[\ell+m-1],[\ell+m-1]}\\
0_{[\ell+m-1],[\ell]}&0_{[\ell+m-1],[\ell+1]}&\dots &0_{[\ell+m-1],[\ell+m-2]}&\omega_{[\ell+m-1],[\ell+m-1]}
 \end{matrix}}
   \end{align*}
   and the upper unitriangular matrix $\zeta^{[\ell,m]}\coloneq H^{[\ell]}\big(TH^{[\ell]}\big)^{-1}\omega^{[\ell,m]}$.
 \end{definition}
 Notice that $\zeta_{[\ell+i],[\ell+i]}^{[\ell,m]}=\I_{[\ell+i]}$, for $i\in\{0,\dots,m-1\}$ and that $\big(H^{[\ell]}\big)^{-1}\zeta^{[\ell,m]}=\big(TH^{[\ell]}\big)^{-1}\omega^{[\ell,m]}$.

 \begin{theorem}\label{superostia}
The following formula relating Christoffel--Darboux kernels after and before the iteration of $m$ elementary Darboux transformations holds true
\begin{align}\label{mCD}
  K^{(\ell+m)}(\x,\y)=\mathcal Q(\x)
  TK^{(\ell)}(\x,\y)+\PARENS{
  \begin{matrix}
TP_{[\ell]}(\y)\\\vdots\\TP_{[\ell+m-1]}(\y)
  \end{matrix}}^\top
\big(H^{[\ell]}\big)^{-1}\zeta^{[\ell,m]}
\PARENS{\begin{matrix}
    P_{[\ell]}(\x)\\\vdots\\P_{[\ell+m-1]}(\x)
  \end{matrix}}.
\end{align}
 \end{theorem}
 \begin{proof}
 See Appendix \ref{mCD mas}
    \end{proof}

For only one elementary Darboux transformation, $m=1$, the above result reduces to
\begin{align*}
  K^{(\ell+1)}(\x,\y)=&(\n\cdot\x-q)
  TK^{(\ell)}(\x,\y)+
TP_{[\ell]}(\y)^\top
H_{[\ell]}^{-1}P_{[\ell]}(\x),
\end{align*}
and we recover Theorem \ref{translatedCD}. For $m=2$, i.e., the two step Darboux transformation, we get
\begin{align*}
  K^{(\ell+2)}(\x,\y)=\mathcal Q(\x)
  TK^{(\ell)}(\x,\y)+
\PARENS{
  \begin{matrix}
H_{[\ell]}^{-1}TP_{[\ell]}(\y)\\H_{[\ell+1]}^{-1}TP_{[\ell+1]}(\y)
  \end{matrix}}^\top\zeta^{([\ell,2]}\PARENS{\begin{matrix}
    P_{[\ell]}(\x)\\P_{[\ell+1]}(\x)
  \end{matrix}}
\end{align*}
where
   \begin{align*}
   \mathcal Q&=(\n^{(1)}\cdot\x-q^{(1)})(\n^{(2)}\cdot\x-q^{(2)}), &
     \zeta^{[\ell,2]}&= \PARENS{
  \begin{matrix}\I_{[\ell]}
& \zeta_{[\ell],[\ell+1]}\\
0_{[\ell+1],[\ell]}& \I_{[\ell+1]}
 \end{matrix}}\end{align*}
   with
   \begin{align*}
     \zeta_{[\ell],[\ell+1]}=H_{[\ell]}(TH_{[\ell]})^{-1}\Big((T\beta)_{[\ell]}\Big(\prod_{i=1}^2(\n^{(i)}\cdot\boldsymbol\Lambda)\Big)_{[\ell-1],[\ell+1]}-
\Big(\prod_{i=1}^2(\n^{(i)}\cdot\boldsymbol\Lambda)\Big)_{[\ell],[\ell+2]}\beta_{[\ell+2]}-(\n\cdot\boldsymbol\Lambda)_{[\ell],[\ell+1]}\Big)
   \end{align*}
   and $\n=q^{(1)}\n^{(2)}+q^{(2)}\n^{(1)}$. Then,
   \begin{multline*}
  K^{(\ell+2)}(\x,\y)=\mathcal Q(\x)
  TK^{(\ell)}(\x,\y)+TP_{[\ell]}(\y)^\top H_{[\ell]}^{-1}P_{[\ell]}(\x) +TP_{[\ell+1]}(\y)^\top H_{[\ell+1]}^{-1}P_{[\ell+1]}(\x)
  \\+TP_{[\ell]}(\y)^\top H_{[\ell]}^{-1}\zeta_{[\ell],[\ell+1]}P_{[\ell+1]}(\x).
 \end{multline*}

 \begin{definition}
   Given the set of points $\mathscr P\coloneq \cup_{i=1}^m\{\boldsymbol p^{(i)}_{j}\}_{j=1}^{|[\ell-1+i]|}$ we define
   the following Christoffel--Darboux  vectors
    \begin{align*}
      \kappa^{(\ell,m)}(\y)\coloneq &\PARENS{\begin{matrix}
        \kappa^{(\ell,m)}_{[\ell]}(\y)\\\vdots\\   \kappa^{(\ell,m)}_{[\ell+m-1]}(\y)
      \end{matrix}}^\top, &
               \kappa^{(\ell,m)}_{[\ell+i-1]}(\y)\coloneq &\PARENS{\begin{matrix}
        K^{(\ell+m)}(\boldsymbol p^{(i)}_1,\y),\\\vdots\\ K^{(\ell+m)}(\boldsymbol p^{(i)}_{|[\ell+i-1]|},\y)
              \end{matrix}}^\top, &i=1,\dots,m.
    \end{align*}
    We also introduce
    \begin{align*}P^{(\ell,m)}(\x)&=\PARENS{
  \begin{matrix}
P_{[\ell]}(\x)\\\vdots\\P_{[\ell+m-1]}(\x)
  \end{matrix}}^\top,&
      \Sigma^{[\ell,m]}\coloneq \PARENS{\begin{matrix}
    \Sigma^{(1),\ell}_{[\ell]} &\dots &\Sigma^{(m),\ell+m-1}_{[\ell]}\\ \vdots&&\vdots\\
    \Sigma^{(1),\ell}_{[\ell+m-1]} &\dots&\Sigma^{(m),\ell+m-1}_{[\ell+m-1]}
    \end{matrix}}.
    \end{align*}
 \end{definition}

 Now, the sample matrix trick lead to the following finding that relates the Christoffel--Darboux kernel evaluated in a poised set and the transformed polynomials, justifying the denomination of kernel polynomials.
 \begin{pro}
   Assume that $\mathscr P\subset \cup_{i=1}^m \pi^{i,+}$ is a poised set, i.e., $\det\Sigma^{(\ell,m)}\neq 0$ ; then,
\begin{align*}
\kappa^{(\ell,m)}(\x)=(TP)^{(\ell,m)}(\x)
\big(H^{[\ell]}\big)^{-1}\zeta^{[\ell,m]} \Sigma^{[\ell,m]}.
\end{align*}
The following quasi-determinantal expressions  hold
\begin{align*}
TP^{(\ell,m)}(\x)=&-\Theta_*\PARENS{\begin{matrix}
 \zeta^{[\ell,m]} \Sigma^{[\ell,m]} & H^{[\ell]} \\
 \kappa^{(\ell,m)}(\x)&0
\end{matrix}}, &
\Theta_*\PARENS{\begin{matrix}
 \zeta^{[\ell,m]} \Sigma^{[\ell,m]} & H^{[\ell]} \\
 \kappa^{(\ell,m)}(\x)&TP^{(\ell,m)}(\x)
\end{matrix}}=0,
\end{align*}
 \end{pro}
One recognizes the first formula  as an extension to arbitrary dimensions and iterations of the 1D formula \eqref{darboux1D}.

Given another set of points $\tilde {\mathscr  P}\coloneq \cup_{i=1}^m\{{\boldsymbol {\tilde p}}^{(i)}_{j}\}_{j=1}^{|[\ell-1+i]|}$ which do not need to be neither in the union of hyper-planes nor poised we form the corresponding matrix $\tilde\Sigma^{[\ell,m]}$  so that
\begin{align*}
  (K(\boldsymbol p^{(i)}_j,\boldsymbol{\tilde p}^{(\tilde i)}_{\tilde j} ))=(T\tilde\Sigma^{[\ell,m]})\big(H^{[\ell]}\big)^{-1}\zeta^{[\ell,m]} \Sigma^{[\ell,m]}.
\end{align*}
We are expressing the evaluation of the Christoffel--Darboux kernel of the LHS as the product of matrices that at the end are expressed as elementary Darboux transforms, or Miwa shifts --see \S\ref{miwa}--, of the quasi-tau  matrices, see   \S \ref{qt}, as it happens in Theorem 8.1 in \cite{adler-van moerbeke 1}. It is not so clear if it is helpful for the computation of the Fredholm determinant $\det(1-\lambda S_\ell)$ of the projection $S_{\ell}$ with kernel $K^{(\ell+1)}(\x,\y)$ \cite{adler-van moerbeke 1}.

\section{On continuous Toda and MVOPR}\label{4}
Now, once we have consider the Toda type discrete flows and the corresponding moments matrices $G(\m)$ we are ready
 to add continuous deformations to the moment matrix. We will see that for given appropriate deformations or flows of a given measure we get an integrable hierarchy
 that extends the 2D Toda lattice hierarchy.
 In our extension  the dependent variables are size varying matrices which satisfy Toda type nonlinear PDE.
\subsection{The continuous flows}
We first introduce  of time deformations
 \begin{definition}
 Let us define the following covector of time variables
 \begin{align*}
   t=&(t_{[0]},t_{[1]},\dots),& t_{[k]}=&(t_{\q_1^{(k)}},\dots,t_{\q^{(k)}_{|[k]|}}), & t_{\q^{(k)}_j}\in\R.
 \end{align*}
 \end{definition}
 Observe that the just introduced times  can be considered as elements in the symmetric algebra $t^\top\in\operatorname{S}(\R^D)$.
\begin{definition}
 The deformation matrix is
 \begin{align*}
  W_0(t,\m)&=\exp\Big(\sum_{k=0}^{\infty}\sum_{j=1}^{|[k]|} t_{\q^{(k)}_j} \Lambda_{\q^{(k)}_j} \Big)\prod_{a=1}^D( \n_a\cdot\boldsymbol\Lambda-q_a)^{m_a},
 \end{align*}
and the deformed moment matrix is
  \begin{align*}
  G(t,\m)&\coloneq W_0(t,\m)G.
 \end{align*}
\end{definition}
Notice that $  G(t,\m)=G(\m)W_0(t,\m)^\top$, and
 \begin{align*}
 \Lambda_{\kk} G(t,\m)&= G(t,\m) \big(\Lambda_{\kk}\big)^\top,&  \forall &\kk\in\Z_+^D, & G(t,\m)&=(G(t,\m))^\top.
  \end{align*}
  \begin{definition}\label{lostiempos}
 We introduce the notation
 \begin{align*}
t(\x)\coloneq t\chi(\x)=\sum_{k=0}^\infty\sum_{j=1}^{|[k]|}t_{{\q^{(k)}_j}}\x^{{\q^{(k)}_j}}.
 \end{align*}
  \end{definition}

It is not difficult to see that
\begin{pro}\label{deformacion medida}
The deformed moment matrix is the moment matrix of the following deformed measure
\begin{align*}
 \dd \mu_{t,\m}(\x)&=\Exp{t(\x)}\dd \mu_{\m}(\x)=\Exp{t(\x)}\Big[\prod_{a=1}^D\big(\n_a\cdot\x-q_a\big)^{m_a}\Big]\dd\mu(\x).
\end{align*}
\end{pro}
The Cholesky factorization
 \begin{align}\label{eq:choleksy.evol}
  G(t,\m)&\coloneq (S(t,\m))^{-1}H(t,\m)\big((S(t,\m))^{-1}\big)^\top
 \end{align}
leads  to new MVOPR depending on both continuous and discrete time parameters. We introduce
\begin{definition}
 The wave semi-infinite matrices are
 \begin{align}\label{wave}
  W_1(t,\m)&\coloneq S(t,\m)W_0(t,\m), &  W_2(t,\m)&\coloneq H(t,\m)\big((S(t,\m))^{-1}\big)^\top.
 \end{align}
\end{definition}
An important fact regarding wave matrices and Gaussian decomposition of the evolved moment matrix is
\begin{pro}\label{pro:G_and_evolved_waves}
The  wave matrices  factorizes the non deformed moment matrix as follows
\begin{align}\label{totoW}
 G=&\big(W_1(t,\m)\big)^{-1}W_2(t,\m).
\end{align}
\end{pro}
\begin{proof}
From \eqref{eq:choleksy.evol} we deduce that\footnote{The product of two semi-infinite matrices is  a delicate issue. There is no problem if we multiply lower triangular with lower triangular, upper triangular with upper triangular and even lower triangular with upper triangular, as all the coefficients of the resulting matrix are finite sums. But the multiplication of upper triangular with lower triangular could lead to problems as sums are now infinite series that need not to converge. This is why $W_1(t)$ is well defined being $S(t)$ lower triangular and $W_0(t)$ upper triangular. However, for $(W_1)^{-1}$ we need to be more careful as the \emph{na\"{\i}ve} answer  $(W_1)^{-1}=W_0^{-1}S_1^{-1}$ involves  the product of an upper with a lower triangular. A possible answer is to say that the inverse from the right is  $S^{-1}H\big(S^{-1}\big)^\top (W_2(t))^{-1}$, which in fact is a consequence of this Proposition. Despite of being a \emph{formal} Proposition, as we are assuming the existence of the inverse $W_1^{-1}$ one could compute this inverse in appropriate  domains. That is the case for the  adjoint Baker functions.}
  \begin{align}\label{totoWW}
 G=&(W_0(t,\m)^{-1}G(t,\m)\\
 =&(W_0(t,\m)^{-1}(S(t,\m))^{-1}\,H(t,\m)\big((S(t,\m))^{-1}\big)^\top\\
 =&\big(W_1(t,\m)\big)^{-1}W_2(t,\m).
\end{align}
\end{proof}
In what follows we will use the splitting as a direct sum of the linear  semi-infinite matrices in strictly block lower triangular matrices and upper block triangular matrices. Then,  $M_+$ will denote the projection of $M$ in the upper triangular matrices while $M_-$ the projection in the strictly lower triangular matrices.
\begin{pro}
The wave matrices  $W_1$ and $W_2$  solve the following system of  linear differential equations
 \begin{align*}
  \frac{\partial W}{\partial t_{\q^{(k)}_j}}&=\big(J_{{\q^{(k)}_j}}\big)_{+} W,  & j&=1,\dots,|[k]|, k=0,1,\dots
 \end{align*}
\end{pro}
\begin{proof}
Differentiate \eqref{toto} to obtain
 \begin{align*}
\frac{\partial S(t)}{\partial t_{{\q^{(k)}_j}}}S(t)^{-1}&=-\big(J_{{\q^{(k)}_j}}(t)\big)_- ,&
\frac{\partial W_2(t)}{\partial t_{{\q^{(k)}_j}}}W_2(t)^{-1}&=\big(J_{{\q^{(k)}_j}}(t)\big)_+.\\
 \end{align*}
 and the result follows.
\end{proof}
Let us observe that
\begin{align*}
  \frac{\partial S(t)}{\partial t_{{\q^{(k)}_j}}}S(t)^{-1}+\big(J_{{\q^{(k)}_j}}(t)\big)_-=0
\end{align*}
is of particular relevance. Put for example $k=1$ and consider the equations for the times $t_{[1]}=(t_1,\dots,t_D)$, the \emph{first level times},
\begin{align}\label{diff jacobi}
  \frac{\partial S}{\partial t_a}S^{-1}+(J_a)_-=0.
\end{align}
\begin{definition}
  Let us decompose the matrices by diagonals, we write
\begin{align}
  \label{diagonal splitting}
  S=\I+\beta^{(1)}+\beta^{(2)}+\cdots
\end{align}
where $\beta^{(1)}$ is the first subdiagonal, i.e. $\beta=\beta^{(1)}$, and in general $\beta^{(k)}$ is the $k$-th subdiagonal of $S$.
\end{definition}
Then
\begin{pro}\label{otro mas}
The coefficients $S_{[k],[k-j]}=\beta^{(j)}_{[k]}$ of the MOVPR are subject to differential relations and the three first are
  \begin{align*}
  \frac{\partial\beta_{[k]}}{\partial t_a}=&J_{[k],[k-1]},\\
  \frac{\partial\beta^{(2)}_{[k]}}{\partial t_a}=&\frac{\partial\beta^{(1)}_{[k]}}{\partial t_a}\beta^{(1)}_{[k-1]},\\
  \frac{\partial\beta^{(3)}_{[k]}}{\partial t_a}= &\frac{\partial\beta^{(2)}_{[k]}}{\partial t_a}\beta^{(1)}_{[k-1]}
  +\frac{\partial\beta^{(1)}_{[k]}}{\partial t_a}\beta^{(2)}_{[k-1]}-\frac{\partial\beta^{(1)}_{[k]}}{\partial t_a}
  \beta^{(1)}_{[k-1]}    \beta^{(1)}_{[k-2]}
\end{align*}
\end{pro}
\begin{proof}
See Appendix \ref{yo que se}.
\end{proof}

\subsection{Baker functions. Lax and Zakharov--Shabat equations}
\begin{definition}\label{definition:Baker_functions}
Baker functions are defined by
\begin{align*}
  \Psi_1&\coloneq W_1\chi, & \Psi_2&\coloneq W_2\chi^*,
  \end{align*}
  while adjoint Baker functions are given by
  \begin{align*}
  \Psi_1^*&\coloneq (W_1^{-1})^\top\chi^*,& \Psi_2^*&\coloneq (W_2^{-1})^\top\chi.
\end{align*}
\end{definition}
We notice that $\Psi_1$ and $\Psi_2^*$ lead to the computation of finite sums, but $\Psi_1^*$ and $\Psi_2$ involves Laurent series; however $(\Psi_2)_{\q_i}=C_{\q_i}(t,\m)$ and its domain of convergence is $\Ds_{\q_i}(t,\m)$. We will denote by $\Ds_{\q_i}^*(t,\m)$ the domain of convergence of $(\Psi_1^*)_{\q_i}(t,\m)$.
\begin{pro}\label{pro.baker.expressions}
The following expressions for the Baker functions in terms of MVOPR and its multivariate Cauchy transforms hold true
  \begin{flalign*}
 (\Psi_1)_{\q_i}(\z)&=\Exp{ t(\z)}\Big[\prod_{a=1}^D\big(\n_a\cdot\z-q_a\big)^{m_a}\Big]P_{\q_i}(\z,t,\m),\\
  (\Psi_2)_{\q_i}(\z)&=\int_\Omega\frac{P_{\q_i}(\y,t)}{(z_1-y_1)\cdots(z_D-y_D)}\dd\mu_{t,\m}(\y),& \z\in\Ds_{\q_i}(t,\m)\setminus\operatorname{supp}(\dd\mu),\\
  (\Psi_1^*)_{\q_i}(\z)&=\sum_{j=1}^{|[k]|}(H(t,\m)^{-1})_{\q_i,\q_j}\int_\Omega\frac{P_{\q_j}(\y,t,\m)}{(z_1-y_1)\cdots(z_D-y_D)}
  \dd\mu(\y),&\z\in\cap_{j=1}^{|[k]|}\Ds_{\q_j}^*(t,\m)\setminus\operatorname{supp}(\dd\mu),\\
  (\Psi_2^*)_{\q_i}(\z)&=\sum_{j=1}^{|[k]|}(H(t,\m)^{-1})_{\q_i,\q_j}P_{\q_j}(\z,t,\m).
  \end{flalign*}
  \end{pro}
\begin{proof}
See Appendix \ref{proof9}.
  \end{proof}

  \begin{pro}
   The Baker functions and the adjoint Baker functions satisfy
   \begin{align*}
    J_a \Psi_1&=x_a\Psi_1 &  J_a \Psi_2,&=x_a\Psi_2-\lim_{x_a\to \infty} \left[x_a \Psi_2 \right], \\
    J_a^{\top} \Psi_1^{*}&=x_a\Psi_1^{*}-\lim_{x_a\to\infty} \left[x_a \Psi_1^{*} \right], & J_a^{\top} \Psi_2^{*}&=x_a\Psi_2^{*}.
   \end{align*}
  \end{pro}
\begin{pro}
 \begin{enumerate}
 \item
 The Baker functions are subject to the following linear system of differential equations
 \begin{align*}
   \frac{\partial \Psi_i}{\partial t_{\q_j}}=&\big(J_{\q^{(\ell)}_j}\big)_+\Psi_i, &
     \frac{\partial \Psi_i^*}{\partial t_{\q_j}}=&-\big(J_{\q^{(\ell)}_j}\big)_+^\top\Psi_i^*, & i&=1,2.
 \end{align*}
 \item The MVOPR and its second kinds functions satisfy
 \begin{align*}
      \frac{\partial P}{\partial t_{\q_j}}=&-\x^{\q_j}P+\big(J_{\q^{(\ell)}_j}\big)_+P, &  \frac{\partial C}{\partial t_{\q_j}}=&\big(J_{\q^{(\ell)}_j}\big)_+C.
 \end{align*}
 \item The following Lax equations hold
 \begin{align*}
\frac{\partial J_{\q_i^{(k)}}}{\partial t_{\q^{(\ell)}_j}}&=\big[\big(J_{\q^{(\ell)}_j}\big)_+,J_{\q^{(k)}_i}\big].
 \end{align*}
\item
The  Zakharov--Shabat type  equations
 \begin{align*}
\frac{\partial \big(J_{\q^{(k)}_i}\big)_+}{\partial t_{\q^{(\ell)}_j}}-
\frac{\partial \big(J_{\q^{(\ell)}_j}\big)_+}{\partial t_{\q_i^{(k)}}}+
\big[
\big(J_{\q^{(k)}_i}\big)_+,
\big(J_{\q^{(\ell)}_j}\big)_+
\big]&=0,\\
\frac{\partial \omega_a}{\partial t_{\q}}-(T_aJ_{\q})_+\omega_a+\omega_a(J_{\q})_+
&=0.
 \end{align*}
are fulfilled.
\end{enumerate}
\end{pro}

\subsection{Miwa shifts and discrete flows}\label{miwa}
We will reproduce a characteristic fact in integrable systems, the Miwa's coherent shift in the time variables lead to discrete flows and Darboux transformations. We now will indicate how these Miwa shifts are for this multivariate context.
The simplest case is perhaps the most interesting one as it reproduces the discrete flows we have considered previously. The  coherent shift in the times
\begin{align*}
  t&\to t'=t\pm [q]_{a}, & t'_{\q}=\ccases{
    t_{\q}, & \q\not\in \Z_+\ee_a,\\
    t_{m\ee_a}\pm \dfrac{1}{m q^m}, &\q=n\ee_a \text{ with $m\in\Z_+$}.
  }
\end{align*}
lead to the following deformation of the measure $\dd\mu_t$
\begin{align*}
 \dd \mu_t (\boldsymbol{x})&\longrightarrow \dd\mu_{t'}= \left(1-\frac{x_a}{q }\right)^{\mp 1} \dd \mu_t (\boldsymbol{x})=-q ^{\pm 1}\dd \left(T_a^{\mp}\mu_t(\boldsymbol{x})\right),
\end{align*}
which follows from
\begin{align*}
\log \left(\left( 1-\frac{x_a}{q}\right)^{\mp 1}\right)&=
\pm\sum_{m=1}^{\infty}\frac{\left(x_a\right)^m}{mq^m}.
\end{align*}
Given $q\in\C$ and $\n\in\R^D$, in order to recover $T$, we need the coherent shift given by
\begin{align*}
  [q]_{\n}=\Big(\frac{\n}{q},\frac{\n^{\odot 2}}{2q^2},\frac{\n^{\odot3}}{3q^3},\cdots\Big).
\end{align*}
In fact, considering $[q]_{\n}$ as a semi-infinite vector of time perturbations of the times variables $t$, we get
\begin{align*}
[q]_{\n}(\x)=\sum_{m=1}^\infty \frac{1}{m q^m}(\n\cdot\x)^m=-\log\Big(1-\frac{\n\cdot\x}{q}\Big).
\end{align*}
Consequently, for the shifted times $t'=t\pm[q]_{\n}$ we find that
\begin{align*}
   \exp(t'(\x)))&=\exp(t(\x)\pm[q]_{\n}(\x))=\exp(t(\x))\exp\Big(\mp\log\Big(1-\frac{\n\cdot\x}{q}\Big)\Big)\\&=\exp(t(\x))\exp\Big(\log\left(\Big(1-\frac{\n\cdot\x}{q}\Big)^{\mp1}\right)\Big)\\
   &=\Big(1-\frac{\n\cdot\x}{q}\Big)^{\mp 1}\exp(t(\x)),
\end{align*}
which immediately leads to the identification
\begin{align*}
  \dd\mu_{t\pm[q]_{\n}}(\x)&= \Big(1-\frac{\n\cdot\x}{q}\Big)^{\mp 1}\dd\mu_t(\x)\\
  &=-q^{\pm 1}\dd (T_{\n}^{\mp 1}\mu_t(\x)).
\end{align*}

\begin{pro}
  The Miwa shifts can be constructed as follows
   \begin{align*}
\big([q]_{\n}\big)_{[k]}=\frac{\n^{\odot k}}{kq^k}
=\frac{1}{k}\mathcal M_{[k]}\chi_{[k]}\Big(\frac{\n}{q}\Big),
 \end{align*}
 \end{pro}
\begin{proof}
Use Proposition \ref{chi-symmetric}.
\end{proof}
For each level of times we consider the corresponding nabla or gradient operators
\begin{align*}
  \nabla_{[k]}=\PARENS{\begin{matrix}
    \dfrac{\partial  }{\partial t_{\q^{(k)}_1}}\\\vdots\\\dfrac{\partial }{\partial t_{\q^{(k)}_{|[k]|}}}
  \end{matrix}}
\end{align*}
and also the normal derivatives
\begin{align*}
  \frac{\partial}{\partial \n^{\odot k }}\coloneq \langle \n^{\odot k},\nabla_{[k]}\rangle^{(k)}.
\end{align*}
Then, the Miwa shifts are modelled by the left factor of the vertex type operator
\begin{align*}
 \exp(t(\x)) \exp\Big(\sum_{k=1}^\infty\frac{1}{k q^n}  \frac{\partial}{\partial \n^{\odot k }}\Big).
\end{align*}

 \subsection{Bilinear equations}
We begin with the following observation
\begin{pro}\label{proposition:bilinear_wave}
  Wave matrices evaluated at different times $(t,\m)$ and $(t',\m')$ fulfill
\begin{align*}
  W_1(t,\m)\big(W_1(t',\m')\big)^{-1}=W_2(t,\m)\big(W_2(t',\m')\big)^{-1}.
  \end{align*}
\end{pro}
\begin{proof}
  From Proposition \ref{pro:G_and_evolved_waves} we  have
  \begin{align*}
    W_1(t,\m)G=&W_2(t,\m), &    W_1(t',\m')G=&W_2(t',\m'),
  \end{align*}
  for the same initial moment matrix $G$, from where the result follows immediately.
\end{proof}
\begin{lemma}\label{lemma:integrals_chi_chi_star}
  We have
  \begin{align*}
    \int_{\T^D(\boldsymbol r)} \chi(\boldsymbol z)\chi^*(\boldsymbol z)^\top\dd z_1\cdots\dd z_D=
    \int_{\T^D(\boldsymbol r)} \chi^*(\boldsymbol z)\chi(\boldsymbol z)^\top\dd z_1\cdots\dd z_D=(2\pi \operatorname{i})^D\mathbb I.
  \end{align*}
\end{lemma}
\begin{proof}
Observe that
\begin{align*}
 \chi(\chi^*)^\top&=\PARENS{\begin{matrix}
    Z_{[0],[0]} &Z_{[0],[1]}&\dots\\
     Z_{[1],[0]} &Z_{[1],[1]}&\dots\\
     \vdots & \vdots&
  \end{matrix}}, & Z_{[k],[\ell]}&\coloneq \frac{1}{z_1\cdots z_D}\PARENS{\begin{matrix}
    \z^{\kk_1-\ele_1} &  \z^{\kk_1-\ele_2}&\dots &   \z^{\kk_1-\ele_{|[\ell]|}}\\
    \z^{\kk_2-\ele_1} &  \z^{\kk_2-\ele_2}&\dots &   \z^{\kk_2-\ele_{|[\ell]|}}\\
    \vdots & \vdots &&\vdots\\
        \z^{\kk_{|[k]|}-\ele_1} &  \z^{\kk_{|[k]|}-\ele_2}&\dots &   \z^{\kk_{|[k]|}-\ele_{|[\ell]|}}
  \end{matrix}}.
\end{align*}
If we now integrate in the polydisk distinguished border  $\T^{D}(\boldsymbol r)$ using the Fubini theorem we factor each integral
in a product of $D$ factors, where the $i$-th factor is an integral over $z_i$ on the circle centered at origin of radius $r_i$. This is zero unless the integrand is $z_i^{-1}$ which occurs only in the principal diagonal.
\end{proof}
\begin{lemma}\label{lemma:U_V_integrals}
  Given two semi-infinite matrices $U$ and $V$ we have
  \begin{align*}
    UV=&\frac{1}{(2\pi \operatorname{i})^D}\int_{\T^D(\boldsymbol r)} U\chi(\boldsymbol z) \big(V^T\chi^*(\boldsymbol z)\big)^\top\dd z_1\cdots\dd z_D
    =\frac{1}{(2\pi \operatorname{i})^D}\int_{\T^D(\boldsymbol r)} U\chi^*(\boldsymbol z) \big(V^T\chi(\boldsymbol z)\big)^\top\dd z_1\cdots\dd z_D.
  \end{align*}
\end{lemma}
\begin{proof}
  Use Lemma \ref{lemma:integrals_chi_chi_star}.
\end{proof}
\begin{theorem}
  \label{proposition:Baker_bilinear_equation}
For any pair of times $(t,\m)$ and $(t',\m')$, points  $\boldsymbol r_1\in\Ds^*_{\q_j^{(\ell)}}(t',\m')$ and $\boldsymbol r_2\in\Ds_{\q^{(k)}_i}(t,\m)$ in the respective domains of convergence and $D$-dimensional tori  $\T^D(\boldsymbol r_1)$ and $\T^D(\boldsymbol r_2)$ (Shilov borders of polydisks) we can ensure that
  Baker and adjoint Baker functions satisfy the following bilinear identity
  \begin{align*}
    \int_{\T^D(\boldsymbol r_1)}( \Psi_1)_{\q_i^{(k)}}(\boldsymbol{z},t,\m)(\Psi_1^*)_{\q_j^{(\ell)}}(\boldsymbol{z},t',\m') \dd z_1\cdots\dd z_D=
       \int_{\T^D(\boldsymbol r_2)} ( \Psi_2)_{\q_i^{(k)}}(\boldsymbol{z},t,\m)(\Psi_2^*)_{\q_j^{(\ell)}}(\boldsymbol{z},t',\m') \dd z_1\cdots\dd z_D.
  \end{align*}
\end{theorem}
\begin{proof}
We give two different proofs
\begin{itemize}
  \item First proof:   Use Proposition \ref{proposition:bilinear_wave}, Lemma \ref{lemma:U_V_integrals}  and Definition \ref{definition:Baker_functions} to get the result.
  \item Second proof: For any  couple of set of times $(t,\m)$ and $(t',\m')$   and
      \begin{align*}
        \z\in\big(\Ds_{\q_i^{(k)}}(t,\m)\cap\Ds_{\q_j^{(\ell)}}^*(t',\m')\big)\setminus\operatorname{supp}(\dd\mu)
      \end{align*}
  to study
\begin{align*}
   \int_\Omega \dd\mu_t(\y) P_{\q^{(k)}_i}(\y,t,\m)P_{\q^{(\ell)}_j}(\y,t',\m') ,
    \end{align*}
   we can use the Fubini and the integral Cauchy formula --recalling that we are dealing with domains of holomorphy-- in each of the two factors to get
    \begin{multline*}
     \sum_{j'=1}^{|[\ell]|}(H^{-1}(t',\m'))_{\q^{(\ell)}_j,\q^{(\ell)}_{j'}}\int_{\T^D(\boldsymbol r_1)} \dd z_1\cdots\dd z_D\int_\Omega \dd\mu(\y) \frac{ P_{\q^{(k)}_i}(\z,t,\m)P_{\q^{(\ell)}_{j'}}(\y,t',\m')}{(z_1-y_1)\cdots(z_D-y_D)} \Exp{t(\z)}\Big[\prod_{a=1}^D\big((\n_a\cdot\z)-q_a\big)^{m_a}\Big]\\ =
 \sum_{j'=1}^{|[\ell]|}(H^{-1}(t',\m'))_{\q^{(\ell)}_j,\q^{(\ell)}_{j'}}\int_{\T^D(\boldsymbol r_2)}\dd z_1\cdots\dd z_D\int_\Omega \dd\mu(\y) \frac{P_{\q^{(k)}_i}(\y,t,\m)P_{\q^{(\ell)}_{j'}}(\z,t',\m')}{(z_1-y_1)\cdots(z_D-y_D)} \Exp{t(\y)}\Big[\prod_{a=1}^D\big((\n_a\cdot\y)-q_a\big)^{m_a}\Big],
  \end{multline*}
  from where the bilinear identify follows.
\end{itemize}
\end{proof}
\subsection{Toda type integrable equations}
We explore now  the nonlinear partial differential equations   satisfied by the quasi-tau matrices and the $\beta $ matrices.
\begin{pro}\label{pro10}
The following relations hold true
\begin{align}
  \label{todaHbeta1}\frac{\partial H_{[k]}}{\partial t_a}H_{[k]}^{-1}=&\beta_{[k]}(\Lambda_a)_{[k-1],[k]}-(\Lambda_a)_{[k],[k+1]}\beta_{[k+1]},\\
H_{[k+1]}\big[(\Lambda_a)_{[k],[k+1]}\big]^\top H_{[k]}^{-1}=&-\frac{\partial\beta_{[k+1]}}{\partial t_a}.\notag
\end{align}
\end{pro}
\begin{proof}
See Appendix \ref{proof10}
\end{proof}
For $k=0$ we have
\begin{align}\label{toda0}
  -(\Lambda_a)_{[0],[1]}\beta_{[1]}
=\frac{\partial H_{[0]}}{\partial t_a}H_{[0]}^{-1}.
\end{align}
Observe also that
\begin{align*}
 (J_a)_{[k+1],[k]}=-\frac{\partial\beta_{[k+1]}}{\partial t_a},
\end{align*}
which agrees with
\begin{align*}
 \left( H_{[k+1]}^{-1}(J_a)_{[k+1],[k]}H_{[k]}\right)^{\top}=(J_a)_{[k],[k+1]}.
\end{align*}

 \begin{theorem}\label{toda equations 2}
The quasi-tau matrices $H_{[k]}$ are subject to the following 2D Toda lattice type equations
\begin{align}\label{toda}
  \frac{\partial}{\partial t_b}\Big(\frac{\partial H_{[k]}}{\partial t_a}H_{[k]}^{-1}\Big)=(\Lambda_a)_{[k],[k+1]}H_{[k+1]}\big[(\Lambda_b)_{[k],[k+1]}\big]^\top H_{[k]}^{-1}
-H_{[k]}\big[(\Lambda_b)_{[k-1],[k]}\big]^\top H_{[k-1]}^{-1}(\Lambda_a)_{[k-1],[k]}.
\end{align}
which can be rewritten for the rectangular matrices $\beta_{[k]}$ as
\begin{align*}
  \frac{\partial^2\beta_{[k]}}{\partial t_a\partial t_b}&=
  -\frac{ \partial}{\partial t_a}\Big(\beta_{[k]}(\Lambda_b)_{[k-1],[k]}\beta_{[k]}\Big)+\frac{\partial\beta_{[k]}}{\partial t_a}\beta_{[k-1]}(\Lambda_b)_{[k-2],[k-1]}
+(\Lambda_b)_{[k],[k+1]}\beta_{[k+1]}\frac{\partial\beta_{[k]}}{\partial t_a}.
\end{align*}
We also have  the following partial differential-difference equations
\begin{multline*}
    \Delta_b\Big( \frac{\partial H_{[k]}}{\partial t_a}H_{[k]} ^{-1}\Big)=
    (\Lambda_a)_{[k],[k+1]}H_{[k+1]}[( \n_b\cdot\boldsymbol\Lambda)_{[k],[k+1]}]^\top \big( T_bH_{[k]}\big) ^{-1}
    \\ -H_{[k]}[( \n_b\cdot\boldsymbol\Lambda)_{[k-1],[k]}]^\top \big(T_bH_{[k-1]}\big) ^{-1}( \Lambda_a)_{[k-1],[k]},
  \end{multline*}
and
  \begin{multline*}
    \frac{\partial}{\partial t_a}\big( (\Delta_{b}H_{[k]})H_{[k]} ^{-1}\big)=
     ( \n_b\cdot\boldsymbol\Lambda)_{[k],[k+1]}H_{[k+1]}[( \Lambda_a)_{[k],[k+1]}]^\top (H_{[k]}) ^{-1}
    \\-(T_bH_{[k]})[( \Lambda_a)_{[k-1],[k]}]^\top \big( T_bH_{[k-1]}\big) ^{-1}( \n_b\cdot\boldsymbol\Lambda)_{[k-1],[k]}.
  \end{multline*}
 \end{theorem}
Notice that \eqref{toda} resembles the non Abelian Toda lattice discussed in \cite{matveev} Chapter 5 \S 3. However, in this case we have two main differences: in the first place the varying size of the matrices and in the second place we also have the \emph{connectors} $(\Lambda_a)_{[k],[k+1]}$,  $(\Lambda_a)_{[k-1],[k]}$ and its transpositions connecting different sized matrices.

Following the ideas of Theorem \ref{theorem:tauMVOPR}
\begin{definition}
We introduce
\begin{align*}
 [\Lambda]_{k}&\coloneq \PARENS{\begin{matrix}
  (\Lambda_1)_{[k],[k+1]} \\  \vdots \\(\Lambda_D)_{[k],[k+1]}
 \end{matrix}} \in \R^{D|[k]|\times |[k+1]|},\end{align*}
the gradient operator
 \begin{align*}
 [{ \nabla} H]_k&\coloneq \PARENS{\begin{matrix}
    \frac{\partial H_{[k]}}{\partial t_1}\\\vdots\\
     \frac{\partial H_{[k]}}{\partial t_D}
     \end{matrix}}
     \in\R^{D|[k]|\times|[k]|},
 \end{align*}
and
\begin{align*}
  \beta_{[k]}\otimes \I_D=\diag( \underbracket{ \beta_{[k]},\dots,  \beta_{[k]}}_{\text{$D$ times}})\in
   \R^{D|[k]|\times D|[k-1]|}.
\end{align*}
\end{definition}
The matrix $[\Lambda]_k$ has full column rank and therefore, see Appendix \ref{appendix pseudoinverse}, the correlation matrix $[\Lambda]_{k}^\top[\Lambda]_{k}\in\R^{|[k+1]|\times |[k+1]|}$ is invertible and the Moore--Penrose pseudo-inverse $[\Lambda]_k^{+}=\big([\Lambda]_{k}^\top[\Lambda]_{k}\big)^{-1}[\Lambda]_{k}^\top\in \R^{|[k+1]|\times D|[k]|}$ is the left inverse  $[\Lambda]_k^{+}[\Lambda]_k=\I_{[k+1]}$.
\begin{pro}
The $\beta$ matrices are subject to the following recurrence
  \begin{align*}
  \beta_{[k+1]}=&-[\Lambda]_k^+[\nabla H]_kH_{[k]}^{-1}+[\Lambda]_k^+(\beta_{[k]}\otimes\I_D)[\Lambda]_{k-1}, &
  \beta_{[1]}
=&-(\nabla H_{[0]})H_{[0]}^{-1}.
\end{align*}
\end{pro}
\begin{proof}
 Follows immediately from \eqref{todaHbeta1}, \eqref{toda0} and the fact that $[\Lambda]_0=\I_D$.
\end{proof}
Iterating once and twice the above result we get
\begin{align*}
  \beta_{[k+1]}  =&
  \begin{multlined}[t]
  -[\Lambda]_k^+[\nabla H]_kH_{[k]}^{-1}-[\Lambda]_k^+\Big(\Big[[\Lambda]_{k-1}^+[\nabla H]_{k-1}H_{[k-1]}^{-1}\Big]\otimes\I_D\Big)[\Lambda]_{k-1}\\+[\Lambda]_k^+\Big(\Big[[\Lambda]_{k-1}^+(\beta_{[k-1]}\otimes\I_D)[\Lambda]_{k-2} \Big]\otimes\I_D\Big)[\Lambda]_{k-1}
 \end{multlined}\\
 =&  \begin{multlined}[t] -[\Lambda]_k^+[\nabla H]_kH_{[k]}^{-1}-[\Lambda]_k^+\Big(\Big[[\Lambda]_{k-1}^+[\nabla H]_{k-1}H_{[k-1]}^{-1}\Big]\otimes\I_D\Big)[\Lambda]_{k-1}\\-[\Lambda]_k^+\Big(\Big[[\Lambda]_{k-1}^+
 \Big(\left[[\Lambda]_{k-2}^+[\nabla H]_{k-2}H_{[k-2]}^{-1}\right]\otimes\I_D\Big)[\Lambda]_{k-2} \Big]\otimes\I_D\Big)[\Lambda]_{k-1}\\+[\Lambda]_k^+\Big(\Big[[\Lambda]_{k-1}^+\Big(\Big[[\Lambda]_{k-2}^+(\beta_{[k-2]}\otimes\I_D)[\Lambda]_{k-3}\Big]\otimes\I_D\Big)[\Lambda]_{k-2} \Big]\otimes\I_D\Big)[\Lambda]_{k-1},
  \end{multlined}
  \end{align*}
respectively. By induction we deduce the following
\begin{pro}
 In terms of logarithmic right derivatives of the quasi-tau matrices  the $\beta$ matrices are expressed by
\begin{multline*}
   \beta_{[k+1]}=
    -[\Lambda]_k^+[\nabla H]_kH_{[k]}^{-1}-[\Lambda]_k^+\Big(\Big[[\Lambda]_{k-1}^+[\nabla H]_{k-1}H_{[k-1]}^{-1}\Big]\otimes\I_D\Big)[\Lambda]_{k-1}+\cdots
   \\-[\Lambda]_k^+\Big(\Big[[\Lambda]_{k-1}^+\Big(\Big[[\Lambda]_{k-2}^+\cdots \Big(\Big[[\Lambda]_1^+[\nabla H]_{0})H_{[0]}^{-1}\Big]\otimes\I_D\Big)[\Lambda]_1^+\cdots\Big]\otimes\I_D\Big)[\Lambda]_{k-2} \Big]\otimes\I_D\Big)[\Lambda]_{k-1}.
\end{multline*}
\end{pro}
In the 1D scenario the above formula simplifies and gives the classical $\tau$-expressions for $\beta_k$. Indeed, now
\begin{align*}
   \beta_{k+1}=-\frac{ \partial\log H_{k}}{\partial t_1}-\frac{ \partial\log H_{k-1}}{\partial t_1}-\cdots-\frac{ \partial\log H_{0}}{\partial t_1}
\end{align*}
and from $H_{k}=\frac{\tau_{k+1}}{\tau_k}$ we get a telescopical series giving
\begin{align*}
  \beta_{k+1}=-\frac{\partial\log\tau_{k+1}}{\partial t_1}.
\end{align*}

\section{KP type equations via congruences}\label{5}
We study how the previous construction lead to families of nonlinear partial differential-difference equations involving a fixed site, say the  $k$-th position,  in the lattice
and therefore not mixing several sites in the lattice ---notice that in the Toda type equations derived before, see Theorems \ref{toda equations 1} and  \ref{toda equations 2}, we are faced with relations involving three contiguous sites, $k-1$, $k$ and $k+1$.  We refer the reader to \cite{kac,mmm,mm,mma,tasaki-takabe}.
\subsection{The congruence technique}
Let us first introduce some notation
\begin{definition}\label{def:asymptotic-module}
Given two semi-infinite matrices $R_1(t,\m)$ and $R_2(t,\m)$ we say that
\begin{itemize}
 \item  $R_1(t)\in\mathfrak{l}W_0$ if $R_1(t)\big(W_0(t,\m)\big)^{-1}$ is a block strictly lower triangular matrix.
 \item  $R_2(t)\in\mathfrak{u}$ if it is a block upper triangular matrix.
\end{itemize}
\end{definition}
Then, we can state the following \emph{congruences} \cite{mm} or \emph{asymptotic module} \cite{jaulent} style result
\begin{pro}\label{pro:asymptotic-module}
Given two semi-infinite matrices $R_1(t,\m)$ and $R_2(t,\m)$ such that
\begin{itemize}
 \item  $R_1(t,\m)\in\mathfrak{l}W_0(t,\m)$,
 \item  $R_2(t,\m)\in\mathfrak{u}$,
 \item $R_1(t,\m)G=R_2(t,\m)$.
\end{itemize}
then
\begin{align*}
  R_1(t,\m)&=0, & R_2(t,\m)&=0.
\end{align*}
\end{pro}
\begin{proof}
Observe that
  \begin{align*}
    R_2(t,\m)=R_1(t,\m)G=R_1(t,\m)\big(W_1(t,\m)\big)^{-1}W_1(t,\m)G=R_1(t,\m)\big(W_1(t,\m)\big)^{-1}W_2(t,\m),
  \end{align*}
  where we have used \eqref{totoW}. From here we get
  \begin{align*}
  R_1(t,\m)\big(W_0(t,\m)\big)^{-1}\big(S(t,\m)\big)^{-1}=R_2(t,\m)\Big(H(t,\m)\big((S(t,\m))^{-1}\big)^\top\Big)^{-1},
  \end{align*}
  and, as in the LHS we have a strictly lower triangular matrix while in the RHS we have an upper triangular matrix, both sides must vanish and the result follows.
\end{proof}
We use the congruence notation
\begin{definition}
When $A-B\in\mathfrak{l}W_0$  we write $A=B+\mathfrak{l}W_0$ and if $A-B\in\mathfrak{u}$ we write $A=B+\mathfrak{u}$.
\end{definition}

We  introduce the following notation
\begin{align}
  \partial_a&=\frac{\partial}{\partial t_a}, &\partial_{(a,b)}&=\frac{\partial}{\partial t_{\ee_a+\ee_b}}, &\partial_{(a,b,c)}&=\frac{\partial}{\partial t_{\ee_a+\ee_b+\ee_c}}& a,b,c&=1,\dots, D,
\end{align}
and the normal derivative
\begin{align*}
\frac{\partial}{\partial \n_a}=\sum_{b=1}^Dn_{a,b}\partial_a.
\end{align*}
Notice that we have employed the round bracket notation for the subindexes of the higher times. In fact, this is convenient as its reflects the invariance under the action of the symmetric group on the letters in the labels, for example $t_{(a,b)}=t_{(b,a)}$.
\subsection{Connecting discrete a continuous flows}

To begun with let us show the following ``asymptotic" behaviours
\begin{pro}\label{pro:wave}
We have
\begin{align}\label{eq:partialW1}
\partial_bW_1=&(\Lambda_b+\beta \Lambda_b)W_0+\mathfrak{l}W_0,\\
T_aW_1=&\big((\n_a\cdot\boldsymbol\Lambda)+(T_a\beta) (\n_a\cdot\boldsymbol\Lambda)-q_a\big)W_0+\mathfrak{l}W_0.
\end{align}
\end{pro}
\begin{proof}
From   \eqref{wave} we get
  \begin{align*}
    \partial_bW_1=&(\partial_b S+ S\Lambda_b)W_0\\=&(\Lambda_b+\beta\Lambda_b)W_0+\mathfrak{l}W_0,\\
    T_aW_1=&(T_a S)(\n_a\cdot\boldsymbol\Lambda-q_a)W_0\\=&\big((\n_a\cdot\boldsymbol\Lambda)+(T_a\beta)(\n_a\cdot\boldsymbol\Lambda)-q_a\big)W_0+\mathfrak{l}W_0.
  \end{align*}
\end{proof}
This immediately  leads to the following finding,  translations and derivations are almost the same thing when acting on the Baker functions
\begin{pro}\label{pro: yo que se}
  The Baker functions $\Psi_1$, $\Psi_2$ are both solutions of the following difference-differential linear system
  \begin{align*}
    \frac{\partial\Psi}{\partial \n_a}=&T_a\Psi+\Big(q_a-(\Delta_a\beta)(\n_a\cdot\boldsymbol\Lambda)\Big)\Psi,
  \end{align*}
\end{pro}
\begin{proof}
  From Proposition \ref{pro:wave} we get that
  \begin{align*}
    \Big(\frac{\partial}{\partial\n_a}-T_a\Big)W_1=&\big(q_a-(\Delta_a\beta)( \n_a\cdot\boldsymbol \Lambda)\big)W_0+\mathfrak{l}W_0,
  \end{align*}
and we easily conclude that
  \begin{align*}
    \Big(\frac{\partial}{\partial\n_a}-T_a-q_a+(\Delta_a\beta)( \n_a\cdot\boldsymbol \Lambda)\Big)W_1=&\mathfrak{l}W_0.
  \end{align*}
 As $    \big(\frac{\partial}{\partial\n_a}-T_a-q_a-(\Delta_a\beta)( \n_a\cdot\boldsymbol \Lambda \big)W_2\in\mathfrak{u}$
from Proposition \ref{pro:asymptotic-module}, with $R_i=\big(\frac{\partial}{\partial\n_a}-T_a-q_a+(\Delta_a\beta)( \n_a\cdot\boldsymbol \Lambda)\big)W_i$, $i=1,2$,
 we deduce
   \begin{align*}
    \Big(\frac{\partial}{\partial\n_a}-T_a-q_a+(\Delta_a\beta)( \n_a\cdot\boldsymbol \Lambda)\Big)W_1=0
  \end{align*}
  and the result follows.
\end{proof}
For MOVPR we have
\begin{pro}
  The MOVPR satisfy
  \begin{align*}
    \frac{\partial P_{[k]}}{\partial \n_a}=(\n_a\cdot\x-q_a)\Delta_aP_{[k]}-(\Delta_a\beta)_{[k]}(\n_a\cdot\boldsymbol\Lambda)_{[k+1],[k]}P_{[k]}.
  \end{align*}
\end{pro}
\begin{proof}
  Introduce the form of the Baker function $\Psi_1$ given in  Definition \ref{definition:Baker_functions} into Proposition \ref{pro: yo que se}.
\end{proof}
The compatibility of the linear systems satisfied by the Baker functions imply
\begin{theorem}\label{teo3}
The following equation for $\beta$ holds
\begin{align}\label{difference-differential-beta}
 \Delta_b\left[\frac{\partial \beta}{\partial \boldsymbol{n}_a}+(\Delta_a \beta)\big(q_a+(\boldsymbol{n_a}\cdot \boldsymbol{\Lambda}) \beta \big) \right]\boldsymbol{n_b}\cdot \boldsymbol{\Lambda}&=
  \Delta_a\left[\frac{\partial \beta}{\partial \boldsymbol{n}_b}+(\Delta_b \beta)\big(q_b+(\boldsymbol{n_b}\cdot \boldsymbol{\Lambda}) \beta \big) \right]\boldsymbol{n_a}\cdot \boldsymbol{\Lambda}
\end{align}
\end{theorem}
\begin{proof}
  See Appendix \ref{proof11}
\end{proof}
A remarkable fact regarding \eqref{difference-differential-beta} is that is a nonlinear partial differential-difference equation but for a fixed $k$ in the lattice, that is for the rectangular matrix $\beta_{[k]}\in\R^{|[k]|\times|[k-1]|}$.
\subsection{Second order flows}
  We introduce the diagonal matrices $V_{a,b}=\diag((V_{a,b})_{[0]},(V_{a,b})_{[1]},(V_{a,b})_{[2]},\dots)$
  \begin{align}\label{eta}
   V_{a,b}\coloneq & \frac{\partial \beta}{\partial t_a}\Lambda_b,&
      (V_{a,b})_{[k]}= &\frac{\partial \beta_{[k]}}{\partial t_a}(\Lambda_b)_{[k-1],[k]}, &
      U_{a,b}\coloneq & -V_{a,b}-V_{b,a}.
    \end{align}
\begin{pro}
  Both Baker functions $\Psi_1$ and $\Psi_2$ are solutions of
  \begin{align}\label{eq: linear.wave}
 \frac{\partial \Psi}{\partial t_{(a,b)}}&=\frac{\partial^2\Psi}{\partial t_a\partial t_b}+U_{a,b}\Psi.
  \end{align}
\end{pro}
\begin{proof}
On the one hand, from \eqref{wave} we find
\begin{align*}
  \partial_{(a,b)}W_1&=(\partial_{a,b}S+S\Lambda_a\Lambda_b)W_0,\\
  \partial_a\partial_bW_1&=(\partial_a\partial_bS+\partial_a S\Lambda_b+\partial_bS\Lambda_a
  +S\Lambda_a\Lambda_b)W_0
\end{align*}
and therefore
$  (\partial_{(a,b)}-\partial_a\partial_b)W_1=-(\partial_a S\Lambda_b+\partial_bS\Lambda_a)W_0+\mathfrak{l}W_0$
  so that
  \begin{align*}
   \Big( \partial_{a,b}-\partial_a\partial_b+V_{a,b}+V_{b,a}\Big)W_1&\in\mathfrak{l}W_0.
  \end{align*}
  On the other hand, it is obvious that
    \begin{align*}
   \Big( \partial_{(a,b)}-\partial_a\partial_b+V_{a,b}+V_{b,a}\Big)W_2&\in\mathfrak{u}.
  \end{align*}
  Now, we apply Proposition \ref{pro:asymptotic-module} with
  \begin{align*}
    R_i&=  \Big( \partial_{(a,b)}-\partial_a\partial_b +V_{a,b}+V_{b,a}\Big)W_i, &i &=1,2,
  \end{align*}
  to get the result.
  \end{proof}
Observe that for $a=b$ \eqref{eq: linear.wave} reads
  \begin{align*}
 \frac{\partial \Psi_{[k]}}{\partial t^{(2)}_a}&=\frac{\partial^2\Psi_{[k]}}{\partial t_a^2}+(U_a)_{[k]}\Psi_{[k]}, & t^{(2)}\coloneq &t_{(a,a)},&
 U_a=&-2V_{a,a}.
  \end{align*}
  which is a time dependent one-dimensional Schr\"{o}dinger  type equation for the square matrices $\Psi_{[k]}$, the wave functions, and
  potential the square matrix $(U_a)_{[k]}$. Moreover, multidimensional matrix Schr\"{o}dinger equations appear if we look to other directions, thus given  $(a_1,\dots,a_d)\subset\{1,\dots,D\}$, $a_1<\dots<a_d$ we can look at the second order time flow generated by
    $\frac{\partial}{\partial t}\coloneq\frac{\partial}{\partial t_{a_1,a_1}}+\dots+\frac{\partial}{\partial t_{a_d,a_d}}$
  to get in terms of the $d$-dimensional nabla operator $\nabla\coloneq (\frac{\partial}{\partial t_{a_1}},\dots,\frac{\partial}{\partial t_{a_d}})^\top$, Laplacian $\Delta\coloneq\nabla^2=\frac{\partial^2}{\partial t_{a_1}^2}+\dots+\frac{\partial^2}{\partial t_{a_d}^2}$ and matrix potential  $U\coloneq U_{a_1,a_1}+\dots+U_{a_d,a_d}=2\nabla(\beta)\cdot\boldsymbol\Lambda$
    \begin{align*}
 \frac{\partial \Psi_{[k]}}{\partial t}&=\Delta \Psi_{[k]}+U_{[k]}\Psi_{[k]}.
  \end{align*}

\begin{cor}
  The MVOPR satisfy
  \begin{align*}
    \frac{\partial P_{[k]}}{\partial t_{(a,b)}}(\x)=\frac{\partial^2P_{[k]}}{\partial t_a\partial t_b}(\x)+x_a\frac{\partial P_{[k]}}{\partial t_b}(\x)+x_b\frac{\partial P_{[k]}}{\partial t_a}(\x)-\Big(\frac{\partial \beta_{[k]}}{\partial t_a}(\Lambda_b)_{[k-1],[k]}+\frac{\partial \beta_{[k]}}{\partial t_b}(\Lambda_a)_{[k-1],[k]}\Big)P_{[k]}(\x).
  \end{align*}
\end{cor}
\begin{proof}
  Just introduce expressions for the Baker functions in Proposition  \ref{pro.baker.expressions} in the previous Proposition.
\end{proof}
We see that again only $k$-th  site of the lattice is involved in these linear equations and, consequently, its compatibility will lead to equations for the coefficients evaluated at that site.
These nonlinear equation for which $\beta_{[k]}$ is a solution are
  \begin{theorem}\label{teo4}
The following nonlinear partial differential equation
    \begin{align}\label{eq:beta}\begin{aligned}
      \partial_{(c,d)}(\partial_a\beta\Lambda_b+\partial_b\beta\Lambda_a)-\partial_{(a,b)}(\partial_c\beta\Lambda_d+\partial_d\beta\Lambda_c)
  =
   & \partial_a\partial_b(\partial_c\beta\Lambda_d+\partial_d\beta\Lambda_c)-  \partial_c\partial_d(\partial_a\beta\Lambda_b+\partial_b\beta\Lambda_a)\\
    &+    (\partial_b\partial_c\beta)(\Lambda_d\beta\Lambda_a-\Lambda_a\beta\Lambda_d)
    +(\partial_b\partial_d\beta)(\Lambda_c\beta\Lambda_a-\Lambda_a\beta\Lambda_c)
   \\& +(\partial_a\partial_c\beta)(\Lambda_d\beta\Lambda_b-\Lambda_b\beta\Lambda_d)
    +(\partial_a\partial_d\beta)(\Lambda_c\beta\Lambda_b-\Lambda_b\beta\Lambda_c)\\
&+ \big[ \partial_a\beta\Lambda_b+ \partial_b\beta\Lambda_a,\partial_c\beta\Lambda_d+\partial_d\beta\Lambda_c\big]
    \end{aligned}
    \end{align}
    is satisfied for $a,b,c,d\in\{1,\dots, D\}$.
  \end{theorem}
  \begin{proof}
 See Appendix \ref{proof12}.
  \end{proof}
  Observe that this equation decouples giving for each $k$  the same equation \eqref{eq:beta} up to the replacements $\beta\to\beta_{[k]}$ and $\Lambda_A\to(\Lambda_A)_{[k-1],[k]}$, $k=1,2,\dots$ and $A=a,b,c,d$.
    For the particular case $a=b=A$ and $c=d=B$ and use the notation $t_A=x$,  $t_B=y$, $t_{(A,A)}=$ and $t_{B,B}=t$ we get
     \begin{align*}
\frac{\partial^2\beta}{\partial t\partial x}\Lambda_A-\frac{\partial\beta}{\partial s\partial y}\Lambda_B
  =
   &\frac{\partial^3\beta}{\partial x^2\partial y}\Lambda_B- \frac{\partial^3\beta}{\partial x\partial y^2}\Lambda_A
   +2 \Big[\frac{\partial\beta}{\partial x}\Lambda_A,\frac{\partial\beta}{\partial y}\Lambda_B\Big] + 2\frac{\partial^2\beta}{\partial x\partial y}(\Lambda_B\beta\Lambda_A-\Lambda_A\beta\Lambda_B).
    \end{align*}

\subsection{Exploring third  order flows}

Associated with the third order times $t_{(a,b,c)}$  we introduce the following block diagonal matrices
\begin{align*}
  V_{a,b,c}=\diag((V_{a,b,c})_{[0]},(V_{a,b,c})_{[1]},(V_{a,b,c})_{[2]},\dots)
\end{align*}
  with
  \begin{align*}
    V_{a,b,c}\coloneq &\frac{\partial\beta}{\partial t_a}[\beta,\Lambda_b]\Lambda_c, & ( V_{a,b,c})_{[k]} =&\frac{\partial\beta_{[k]}}{\partial t_a}\Big(\beta_{[k-1]}\big(\Lambda_b\big)_{[k-2],[k-1]}-\big(\Lambda_b\big)_{[k-1],[k]}\beta_{[k]}\Big)\big(\Lambda_c\big)_{[k-1],[k]},\\
    =& \frac{\partial\beta^{(2)}}{\partial t_a}\Lambda_b\Lambda_c
    -\frac{\partial\beta}{\partial t_a}\Lambda_b\beta\Lambda_c,& =&\frac{\partial\beta^{(2)}_{[k]}}{\partial t_a}
    \big(\Lambda_b\big)_{[k-2],[k-1]}-\frac{\partial\beta_{[k]}}{\partial t_a}\big(\Lambda_b\big)_{[k-1],[k]}\beta_{[k]}\big(\Lambda_c\big)_{[k-1],[k]}.
  \end{align*}
  Observe the use of Proposition \ref{otro mas};  we remark  that $( V_{a,b,c})_{[k]}$ depends on $\beta_{[k]}$ and $\beta^{(2)}_{[k]}$ only, coefficients of the MOVPR for the second and third higher degree monomials, $
  P_{[k]}(\x)=
  \chi_{[k]}(\x)+\beta_{[k]}\chi_{[k-1]}(\x)+\beta_{[k]}^{(2)}\chi_{[k-2]}(\x)+\cdots+\beta_{[k]}^{(k)}
  $.
   If we insist in using only the second higher total degree coefficient and not the third higher total degree coefficient there is price we must pay, now we involve two polynomials $P_{[k]}$ and $P_{[k-1]}$ --as we require of $\beta_{[k]}$ and $\beta_{[k-1]}$.
  Then
  \begin{pro}\label{latriple}
    The Baker functions $\Psi_1$ and $\Psi_2$ are both solutions of the third order linear differential equations
    \begin{align*}
      \frac{\partial\Psi}{\partial t_{(a,b,c)}}=&\frac{\partial^3\Psi}{\partial t_a\partial t_b\partial t_c}
      -V_{a,b}\frac{\partial\Psi}{\partial t_{c}} -V_{c,a}\frac{\partial\Psi}{\partial t_{b}} -V_{b,c}\frac{\partial\Psi}{\partial t_{a}}
            -\Big(\frac{\partial V_{a,b}}{\partial t_c}+\frac{\partial V_{b,c}}{\partial t_a}+\frac{\partial V_{c,a}}{\partial t_b}+ V_{a,b,c}+ V_{b,c,a}+ V_{c,b,a}\Big)\Psi.
    \end{align*}
  \end{pro}
\begin{proof}
See Appendix \ref{prooftriple}.
\end{proof}

For $a=b=c$ in terms of $t^{(3)}_a\coloneq t_{(a,a,a)}$ and
\begin{align*}
  \tilde U_a\coloneq -3\Big(\frac{\partial^2\beta}{\partial t_a^2}\Lambda_a+\frac{\partial\beta^{(2)}}{\partial t_a}\Lambda_a^2
    -\frac{\partial\beta}{\partial t_a}\Lambda_a\beta\Lambda_a\Big),
\end{align*}
the Baker functions satisfies
  \begin{align*}
      \frac{\partial\Psi}{\partial t^{(3)}_a}=&\frac{\partial^3\Psi}{\partial t_a^3}
      +\frac{3}{2}U_a\frac{\partial\Psi}{\partial t_{a}}+\tilde U_a\Psi.
    \end{align*}
Therefore, in terms of the differential operators
   \begin{align}
   \label{Lax pair}\begin{aligned}
    \mathcal L_a&\coloneq\frac{\partial^2}{\partial t_a^2}+U_a, & U_a&=-2\frac{\partial \beta}{\partial t_a}\Lambda_a,\\
 \mathcal P_b &\coloneq\frac{\partial^3}{\partial t_b^3}+\frac{3}{2}U_b\frac{\partial}{\partial t_{a}}+\tilde U_b,&\tilde U_b &= -3\frac{\partial^2\beta}{\partial t_b^2}\Lambda_b-3\frac{\partial\beta^{(2)}}{\partial t_b}\Lambda_b^2
    +3\frac{\partial\beta}{\partial t_b}\Lambda_b\beta\Lambda_b,
   \end{aligned}
  \end{align}
  the wave matrix fufills
  \begin{align*}
    \frac {\partial W_1}{\partial t^{(2)}_a}&= \mathcal L_a(W_1),&
     \frac {\partial W_1}{\partial t^{(3)}_b}&= \mathcal P_b(W_1),
      \end{align*}
Hence, the following  compatibility equations  holds
\begin{align*}
 R_{ab}(W_1)&=0, &R_{ab}\coloneq \frac {\partial \mathcal L_a}{\partial t^{(3)}_b}-\frac {\partial \mathcal P_b}{\partial t^{(2)}_a}+\big[\mathcal L_a,\mathcal P_b\big] .
\end{align*}
Observe that
\begin{align*}
 R_{ab}= -3\frac{\partial U_a}{\partial t_b}\frac{\partial^2}{\partial t_b^2}+3\frac{\partial U_b}{\partial t_a}\frac{\partial^2}{\partial t_a\partial t_b}+
    2\frac{\partial \tilde U_b}{\partial t_a}\frac{\partial}{\partial t_a}+\Big(-\frac{3}{4}\frac{\partial U_b}{\partial t^{(2)}_a}+\frac{3}{2}\frac{\partial^2U_b}{\partial t_a^2} -3\frac{\partial^2 U_a}{\partial t_b^2}+\frac{3}{2}[U_a,U_b]\Big)\frac{\partial}{\partial t_b}\\+
    \frac{1}{2}\frac{\partial U_a}{\partial t^{(3)}_b}-\frac{1}{2}\frac{\partial\tilde U_b}{\partial t^{(2)}_a}+\frac{\partial^2\tilde U_b}{\partial t_a^2}-\frac{\partial^3 U_a}{\partial t_b^3}-\frac{3}{2}U_b\frac{\partial U_a}{\partial t_b}+[U_a,\tilde U_b]
\end{align*}
is a matrix second order partial differential operator in $t_a,t_b$. Its action on the wave matrix can be split into diagonals, $\big(R_{ab}(W)\big)W_0^{-1}$ is a block Hessenberg matrix with all its block superdiagonals equal to zero but for the first and second ones.\footnote{It is a consequence of $\partial_aW_1=(\Lambda_a+\beta\Lambda_a)W_0+\mathfrak uW_0$ and
$\partial_a\partial_bW_1=\big(\Lambda_a\Lambda_b+\beta\Lambda_a\Lambda_b+\frac{\partial\beta}{\partial t_a}\Lambda_b+\frac{\partial\beta}{\partial t_b}\Lambda_a+\beta^{(2)}\Lambda_a\Lambda_b\big)W_0+\mathfrak u W_0$} Thus, we should ask these two superdiagonals as well to the main diagonal to cancel; i.e.,
\begin{pro}
  The matrices $\beta$ and $\beta^{(2)}$ are subject to
\begin{align*}
 \frac{\partial^2 \beta^{(2)}}{\partial t_a \partial t_b}\Lambda_a&=
 \frac{\partial}{\partial t_b}\left[\frac{\partial \beta}{\partial t_a} \Lambda_a \beta-
 \frac{1}{2}\frac{\partial^2 \beta}{\partial t_a^2}+\frac{1}{4}\frac{\partial \beta}{\partial t_a^{(2)}} \right],\\
 0&=3\frac{\partial^2}{\partial t_b^2}\left[\frac{1}{2}\frac{\partial \beta }{\partial t_a^{(2)}}-\frac{\partial^2 \beta }{\partial t_a^{2}}+2\frac{\partial \beta }{\partial t_a}\Lambda_a \beta \right]\Lambda_b
 +\frac{\partial}{\partial t_a}\left[2\frac{\partial^3 \beta }{\partial t_b^{3}}-\frac{\partial \beta }{\partial t_b^{(3)}}
 +\left(\frac{\partial \beta }{\partial t_b}\Lambda_b \beta-\frac{\partial \beta^{(2)}\Lambda_b }{\partial t_b} \right)\Lambda_b \beta\right]\Lambda_a\\
 &+3\frac{\partial}{\partial t_b}\left[\left(2\frac{\partial \beta }{\partial t_a}\Lambda_a \beta^{(2)}+
 \frac{1}{2}\frac{\partial \beta^{(2)}}{\partial t_a^{(2)}}-\frac{\partial^2 \beta^{(2)}}{\partial t_a^2}\right)\Lambda_b^2
 -2\frac{\partial \beta}{\partial t_b}\Lambda_b \frac{\partial \beta}{\partial t_a}\Lambda_a \right]\\
&+3 \frac{\partial \beta}{\partial t_b}\Lambda_b \left[\frac{\partial^2 \beta}{\partial t_a \partial t_b}-
 2\frac{\partial \beta}{\partial t_a}\Lambda_a \beta -\frac{1}{2}\frac{\partial \beta}{\partial t_a}\right]\Lambda_b
 -6\frac{\partial^2 \beta}{\partial t_b \partial t_a} \Lambda_b \beta^{(2)}\Lambda_a \Lambda_b.
 \end{align*}
\end{pro}

\section{Linear isometry invariant measures and MOVPR}\label{6}
 In this section we consider orthogonal transformations  $R\in \operatorname{O}(\R^D)$; i.e., linear isometries $R:\R^D\to\R^D$ preserving the dot or scalar product:  $R\boldsymbol u\cdot R\boldsymbol v=\boldsymbol u\cdot \boldsymbol v$, $\forall \boldsymbol u,\boldsymbol v\in\R^D$. For the matrix $[R]_B$ in the canonical basis $B=\{\ee_1,\dots,\ee_D\}$ of $\R^D$  of the orthogonal endomorphism means $[R]_B^\top=[R]_B^{-1}$. Given such an orthogonal transformation
$  \x\rightarrow R\x$,
we assume the linear \emph{isometry} invariance condition
$  \dd\mu(\x)=\dd\mu(R\x)$.\footnote{The measure $\mu$ is said to be invariant under $R$ if for every measurable set $A\subset\R^D$ we have $\mu \left( R^{-1} (A) \right) = \mu (A)$.}

\subsection{Symmetric powers of a linear isometry. Orthonormal basis and biorthogonal systems}
 What is the action of this linear isometry in the set of MVOPR? or putting it in other equivalent terms, how do it act of the corresponding symmetric tensor powers?
Given any set  of linear transformations $\{f_i\}_{i=1}^m\in\text{End}(\R^D)$ one can construct a map $f_1\odot \dots\odot f_m\in\operatorname{End}\left(\big(\R^D\big)^{\odot k}\right)$ such that
\begin{align*}
  (f_1\odot \dots\odot f_m)(\boldsymbol u_1\odot\dots\odot \boldsymbol u_m)=&f_1(\boldsymbol u_1)\odot\dots\odot f_m(\boldsymbol u_m), & \forall \boldsymbol u_i&\in\R^D.
\end{align*}
In this manner we introduce the $k$-th symmetric power of the endomorphism $R$ acting on symmetric tensor powers, $R^{\odot k}\in\operatorname{End} \left(\big(\R^D\big)^{\odot k}\right)$ and,  moreover,  a diagonal block endomorphism in the symmetric algebra, $\mathscr R\coloneq \diag(1,\mathscr R_{[1]},\mathscr R_{[2]},\dots)=1\oplus\mathscr R_{[1]}\oplus\mathscr R_{[2]}\oplus\cdots\in\operatorname{End}\big(\operatorname{S}(\R^D)\big)$, with its diagonal blocks given by $\mathscr R_{[k]}\coloneq R^{\odot k}$. For a given  invertible endomorphism $R$, with inverse $R^{-1}$,  the corresponding endomorphism $\mathscr R$ in the symmetric algebra is invertible with inverse $(\mathscr R^{-1})_{[k]}=(R^{-1})^{\odot k}$.
The $\operatorname S^k(\R^D)$ is equipped with a natural scalar product $\langle\cdot,\cdot\rangle^{(k)}$ given  in terms of permanents, see Appendix \ref{spowers}. In particular, for decomposable symmetric tensors \eqref{permanents}
\begin{align*}
  \langle R^{\odot k}(\boldsymbol u_1\odot\dots\odot \boldsymbol u_k), R^{\odot k}(\boldsymbol v_1\odot\dots\odot \boldsymbol v_k) \rangle^{(k)}=&\frac{1}{k!}\operatorname{perm}\PARENS{\begin{matrix}
R \boldsymbol u_1\cdot R\boldsymbol v_1 &\cdots& R\boldsymbol u_1\cdot R\boldsymbol v_k\\
 \vdots & & \vdots\\
 R \boldsymbol u_k\cdot R\boldsymbol v_1 &\cdots &R\boldsymbol u_k\cdot R\boldsymbol v_k
  \end{matrix}}\\
  =&\frac{1}{k!}\operatorname{perm}\PARENS{\begin{matrix}
 \boldsymbol u_1\cdot \boldsymbol v_1 &\cdots& \boldsymbol u_1\cdot \boldsymbol v_k\\
 \vdots & & \vdots\\
  \boldsymbol u_k\cdot \boldsymbol v_1 &\cdots &R\boldsymbol u_k\cdot \boldsymbol v_k
  \end{matrix}}\\=&  \langle \boldsymbol u_1\odot\dots\odot \boldsymbol u_k), \boldsymbol v_1\odot\dots\odot \boldsymbol v_k \rangle^{(k)}
\end{align*}
Thus, the $k$-th tensor power $R^{\odot k}$ is an orthogonal transformation in $\big\{\operatorname{S}^k(\R^D),\langle\cdot,\cdot\rangle^{(k)}\big\}$. At this point we stress that care must be taken when we express this fact in terms of matrices. Observe that the  \emph{canonical} basis $\{\ee^{\q_i}\}_{i=1}^{|[k]|}$, with $\ee^{\q_i}=\ee_1^{\odot \alpha_{i,1}}\odot\dots\odot\ee_D^{\odot \alpha_{i,D}}$, is an orthogonal set for $\langle\cdot,\cdot\rangle^{(k)}$ but not and orthonormal basis. In fact,  we know that, see Appendix \ref{spowers}, $
  \|\ee^{\q_i}\|^{2}={k \choose \q_i}^{-1}$, being the metric matrix in this basis  the inverse of the multinomial matrix.
Let us find the matrix representing  $R^{\odot k}$ in the canonical basis $B_c=\{\ee^{\q_i}\}_{i=1}^{|[k]|}$; we proceed to compute
\begin{align*}
  R^{\odot k}\ee^{\q_i}&= (R\ee_1)^{\odot \alpha_{i,1}}\odot\dots\odot(R\ee_D)^{\odot \alpha_{i,D}}\\
  &=\big(\sum_{j=1}^DR_{j_1,1}\ee_{j}\big)^{\odot \alpha_{i,1}}\odot\dots\odot\big(\sum_{j=1}^DR_{j_1,1}\ee_{j}\big)^{\odot \alpha_{i,D}}\\
  &=\Big[\sum_{|\q'|=\alpha_{i,1}}{\alpha_{i,1}\choose\q'}\prod_{j=1}^D R_{j,1}^{\alpha'_{j}}\ee^{\alpha'}\Big]\odot\cdots\odot
  \Big[\sum_{|\q'|=\alpha_{i,D}}{\alpha_{i,D}\choose\q'}\prod_{j=1}^D R_{j,D}^{\alpha'_{j}}\ee^{\alpha'}\Big]\\
  &=\sum_{j=1}^{|[k]|}(\mathscr R_{[k]})_{\q_j,\q_i}\ee^{\q_j},
\end{align*}
which gives the matrix $  [\mathscr R_{[k]}]_{B_c}=\big[(\mathscr R_{[k]})_{\q_i,\q_j}\big]$. Then, as the transformation preserves the scalar product with metric matrix  given by $\mathcal M_{[k]}^{-1}$, the matrix in the canonical basis of the $k$-th symmetric tensor power satisfies
\begin{align}\label{ortocano}
[\mathscr R_{[k]}]_{B_c}^\top\mathcal M_{[k]}^{-1}[\mathscr R_{[k]}]_{B_c}=\mathcal M_{[k]}^{-1}.
\end{align}
Instead of the canonical basis $B_c$ we could consider the orthonormal linear basis $B=\{\boldsymbol u_i\}_{i=1}^{|[k]|}$ with the normalized vectors $  \boldsymbol u_{i}=\Big(\sqrt{k\choose \q_i}\Big)^{-1}\ee^{\q_i}$.
In this basis the matrix for $R^{\odot k}$ is
$  [\mathscr R_{[k]}]_B=\mathcal M_{[k]}^{-1/2}[\mathscr R_{[k]}]_{B_c}\mathcal M_{[k]}^{1/2}$
is an orthogonal matrix; i.e.,
$ [\mathscr R]_{B}^\top=[\mathscr R]_{B}^{-1}=[\mathscr R^{-1}]_{B}$. Another orthogonal  basis is $\tilde B_c=\{\tilde\ee^{\q_i}\}_{i=1}^{|[k]|}$, with
$\tilde\ee^{\q_i}={k\choose \q_i}^{-1}\ee^{\q_i}$,  that
despite not being orthonormal, forms with the canonical basis a biorthogonal system, i.e. $\langle \tilde\ee^{\q_i},\ee^{\q_j}\rangle^{(k)}=\delta_{i,j}$.
The matrix representing  $\mathscr R_{[k]}$ in this orthogonal basis is
\begin{align*}
  \eta_{R,[k]}\coloneq [\mathscr R_{[k]}]_{\tilde B_c}=&\mathcal M_{[k]}^{-1}[\mathscr R_{[k]}]_{B_c}\mathcal M_{[k]}\\
  =&\mathcal M_{[k]}^{-1/2}[\mathscr R_{[k]}]_{B}\mathcal M_{[k]}^{1/2},
\end{align*}
and in the symmetric algebra we have the corresponding block diagonal matrix
\begin{align}\label{varsigma}
  \eta_R=&\mathcal M^{-1}[\mathscr R]_{B_c}\mathcal M.
\end{align}
which satisfies
$\eta_{R}^\top\mathcal M\eta_{R}=\mathcal M$. Here $[\mathscr R]_{B_c}\coloneq \diag([\mathscr R_{[0]}]_{B_c},[\mathscr R_{[1]}]_{B_c},\dots)$ and $[\mathscr R]_{B}\coloneq \diag([\mathscr R_{[0]}]_{B},[\mathscr R_{[1]}]_{B},\dots)$.
Observe that
\begin{align*}
\eta_{[R]^\top}=  \eta_{R^{-1}}=\eta_R^{-1}=&\mathcal M^{-1/2}[\mathscr R^{-1}]_{B}\mathcal M^{1/2}=\mathcal M^{-1/2}[\mathscr R]_{B}^\top\mathcal M^{1/2}
  =(\mathcal M^{1/2}[\mathscr R]_{B}\mathcal M^{-1/2})^\top\\=&\mathcal M^{-1}(\mathcal M^{-1/2}[\mathscr R]_{B}\mathcal M^{1/2})^\top\mathcal M\\=&\mathcal M^{-1}\eta_R^\top\mathcal M,
\end{align*}
and also that, as $B_c$ and $\tilde B_c$ forms a biorthogonal system,  we have the following relations
\begin{align}\label{etarelations}
  \eta_R^\top=&[\mathscr R]_{B_c}^{-1}, & \eta_R^{-1}=&[\mathscr R]_{B_c}^\top.
\end{align}

\subsection{Applications to monomials and shift matrices}
We now see how the above developments apply to the monomials $\chi$ and the shift matrices $\Lambda$ introduced previously.
\begin{pro}
We have
  \begin{align}
\notag\chi(R\x)&=\eta_R\chi(\x),\\
  R\n\cdot\boldsymbol\Lambda&=\eta_R (\n\cdot\boldsymbol\Lambda) \eta_R^{-1}.\label{a invertir}
 \end{align}
\end{pro}
\begin{proof}
  To prove the first relation notice that form  \eqref{chi-symmetric power} we get
  \begin{align*}
  \chi_{[k]}(R\x)=&\mathcal M_{[k]}^{-1}[(R\x)^{\odot k}]_{B_c}=\mathcal M_{[k]}^{-1}[\mathscr R_{[k]}]_{B_c}[\x^{\odot k}]_{B_c}=\mathcal M_{[k]}^{-1}[\mathscr R_{[k]}]_{B_c}\mathcal M_{[k]}\chi_{[k]}(\x)\\=&\eta_{R,[k]}\chi_{[k]}(\x).
  \end{align*}
 For the second formula we observe that \begin{align*}\eta_R(\n\cdot\boldsymbol\Lambda)\eta_R^{-1}\chi(\x)=&\eta_R(\n\boldsymbol\cdot\boldsymbol\Lambda)\chi(R^{-1}\x)=(\n\cdot R^{-1}\x)\eta_R\chi(R^{-1}x)=(R\n\cdot \x)\chi(\x)\\=&(R\n\cdot \boldsymbol\Lambda)\chi(\x),
  \end{align*}which holds $\forall \x\in\R^D$ so that the result follows.
\end{proof}
When $R\in\mathfrak S_D\subset \operatorname{O}(\R^D)$ we also have $ \chi^*(R\x)=\eta_R\chi^*(\x)$.
We know that $\Lambda_a$ has no left inverse but it does have a right inverse, its transpose, $\Lambda_a\Lambda_a^\top=\I$. In this paper we have derived a number of results with $\n\cdot\boldsymbol\Lambda$ and sometimes, for example in Proposition \ref{betaH2step} and Proposition \ref{ostias}, is useful to find the right inverse of $\n\cdot\boldsymbol\Lambda$ .

\begin{pro}\label{invirtiendo 1}
  Given any vector $\n\in\R^D$ find $R\in\operatorname{O}(\R^D)$ such that $R\ee_a=\n$. Then, a right inverse of $ (\n\cdot\boldsymbol\Lambda)$ is
  $\eta_R\Lambda_a^\top\eta_R^{-1}$; i.e.,
  \begin{align*}
    (\n\cdot\boldsymbol\Lambda)\eta_R\Lambda_a^\top\eta_R^{-1}=\I.
  \end{align*}
\end{pro}
\begin{proof}
  From \eqref{a invertir} we know that $\n\cdot\boldsymbol\Lambda=(R\ee_a)\cdot\boldsymbol\Lambda=\eta_R \Lambda_a\eta_R^{-1}$ but the right inverse of this matrix is $\eta_R\Lambda_a^\top\eta_R^{-1}$ and the result is proven.
\end{proof}
In our opinion a nicer result, as is does not depend on any alien isometry $R$, is
\begin{pro}\label{moorepenroseright}
A right inverse for $(\n\cdot\boldsymbol\Lambda)_{[k-1],[k]}$, $k>0$, is
\begin{align*}
 \mathcal M_{[k]}^{-1/2}\big( (\n\cdot\boldsymbol\Lambda)_{[k-1],[k]}\mathcal M_{[k]}^{-1/2}\big)^+=\mathcal M_{[k]}^{-1}((\n\cdot\boldsymbol\Lambda)_{[k-1],[k]}) ^\top\big(
 (\n\cdot\boldsymbol\Lambda)_{[k-1],[k]}\mathcal M_{[k]}^{-1} ( (\n\cdot\boldsymbol\Lambda)_{[k-1],[k]})^\top)
 \big)^{-1}
\end{align*}
\end{pro}
\begin{proof}
  See Appendix \ref{proof moorepenroseright}.
\end{proof}

\subsection{Consequences of the measure symmetry}
\begin{pro}\label{pro13}
Whenever, for a given orthogonal endomorphism $R\in\operatorname{O}(\R^D)$, there is a symmetry in the measure of the form $\dd\mu(\x)=\dd\mu(R\x)$ we have
\begin{enumerate}
  \item The moment matrix satisfies
  \begin{align*}
  \eta_RG \eta_R^\top &=G.
 \end{align*}
  \item The factors of the Cholesky factorization \eqref{cholesky} are such that
  \begin{align}\label{symmetry-S}
    \eta_RS  \eta_R^{-1}&=S, &
   \eta_RH  \eta_R^\top&=H.
  \end{align}
  \item Moreover,
  \begin{align*}
\eta_{R,[k]}\beta_{[k]}=\beta_{[k]}\eta_{R,[k-1]}
 \end{align*}
 \end{enumerate}
\end{pro}
\begin{proof}
See Appendix \ref{proof13}.
\end{proof}
Observe that while we can write for $[\eta_R,S]=0$ for the quasi-tau matrices we can write $[\eta_R,H\mathcal M]=0$ ($\big(\eta_R^\top\big)^{-1}=\mathcal M\eta_R\mathcal M^{-1}=[\mathscr R]_{B_c}^{-1}$).

Now, we are ready to deduce how the MOVPR of a symmetric measure behaves under the symmetry of the measure
\begin{pro}\label{pro14}
Let us assume that for an orthogonal transformation $R\in\operatorname{O}(\R^D)$  the measure satisfies $\dd\mu(\x)=\dd\mu(R\x)$, then:
\begin{enumerate}
 \item The  MOVPR  fulfill
 \begin{align}\label{symmetri-MOVPR}
 P\big(Rx\big)= &\eta_RP(\x).
 \end{align}
\item  The Jacobi  matrices are such that
 \begin{align}\label{symmetry-jacobi}
R\n\cdot\boldsymbol J=  \eta_R (\n\cdot\boldsymbol J)  \eta_R^{-1}.
 \end{align}
 \item The Christoffel--Darboux kernel remains invariant
 \begin{align}\label{RCD}
   K^{(\ell)}(R\x,R\y)&=  K^{(\ell)}(\x,\y).
 \end{align}
\end{enumerate}
 \end{pro}
\begin{proof}
see Appendix \ref{proof14}.
\end{proof}
When $R\in\mathfrak S_D$, as $x_1\cdots x_D$ is invariant under permutation of coordinates, we also have $C\big(R\x\big)=\eta_RC(\x)$ and $Q^{(\ell)}(R\x,R\y)=  Q^{(\ell)}(\x,\y)$.

\subsection{Compatible Toda flows}

We now request that the symmetry  $\dd\mu(\x)=\dd\mu(R\x)$ is preserved under the integrable deformations discussed previously. As before we distinguish two cases, the discrete case and the continuous case.
For the first situation we have
\begin{pro}
If $\n$ is invariant under the transformation $R\in\operatorname{O}(\R^D)$, $\n=R\n$, then the corresponding discrete transformation preserves the isometry invariance of the measure
\begin{align*}
  T\dd\mu(\x)=T\dd\mu(R\x).
\end{align*}
\end{pro}
\begin{proof}
The new measure  is $T\dd\mu(\x)=(\n\cdot\x-q)\dd\mu(\x)$ so that
\begin{align*}
T\dd\mu(R\x)=&(\n\cdot R\x-q)\dd\mu(R\x)\\=&(\n\cdot R\x-q)\dd\mu(\x)\\=&(R^\top\n\cdot\x-q)\dd\mu(\x).
\end{align*}
Therefore when $\n=R^\top\n$, as $R^\top=R^{-1}$ we get the that the new measure is invariant.
\end{proof}

For the continuous flows recall Definition \ref{lostiempos} and Proposition \ref{deformacion medida} and notice that if we order the times $t=(t_{[0]},t_{[1]},\dots)$, $t_{[k]}=(t_{\q_1},\dots,t_{\q_{|[k]|}})$
\begin{pro}
  The continuous Toda flows preserve the symmetry $\dd\mu(\x)=\dd\mu(R\x)$ whenever the times are such
  \begin{align*}
    t_{[k]}\eta_{R,[k]}=t_{[k]}.
  \end{align*}
  \end{pro}
\begin{proof}
  As we know from Proposition \ref{deformacion medida}
  \begin{align*}
   \dd\mu_t(R\x)=&\exp\big(t\chi(R\x)\big)\dd\mu(R\x)
   \\=& \exp\big(t \eta_R\chi(\x)\big)\dd\mu(\x)
  \end{align*}
  and when $t \eta_R=t$  we get $\dd\mu_t(R\x)=\dd\mu_t(\x)$.
\end{proof}

\subsubsection{Linear subspaces of fixed points of a linear isometry}
We have seen that to analyze compatible flows with the linear isometry invariance is crucial to find fixed points for the linear isometries.
\begin{definition}
The linear subspace of fixed points of the linear isometry $R$ is
  \begin{align*}
   V_R=\{\boldsymbol v\in\R^D: R\boldsymbol v=\boldsymbol v\}.
  \end{align*}
\end{definition}

Interesting examples of linear  isometries are provided by reflections. Given a non-zero vector $\n$ the corresponding  Householder reflection is
\begin{align*}
  r_{\n}=\I_D-2\frac{1}{\n\cdot\n}\n \n^\top.
\end{align*}
This is a reflection in the hyperplane $\n^\bot$, it is idempotent $r_{\n}^2=\I_D$, $r_{\n}\big|_{\R\n}=-\operatorname{id}$, (negating any vector component parallel to $\n$),  and
$r_{\n}\big|_{\n^\bot}=\operatorname{id}$. Therefore, for the Householder case $V_R=\n^\bot$.
Any orthogonal matrix $R\in\operatorname{O}(\R^D)$ is as a product of at most $D$ Householder reflections. Given an orthogonal set $\{\n_1,\dots, \n_m\}\subset\R^D$ the product $R$ of the corresponding Householder reflections (which happens to be commutative as the order of the factors do not affect the result) is $R=\I_D-2\sum_{i=1}^m\frac{1}{\n_i\cdot\n_i}\n_i \n_i^\top$. This is a reflection, with reflection hyperplane $\{\n_1,\dots,\n_m\}^\bot$, negating the components parallel to $\R\{\n_1,\dots,\n_m\}$.
Now, the fixed point subspace  $V_R=\{\n_1,\dots,\n_m\}^\bot$ is the reflection plane.

For $D=2$ the orthogonal transformations could be of two types, a rotation of angle $\theta$,
$R_\theta=\begin{psmallmatrix}
\cos \theta & -\sin \theta \\
\sin \theta & \cos \theta
\end{psmallmatrix}$ and a reflection according to the vector $\n=\begin{psmallmatrix}
  -\sin (\theta/2)\\ \cos(\theta/2)
\end{psmallmatrix}$, $R=\I_2-2\n \n^\top=\begin{psmallmatrix}
\cos \theta & \sin \theta \\
\sin \theta & -\cos \theta
\end{psmallmatrix}$ with reflection line given by $\R\begin{psmallmatrix}
  \cos (\theta/2)\\ \sin(\theta/2)
\end{psmallmatrix}$.
For $\theta=\pi/2$ the reflection is about the line  $y=x$ and therefore exchanges $x$ and $y$: it is a permutation matrix $\begin{psmallmatrix}
0 & 1\\
1 & 0
\end{psmallmatrix}$.

In general, given a linear isometry $R\in\operatorname{O}(\R^D)$ there exists an orthonormal basis $\{\boldsymbol u_a\}_{a=1}^D$ such that its matrix reads
\begin{align*}
  R=\diag(\underbracket{1,\dots,1}_{\text{$p$ of them} },\underbracket{-1,\dots,-1}_{\text{$q$ of them}},R_{\theta_1},\dots, R_{\theta_m},)
\end{align*} being $p-q$  even  or odd depending whether $D$ is even or odd, here $R_\theta$ is a two dimensional non trivial rotation of angle $\theta$. Therefore $V_R=\R\{\boldsymbol u_1,\dots,\boldsymbol u_p\}$;
notice that for $D$ even it could happen that $V_R=\{0\}$, but for $D$ odd we always have a nontrivial fixed point subspace,  $\dim V_R\geq 1$.

\subsubsection{Secant varieties of Veronese varieties and linear isometry invariance preserving flows}
\begin{pro}
 If the times $t^\top\in \operatorname{Sym}(\R^D) $ are restricted to  belong to the symmetric algebra of the fixed point subspace of the linear isometry $R$, i.e.  $t^\top\in \operatorname{Sym}(V_R)\subset \operatorname{Sym}(\R^D) $, the linear isometry invariance condition of the measure is preserved as the time passes by, $\dd\mu_t(\x)=\dd\mu_t(R\x)$ $\forall t^\top \in  \operatorname{Sym}(V_R)$.
\end{pro}
\begin{proof}

Observing that $\eta_R^\top=[\mathscr R]_{B_c}$, the linear isometry invariance preserving condition can be written as
\begin{align*}
[R^{\odot k}]_{B_c} t_{[k]}^\top=t_{[k]}^\top.
\end{align*}
The first non trivial condition is that
\begin{align*}
  R [t]_1^\top=& t_{[1]}& [t]_{[1]}^\top=&\PARENS{\begin{matrix}
    t_1\\ \vdots\\t_D
  \end{matrix}}
\end{align*}
and therefore $[t]_{[1]}^\top\in V_R$.

In order to explore what kind of higher flows will preserve the linear isometry invariance condition of the measure we observe that if $V_{[1]}=\{\boldsymbol t\in\R^D: R\boldsymbol t=\boldsymbol t\}$ any symmetric power tensor in $V_{[k]}=(V_{[1]})^{\odot k}$ will be subspace of fixed point for $R^{\odot k}$. Indeed,   $V_{[k]}$ is linearly generated by decomposable symmetric tensors $\boldsymbol v_1\odot\cdots\odot \boldsymbol v_k$ with $\boldsymbol v_i\in V_{[1]}$ and
\begin{align*}
 R^{\odot k} (\boldsymbol v_1\odot\cdots\odot \boldsymbol v_k)=&(R \boldsymbol v_1)
 \odot\cdots\odot (R\boldsymbol v_k)\\=&\boldsymbol v_1\odot\cdots\odot \boldsymbol v_k.
\end{align*}
\end{proof}

The map $\mathcal v_{m,k}:\R^m\to (\R^m)^{\odot k}$ taking $\boldsymbol v\to \boldsymbol v^{\odot k}$  has as its image the Veronese variety
$\mathscr V_{m,k}\coloneq \{\x^{\odot k}\in(\R^m)^{\odot k}: \x\in\R^m\}$.
It happens \cite{Comon} that every symmetric power tensor can be written for some $r \geq  0$  as
$\sum_{i=1}^r \boldsymbol v_i^{\odot k}$
and the symmetric tensor rank is the minimum when this holds.  Hence, the symmetric power can be described by $r$ points of the Veronese variety; the closure of the union of all linear spaces spanned by $r$ points of
the Veronese variety $\mathscr V_{m,k}$  is called the $(r -1)$-th secant variety of $\mathscr V_{m,k}$. Therefore, in this case we require the times to belong to one the secant varieties of the Veronese variety $\mathscr V_{\dim (V_R),k}$.

\begin{appendices}
\section{Compositions, multisets and symmetric algebras}\label{symmetric}\settocdepth{section}
\subsection{Compositions and multisets}
From combinatorics \cite{stanley} we know that  a weak $D$-composition of an integer $k$ is a way of writing $k$ as the sum $D$ non-negative integers.  Notice that while for a composition we require the parts to be positive integers (excluding therefore the zero) for weak compositions the zero is allowed. The problem of counting the number $N(k,D)$ of weak compositions, i.e. the cardinality of the set $[k]$, is related to the problem of counting the number of compositions, which is $\binom{k-1}{D-1}$.
 In fact, given a weak $D$-composition $k=k_{i,1}+\dots+k_{i,D}$ if we put $q_{i,j}=k_{i,j}+1$, $j\in\{1,\dots,D\}$ we have $q_{i,1}+\dots+q_{i,D}=k+D$ and we are dealing with a $(k+D)$-composition. Thus $N(k,D)=|[k]|=\binom{k+D-1}{D-1}=\binom{k+D-1}{k}$ --consider all the possible permutations of
$k+(D-1)$ elements out of which $k$ and $(D-1)$ are repeated--.  Two sequences that differ in the order of their terms define different weak compositions of their sum, while they are considered to define the same partition of that number. Every integer has finitely many distinct compositions.

A multiset \cite{multiset} is 2-tuple $(I, M)$ where $I$ is some set, the underlying set of elements, and the multiplicity $M: I \to \N$ is a function from $I$ to the set of positive integers; for each $a\in I $ the multiplicity or number of occurrences is $M(a)$.  For an indexed family, $(a_i)$, where $i$ is some index-set, we define a multiset $\{a_i\}$, where the multiplicity of any element $a$ is the number of indices $i$ such that $a_i =a$.
A form of describing a multiset that is used in this article is considering non-negative integers $(a_i)_{i=1}^k$ such that $1\leq a_1\leq \dots\leq a_k\leq D$, where repetitions are allowed, e.g.  for $k=5$ we could have $a_1=a_2=a_3<a_4=a_5$, denoting $a_1=a$ and $a_4=b$ we are dealing with the multiset $\{a,a,a,b,b\}$ being three the multiplicity of $a$, $M(a)=3$, and two the multiplicity of $b$, $M(b)=2$.

\subsection{Symmetric tensor powers and symmetric algebras}
We give a brief description of notions and results regarding symmetric algebras, for further information we refer the reader to \cite{Comon,federer,knapp,vinberg}.
\subsubsection{Symmetric tensors}
A symmetric tensor of order $k$  is a tensor of order $k$ that is invariant under a permutation of its vector arguments:
\begin{align*}
T(u_1,\dots,u_k) =\tau_\sigma T(u_1,\dots,u_k) = T(u_{\sigma 1},\dots,u_{\sigma k})
\end{align*}
for every  $\sigma\in\mathfrak S_k$, being $\mathfrak S_k$ the symmetric group of $k$ letters. The coefficients of a symmetric tensor of order $k$ satisfy
$T_{i_1,\dots, i_k} = T_{i_{\sigma 1},\dots, i_{\sigma k }}$.
The space of symmetric tensors of order $k$ on $\R^D$ is naturally isomorphic to the dual of the space of homogeneous multivariate polynomials of total degree $k$ and the graded vector space of all symmetric tensors can be naturally identified with the symmetric algebra $\operatorname{Sym}(\R^D)$.

The symmetric part of  a tensor $T\in \big(\R^D\big)^{\otimes k}$ of order $k$ is defined by
\begin{align}\label{symmetrization}
  \operatorname{Sym}\, T = \frac{1}{k!}\sum_{\sigma\in\mathfrak{S}_k} \tau_\sigma T,
\end{align}
where the summation extends over the symmetric group on $k$ symbols.
If the tensor coefficients of the tensor are $T_{i_1,i_2,\dots, i_k}$ those of the symmetric part are $T_{(i_1,i_2,\dots, i_k)} = \frac{1}{k!}\sum_{\sigma\in \mathfrak{S}_k} T_{i_{\sigma 1},i_{\sigma 2},\dots, i_{\sigma k}}$
For two arbitrary pure tensors  $T=v_1\otimes v_2\otimes\cdots \otimes v_r$ the corresponding symmetrization or symmetric part  is given by
$v_1\odot v_2\odot\cdots\odot v_r \equiv \text{Sym}(v_1\otimes v_2\otimes\cdots \otimes v_r)\coloneq \frac{1}{r!}\sum_{\sigma\in\mathfrak{S}_r} v_{\sigma 1}\otimes v_{\sigma 2}\otimes\cdots\otimes v_{\sigma r}$.
Given two tensors $T_i\in\text{Sym}^{k_i}(\R^D)$, $i\in\{1,2\}$, the symmetrization operator allows us to define $T_1\odot T_2 = \operatorname{Sym}(T_1\otimes T_2)\in\operatorname{Sym}^{k_1+k_2}(\R^D)$.
As the resulting product is commutative and associative some authors write  $T_1T_2 = T_1\odot T_2$. Given a vector $v\in\R^D$ we will use the exponential notation
$v^{\odot k} = \underbracket{v \odot v \odot \cdots \odot v}_{k\text{ times}}=\underbracket{v \otimes v \otimes \cdots \otimes v}_{k\text{ times}}=v^{\otimes k}$. 
\subsubsection{Symmetric tensor powers and symmetric algebra}\label{spowers}
Symmetric tensor powers $\operatorname{S}^k(\R^D)=\big(\R^D)^{\odot k}$  are generated by the so called decomposable (or simple or pure) symmetric tensors $u_1\odot\cdots\odot u_k$, where $u_1,\dots,u_k\in\R^D$. Given a basis $\{\ee_a\}_{a=1}^D$ we  can construct an explicit linear basis of $S^k(\R^D)$ using the concept of multiset. The mentioned  linear basis for the $k$-th symmetric power is $\{\ee_{a_1}\odot\cdots\odot \ee_{a_k}\}_{\substack{1\leq a_1\leq\cdots\leq a_k\leq D\\ k\in\Z_+}}$, or in terms of multisets $I=\{a_1,\dots,a_p\}$ with multiplicities $M(a_i)$, such that $M(a_1)+\dots+M(a_p)=k$ we have $\{\ee_{a_1}^{\odot M(a_1)}\odot\cdots\odot \ee_{a_p}^{\odot M(a_p)}\}_I$.
The dual space of the symmetric powers happens to be isomorphic to the set of symmetric multilinear functionals on $\R^D$, $\big(\text{S}^k(\R^D)\big)^*\cong S((\R^D)^k,\R)$.

 The number of multisets of cardinality $k$, with elements taken from a finite set of cardinality $D$, is known  as the multiset coefficient $\textstyle\big(\!{D\choose k}\!\big)$, see \cite{stanley},  which resembles the  binomial coefficients and we say  ``$D$ multichoose $k$" instead of ``$D$ choose $k$" for $\tbinom nk$. We have that $ \textstyle\big(\!{D\choose k}\!\big)= {D+k-1 \choose k} = \frac{(D+k-1)!}{k!\,(D-1)!} = {D(D+1)(D+2)\cdots(D+k-1)\over k!}$,
  and the number of such multisets is the same as the number of subsets of cardinality $k$ in a set of cardinality $D + k - 1$. Thus, $\dim\text{S}^k(\R^D)\equiv |[k]|= \textstyle\big(\!{D\choose k}\!\big)$.

  We can define a surjective map $\pi:\big(\R^D\big)^{\otimes k}\rightarrow \operatorname{S}(\R^D)$ by the symmetrization $\pi(u_1\otimes\dots\otimes u_k )\coloneq  u_1\odot\dots\odot u_k$. This map has a section, i.e. an injective map $\imath: \operatorname{S}(\R^D)\rightarrow\big(\R^D\big)^{\otimes k}$ such that $\pi\circ\imath=\operatorname{id}$. The maps gives $\imath(u_1\odot\dots\odot u_k)=\operatorname{Sym}(u_1\otimes\dots\otimes u_k)$ so that its image is just the space of symmetric tensors just discussed. Moreover, for the symmetrization of tensors of \eqref{symmetrization} we have  $\operatorname{Sym}\coloneq \imath\circ\pi:   \big(\R^D\big)^{\otimes k}\rightarrow \big(\R^D\big)^{\otimes k}$; notice also that this symmetrization is a projection $\operatorname{Sym}^2=\operatorname{Sym}$ so that
  \begin{align*}
  \big(\R^D\big)^{\otimes k}=  \operatorname{Sym}\Big(\big(\R^D\big)^{\otimes k}\Big)\oplus \operatorname{ker}(\operatorname{Sym})=
\imath\Big(\big(\R^D\big)^{\odot k}\Big)\oplus \operatorname{ker}(\operatorname{Sym}).
  \end{align*}

The direct sum $\text{S}(\R^D)\coloneq \oplus_{k\geq 0}\text{S}^k(\R^D)$ is the symmetric algebra of $\R^D$, which is commutative and associative.  The symmetric algebra of $\R^D$ can be constructed as the tensor algebra $\operatorname{T}(\R^D)$ quotient with the ideal generated tensors of the form $t_1(\x\otimes\y-\y\otimes\x)t_2$ with $t_1,t_2$ homogeneous tensors of arbitrary degree.

\subsubsection{Dot product} There is an interesting inner product in the symmetric tensor power $\operatorname{S}^k(\R^D)$ which is a symmetric definite positive bilinear form, see appendix A in \cite{knapp}, and also \cite{Comon}. It is given by the linear extension of the following definition for decomposable symmetric tensor powers
\begin{align}\label{permanents}
  \langle \boldsymbol u_1\odot\dots\odot \boldsymbol u_k, \boldsymbol v_1\odot\dots\odot \boldsymbol v_k \rangle^{(k)}= \frac{1}{k!}\sum_{\sigma\in\mathfrak S_k}\prod_{a=1}^k \boldsymbol u_a\cdot\boldsymbol v_{\sigma a}=\frac{1}{k!}\operatorname{perm}\PARENS{\begin{matrix}
 \boldsymbol u_1\cdot\boldsymbol v_1 &\cdots& \boldsymbol u_1\cdot\boldsymbol v_k\\
 \vdots & & \vdots\\
  \boldsymbol u_k\cdot\boldsymbol v_1 &\cdots &\boldsymbol u_k\cdot\boldsymbol v_k
  \end{matrix}},
\end{align}
where $\boldsymbol u_a,\boldsymbol v_a\in\R^D$ and we have used permanents \cite{minc,muir}.

  We introduce the semi-infinite multinomial matrix $\mathcal M=\diag(\mathcal M_{[0]},\mathcal M_{[1]},\dots)$ where
  \begin{align}\label{multinomialmatrix}
   \mathcal M_{[k]}\coloneq\diag\Big({k\choose \q_1},\dots,{k\choose \q_{|[k]|}}\Big) \in\R^{|[k]|\times|[k]|},
  \end{align}
   are  diagonal matrices with coefficients the multinomial numbers
 \begin{align*}
  {k\choose \q_j}=&\frac{k!}{\prod_{a=1}^D\alpha_{j,a}!}, & j\in&\{1,\dots,|[k]|\}.
 \end{align*}
According to the proof of Corollary A.24 of \cite{knapp}  for the canonical basis vectors $\ee^{\q_i}=\ee_1^{\odot \alpha_{i,1}}\odot\dots\odot\ee_D^{\odot \alpha_{i,D}}$, for all $\q_i\in[k]\coloneq \big\{\alpha_{i,1}\ee_1+\dots+\alpha_{i,D}\ee_D\in\Z_+^D \text{ with $\alpha_{i,1}+\dots+\alpha_{i,D}=k$}\big\}$\footnote{In \cite{federer} the set $[k]$ is denoted by $\Xi(D,k)$}, we have the following metric coefficients
\begin{align*}
  \langle \boldsymbol e_{\q_i},\boldsymbol e_{\q_j}\rangle^{(k)}=&\delta_{i,j}{k\choose \q_i }^{-1}, & i,j\in&\big\{1,\dots, \textstyle\big(\!{D\choose k}\!\big)\big\}.
\end{align*}
Hence, the interior product $\langle\cdot,\cdot\rangle^{(k)}:\operatorname{S}^k(\R^D)\times \operatorname{S}^k(\R^D)\rightarrow \R$ is given by
\begin{align}\label{explicit interior product symmetric tensor power}
\Big\langle\sum_{i=1}^{|[k]|}a_{\kk_i}\ee^{\kk_i}, \sum_{i=1}^{|[k]|}b_{\kk_i}\ee^{\kk_i}\Big\rangle^{(k)}=\PARENS{\begin{matrix}
    a_{\kk_1}\\\vdots\\a_{\kk_{|[k]|}}
  \end{matrix}}^\top
  \mathcal M_{[k]}^{-1} \PARENS{\begin{matrix}
    b_{\kk_1}\\\vdots\\b_{\kk_{|[k]|}}
  \end{matrix}}.
\end{align}

 It is easy to check that
\begin{align}
  \langle \boldsymbol u_1\odot\dots\odot \boldsymbol u_k,\boldsymbol v^{\odot k}\rangle^{(k)} =&\prod_{a=1}^k(\boldsymbol u_a\cdot\boldsymbol v),\notag\\
    \langle \boldsymbol u^{\odot k},\boldsymbol v^{\odot k}\rangle^{(k)}=&(\boldsymbol u\cdot\boldsymbol v)^k.\label{symmetric interior product pure}
\end{align}
We remark that all these developments are connected with Quantum Physics.
Indeed, when Quantum Mechanics of large systems describes sets of an arbitrary number of  bosons, if $\mathscr H$ is the Hilbert space for the states of a single particle, then $\operatorname{S}^k(\mathscr H)$ will describe the pure states of $k$ identical bosons, and in general the symmetric algebra $\operatorname{S}(\mathscr H)$ is the Hilbert space of pure states of an arbitrary number of bosons, see \cite{thirring}. Thus,  multivariate polynomials are connected, naively if you want, with bosons such that its single particle Hilbert space of pure states is $\R^D$ (which is a real Hilbert space, that still has a physical meaning for even dimensions $D$).

\subsubsection{The shift matrices and symmetric tensors}
The shift matrices  have a natural description in the symmetric algebra, as well.
\begin{pro}
In the symmetric tensor power setting  these blocks can be thought as
\begin{align*}
(\Lambda_a)_{[k-1],[k]}: \operatorname{S}^k(\R^D)\to\operatorname{S}^{k-1}(\R^D)
\end{align*}
 with
\begin{align*}
  (\Lambda_a)_{[k-1],[k]}
  \ee^{\q_i}= \ccases{
   \ee^{\q_i-\ee_a}, & \text{if }\q_i\cdot\ee_a\neq0,\\
   0, &\text{if }\q_i\cdot\ee_a=0.
  }
\end{align*}
\end{pro}

Following \S 1.10.3 of \cite{federer} we introduce interior multiplications. First we consider dual linear space $(\R^D)^*$ of linear functionals $\R^D\to\R$, and the dual basis $\{\boldsymbol\omega_a\}_{a=1}^D\subset (\R^D)^*$, i.e.
$\boldsymbol\omega_a(\ee_b)=\delta_{a,b}$. The symmetric tensors powers $\boldsymbol\omega^{\q}=\boldsymbol\omega_1^{\odot \alpha_1}\odot\dots\odot\boldsymbol\omega_D^{\odot\alpha_D}$ give rise to a linear basis $\{\boldsymbol\omega^\q\}_{\q\in[k]}$ of $\operatorname{S}^{k}\left((\R^D)^*\right)$ dual to  $\big\{\frac{1}{\q!}\ee^\q\big\}_{\q\in[k]}$
\begin{align*}
\boldsymbol\omega{^\q} (\ee^{\boldsymbol\beta}) =\q!\delta_{\boldsymbol\alpha,\boldsymbol\beta}.
\end{align*}
Interior multiplications or right contractions are maps
\begin{align*}
\intprod: \operatorname{S}^{k}(\R^D)\times\operatorname{S}^{k'}\big((\R^D)^*\big)\to  \operatorname{S}^{k-k'}(\R^D),
\end{align*}
with
\begin{align*}
\ee^{\boldsymbol\alpha}\intprod\boldsymbol\omega^{\q'}\coloneq
  \ccases{
    0,  & \alpha_a<\alpha'_a, \text{ for some $a\in\{1,\dots,D\}$},
    \\
    \frac{\boldsymbol\alpha!}{(\boldsymbol\alpha-\q')!}\ee^{\boldsymbol\alpha-\q'},  & \alpha_a\geq\alpha'_a, \forall a\in\{1,\dots,D\}.
  }
\end{align*}
When we take $k'=1$ and consider the linear maps $\intprod\boldsymbol\omega_a: \operatorname{S}^{k}(\R^D)\to  \operatorname{S}^{k-1}(\R^D)$ we find
\begin{align*}
\ee^{\boldsymbol\alpha}\intprod\boldsymbol\omega_a\coloneq (\boldsymbol\alpha\cdot\ee_a) \ee^{\boldsymbol\alpha-\ee_a}.
\end{align*}

In terms of the number operators\footnote{If we follow the quantum interpretation of our symmetric algebra as system of bosons with the single particle described by $\R^D$, we could understand $N_a$ as the number operator of particles in state $\ee_a$} $N_a:\operatorname{S}^k(\R^D)\to\operatorname{S}^k(\R^D) $ given by $N_a(\smashoperator{\sum_{\q\in[k]}}c_{\q}\ee^\q)=\sum_{\q\in[k]}(\q\cdot\ee_a)c_{\q}\ee^\q$ we can write
\begin{align*}
  (\Lambda_a)_{[k-1],[k]}(T)=&(N_a^{-1}T) \intprod\boldsymbol\omega_a, &\forall T\in&(\R^D)^{\odot k},
\end{align*}
as the composition of a number operator with an interior multiplication.\footnote{Following with the boson analogy we have the destruction operators $a_b=N_b^{-1/2}\Lambda_b$  \cite{thirring}.}

  \section{Complements of linear algebra: Pseudo-inverses, Schur complements and quasi-determinants}\label{qd}\settocdepth{section}
\subsection{The Moore--Penrose pseudo-inverse}\label{appendix pseudoinverse}
As this paper requires  pseudo-inverses a number of times we decided  to include  a short resume on the subject, for more information see \cite{pseudoinverse0}. Given a  rectangular matrix $M\in\R^{m\times n}$ its Moore--Penrose pseudo-inverse $M^+\in\R^{n\times m}$ \cite{pseudoinverse1,pseudoinverse2,pseudoinverse3}  is a generalized inverse, i.e.,
\begin{align*}
M M^+M =& M, &M^+M M^+ =& M^+,
\end{align*}
that in addition satisfies
\begin{align*}
(MM^+)^\top =& MM^+, &    (M^+M)^\top =& M^+M.
\end{align*}
Let us say that any matrix has a generalized inverse and  a unique pseudo-inverse.
Obviously, if $M$ is invertible then $M^+=M^{-1}$, any rectangular zero matrix has as it pseudo-inverse its transpose. The pseudo-inverse operation is idempotent $(M^+)^+=M$ and $(M^\top)^+=(M^+)^\top$.

The square  matrices $P \coloneq  MM^+\in\R^{m\times m}$ and $Q \coloneq  M^+M\in\R^{n\times n}$
 are orthogonal projection operators, i.e. $P = P^\top$ and $Q = Q^\top$, $P^2 = P$, $Q^2 = Q$.
Moreover, we have
\begin{enumerate}
\item $PM=M=MQ$ and $M^+P=M^+=QM^+$.
\item $\operatorname{Im}(M)=\operatorname{ker}(M)^\bot=\operatorname{Im}(P)=\operatorname{ker}(\I-P)^\bot$.
\item $\operatorname{Im}(M^\top)=\operatorname{ker}(M^\top)^\bot=\operatorname{Im}(Q)=\operatorname{ker}(\I-Q)^\bot$.
\end{enumerate}

When  $P=M^\top M$  is invertible, e.g. when we have full column rank, there is a  unique matrix  $M^+$  which satisfies  these properties and is given by
$ M^+ = (M^\top M)^{-1} M^\top$, which in addition is a left inverse. When $Q=MM^\top$  is invertible, e.g. when we have full row rank,  then $ M^+ = M^\top (MM^\top)^{-1}$; that moreover is  a right inverse. In these cases $P=M^\top M$ or $Q=MM^\top$ are denominated  \emph{correlation matrices}, respectively.

\subsection{Schur complements}\label{schur}

Given $M=\left(\begin{smallmatrix}
    A & B\\
    C & D
   \end{smallmatrix}\right)$ in block form the Schur complement with
respect to $A$ (if $\det A\neq 0$) is
\begin{align*}
  \operatorname{SC}\PARENS{\begin{array}{c|c}
    A & B\\\hline
    C & D
   \end{array}}\equiv M / A\coloneq D-CA^{-1}B.
\end{align*}
The Schur complement with
respect to $D$ (if $\det D\neq 0$)  is \begin{align*}
  \operatorname{SC}_D\PARENS{\begin{array}{c|c}
    A & B\\\hline
    C & D
   \end{array}}\equiv M / D\coloneq A-BD^{-1}C.
\end{align*}
Observe that we have the block Gauss factorization
\begin{align*}
  \PARENS{\begin{matrix}
    A&B\\C&D
  \end{matrix}}&=
  \PARENS{\begin{matrix}
    \I & 0\\
    CA^{-1} & \I
  \end{matrix}}
  \PARENS{\begin{matrix}
   A & 0\\
   0 &  M/ A
  \end{matrix}}
  \PARENS{\begin{matrix}
    \I &A^{-1}B\\
    0 &\I
  \end{matrix}}\\&=
   \PARENS{\begin{matrix}
    \I & BD^{-1}\\
    0 & \I
  \end{matrix}}
  \PARENS{\begin{matrix}
   M/ D & 0\\
   0 &  D
  \end{matrix}}
  \PARENS{\begin{matrix}
    \I &0\\
    D^{-1}C &\I
  \end{matrix}}
\end{align*}
implies the Schur determinant formula $\det M=\det(A)\det(M/ A)$.
This is in fact the Schur lemma in a disguise form, in fact Schur lemma in \cite{schur} assumes that $[A,C]=0$ so that $\det M=\det(AD-BC)$.
In terms of the Schur complements we have the following well known expressions
for the inverse matrices
\begin{align}%
\notag
M^{-1}&=  \PARENS{\begin{matrix}
    \I &-A^{-1}B\\
    0 &\I
  \end{matrix}}
  \PARENS{\begin{matrix}
   A^{-1} & 0\\
   0 &  (M/ A)^{-1}
  \end{matrix}}\PARENS{\begin{matrix}
    \I & 0\\
   - CA^{-1} & \I
  \end{matrix}}
\\
\label{inverseSchur}
&=
\PARENS{\begin{matrix}
A^{-1}+A^{-1}B(M/ A)^{-1}CA^{-1}
 &-A^{-1}B(M/ A)\\
-(M/ A)^{-1}CA^{-1} &(M/ A)^{-1}
\end{matrix}}\\
\notag
&=  \PARENS{\begin{matrix}
    \I &0\\
   - D^{-1}C &\I
  \end{matrix}}
     \PARENS{\begin{matrix}
   (M/ D)^{-1} & 0\\
   0 &  D^{-1}
  \end{matrix}}
    \PARENS{\begin{matrix}
    \I & -BD^{-1}\\
    0 & \I
  \end{matrix}}\\
  \notag
  &=\PARENS{\begin{matrix}
     (M/ D)^{-1}  & -  (M/ D)^{-1} BD^{-1}\\
     -D^{-1}B (M/ D)^{-1}  &D^{-1}+D^{-1} (M/ D)^{-1} BD^{-1}
  \end{matrix}}.
\end{align}
The two expressions found for the inverse of $M$ are known as  the Matrix Inversion Lemma in Linear Estimation Theory \cite{let} and as Sherman--Morrison--Woodbury formula in Linear Algebra \cite{la}. If both $A$ and $D$ are invertible we deduce that  $M/ A$ is invertible if and only if $M/ D$ is invertible.

\subsection{Quasi-determinants and the heredity principle}
Given  any partitioned matrix
\begin{align}\label{partitioned}
  A=\PARENS{\begin{matrix}
    A_{1,1} & A_{1,2}&\dots &A_{1,k}\\
    A_{2,1} & A_{2,2}&\dots &A_{2,k}\\
   \vdots &\vdots &&\vdots\\
     A_{k,1} & A_{k,2}&\dots &A_{k,k}
  \end{matrix}}
\end{align}
where $A_{i,j}\in\R^{m_i\times m_j}$ for $i,j\in\{1,\dots,k-1\}$, and $A_{k,k}\in\R^{\kappa_1\times\kappa_2}$,$A_{i,k}\in\R^{m_i\times\kappa_2}$  and
$A_{k,j}\in\R^{\kappa_1\times m_j}$, we are going to define its quasi-determinant \emph{\`{a} la Olver} recursively.
We start with $k=2$, so that $A=\begin{psmallmatrix}  A_{1,1} & A_{1,2}\\ A_{2,1} & A_{2,2} \end{psmallmatrix}$, in this case the first quasi-determinant
$\Theta_1(A)\coloneq A/ A_{1,1}$; i. e., a Schur complement which requires $\det A_{1,1}\neq 0$. The notation of Olver
   is different to that of the Gel'fand school where $\Theta_1(A)=|A|_{2,2}=\begin{vsmallmatrix}  A_{1,1} & A_{1,2}\\ A_{2,1} & {\tiny\boxed{ A_{2,2}}} \end{vsmallmatrix}$.
There is another quasi-determinant $\Theta_2(A)=A/ A_{22}=|A|_{1,1}=\begin{vsmallmatrix} {\tiny\boxed{A_{1,1}}} & A_{1,2}\\ A_{2,1} & A_{2,2} \end{vsmallmatrix}$, the other Schur complement, and we need $A_{2,2}$ to be a invertible square matrix. Other quasi-determinants that can be considered for regular square blocks are
$\begin{vsmallmatrix}
A_{1,1} &A_{1,2}\\
{\tiny\boxed{A_{2,1}}}&A_{2,2}
\end{vsmallmatrix}
$
 and $\begin{vsmallmatrix} A_{1,1} &{\tiny\boxed{A_{1,2}}}\\ A_{2,1} & A_{2,2} \end{vsmallmatrix}$. Notice that this last two quasi-determinants are out of the scope of the paper,  as the partitioned matrix considered here have rectangular off diagonal blocks and therefore are not invertible.

Following \cite{olver} we remark that  quasi-determinantal reduction is a commutative
operation. This is the heredity principle formulated by Gel'fand and Retakh \cite{gelfand,quasidetermiant9}: quasi-determinants of quasi-determinants are quasi-determinants.
Let us illustrate this by reproducing a nice example discussed in \cite{olver}.  We consider the matrix
\begin{align*}
  A=\PARENS{\begin{array}{c|c:c}
    A_{1,1} & A_{1,2}&A_{1,3}\\\hline
    A_{2,1} & A_{2,2}&A_{2,3}\\\hdashline
     A_{3,1} & A_{3,2}&A_{3,3}
  \end{array}}
  \end{align*}
and take the quasi-determinant with respect  the first diagonal block, which we define as the Schur complement indicated by the non dashed lines, to get
\begin{align*}
  \Theta_1(A)&=\begin{vmatrix}  A_{11,1} &\begin{matrix} A_{1,2} &A_{1,3}\end{matrix}\\\begin{matrix}
  A_{2,1}\\A_{3,1}
  \end{matrix} &\boxed{\begin{matrix}A_{2,2}&A_{2,3}
\\ A_{3,2}&  A_{3,3}\end{matrix}}
\end{vmatrix}=\PARENS{\begin{matrix}
     A_{2,2}&A_{2,3}\\
      A_{3,2}&A_{3,3}
  \end{matrix}}-
  \PARENS{\begin{matrix}
     A_{2,1} \\
     A_{3,1}
  \end{matrix}}A_{1,1}^{-1}\PARENS{\begin{matrix}
     A_{1,2}&A_{1,3}
  \end{matrix}}
  \\&=
 \PARENS{\begin{array}{c:c}
  A_{2,2}- A_{2,1}A_{1,1}^{-1}A_{1,2}  & A_{2,3}-A_{2,1}A_{1,1}^{-1}A_{1,3}\\\hdashline
    A_{3,2}-A_{3,1}A_{1,1}^{-1} A_{1,2}& A_{3,3}-A_{3,1}A_{1,1}^{-1} A_{1,3}
  \end{array}},
\end{align*}
  a matrix with blocks with subindexes involving 2 and 3 but not 1.  Notice also, that us we are allowed to take blocks of different sizes we have taken the quasi-determinant with respect to a bigger block, composed of two rows and columns of basic blocks. This is the Olver's generalization of Gel'fand's et al construction.
  Now,  we take the quasi-determinant given by the Schur complement as indicated by the dashed lines, to get
  \begin{align}\label{1.2}
    \Theta_2(\Theta_1(A))&=\begin{vmatrix}
  A_{2,2}- A_{2,1}A_{1,1}^{-1}A_{1,2}  & A_{2,3}-A_{2,1}A_{1,1}^{-1}A_{1,3}\\
    A_{3,2}-A_{3,1}A_{1,1}^{-1} A_{1,2}& \boxed{A_{3,3}-A_{3,1}A_{1,1}^{-1} A_{1,3}}
    \end{vmatrix}\\
    &=A_{3,3}-A_{3,1}A_{1,1}^{-1} A_{1,3}-  \big(A_{3,2}-A_{3,1}A_{1,1}^{-1} A_{1,2}\big)\big(A_{2,2}- A_{2,1}A_{1,1}^{-1}A_{1,2}\big)^{-1}\big(A_{2,3}-A_{2,1}A_{1,1}^{-1}A_{1,3}\big).
  \end{align}
  We are ready to compute, for the very same matrix
 \begin{align}\label{a1}
  A=\PARENS{\begin{array}{cc|c}
    A_{1,1} & A_{1,2}&A_{1,3}\\
    A_{2,1} & A_{2,2}&A_{2,3}\\\hline
     A_{3,1} & A_{3,2}&A_{3,3}
  \end{array}},
  \end{align}
   the quasi-determinant associated to the two first diagonal blocks, that we label as $\{1, 2\}$; i.e., the Schur complement indicated by the non-dashed lines in \eqref{a1}, to get
\begin{align*}
    \Theta_{\{1,2\}}(A)=&\begin{vmatrix}  A_{1,1} & A_{1,2} & A_{1,3}\\ A_{2,1} & A_{2,2} &A_{2,3}\\
  A_{1,3}& A_{2,3} & \boxed{A_{3,3}}
\end{vmatrix}\\=&A_{3,3}- \PARENS{\begin{matrix}
   A_{3,1} & A_{3,2}
    \end{matrix}}
    \PARENS{\begin{matrix}
    A_{1,1} & A_{1,2}\\
    A_{2,1} & A_{2,2}
    \end{matrix}}^{-1}
    \PARENS{\begin{matrix}
A_{1,3}\\
A_{2,3}
    \end{matrix}}
      \end{align*}
      But recalling \eqref{inverseSchur}
    \begin{align*}
    \PARENS{\begin{matrix}
    A_{1,1} & A_{1,2}\\
    A_{2,1} & A_{2,2}
    \end{matrix}}^{-1}= \PARENS{\begin{matrix}
A_{1,1}^{-1}+A_{1,1}^{-1}A_{1,2}( A_{2,2}-A_{2,1}A_{1,1}^{-1}A_{1,2})^{-1}A_{2,1}A_{1,1}^{-1}
 &-A_{1,1}^{-1}A_{1,2}( A_{2,2}-A_{2,1}A_{1,1}^{-1}A_{1,2})\\
-( A_{2,2}-A_{2,1}A_{1,1}^{-1}A_{1,2})^{-1}A_{2,1}A_{1,1}^{-1} &( A_{2,2}-A_{2,1}A_{1,1}^{-1}A_{1,2})^{-1}
\end{matrix}}
\end{align*}
we get
\begin{align*}
      \Theta_{\{1, 2\}}(A)=&A_{3,3}-A_{3,1}A_{1,1}^{-1}A_{1,3}+A_{3,1}A_{1,1}^{-1}A_{1,2}
      \big(A_{2,2}-A_{2,1}A_{1,1}^{-1}A_{1,2}\big)^{-1}A_{2,1}A_{1,1}^{-1}A_{1,3}\\&-
      A_{3,2}\big(A_{2,2}-A_{2,1}A_{1,1}^{-1}A_{1,2}\big)^{-1}A_{2,1}A_{1,1}^{-1}A_{1,3}-A_{3,1}A_{1,1}^{-1}A_{1,2}
      \big(A_{2,2}-A_{2,1}A_{1,1}^{-1}A_{1,2}\big)^{-1}A_{2,3}\\
      &+A_{3,2}\big(A_{2,2}-A_{2,1}A_{1,1}^{-1}A_{1,2}\big)^{-1}A_{2,3}
\end{align*}
  which is identical to \eqref{1.2}, so that
\begin{align*}
  \Theta_2(\Theta_1(A))=\Theta_{\{1, 2\}}(A).
\end{align*}

 Given any set $I=\{i_1,\dots,i_m\}\subset\{1,\dots,k\}$ the heredity principle allows us to define the quasi-determinant\footnote{In \cite{olver} it is defined as the Schur complement with respect a big block built up by the blocks determined by  the indices $I$. }
 \begin{align*}
   \Theta_I(A)=\Theta_{i_1}(\Theta_{i_2}(\cdots\Theta_{i_m}(A)\cdots))
 \end{align*}
 and the $\ell$-th quasi-determinant is
 \begin{align*}
   \Theta^{(\ell)}(A)=\Theta_{ \{1,\dots,\ell-1,\ell+1,\dots,k\}}(A)=|A|_{\ell,\ell}=\begin{vmatrix}
    A_{1,1} & A_{1,2}&\dots&A_{1,\ell}&\dots &A_{1,k}\\
    A_{2,1} & A_{2,2}&\dots&A_{2,\ell}&\dots &A_{2,k}\\
   \vdots &\vdots & &\vdots&&\vdots\\
   \\
   A_{\ell,1} & A_{\ell,2}&\dots &\boxed{A_{\ell,\ell}}&\dots&A_{\ell,k}\\
  \vdots &\vdots & &\vdots&&\vdots\\
     A_{k,1} & A_{k,2}&\dots& A_{k,\ell}&\dots&A_{k,k}
  \end{vmatrix}.
 \end{align*}
 The last quasi-determinant is denoted by
 \begin{align*}
   \Theta_*(A)=\Theta^{(k)}(A)= |A|_{k,k}=\begin{vmatrix}
    A_{1,1} & A_{1,2}&\dots &A_{1,k}\\
    A_{2,1} & A_{2,2}&\dots &A_{2,k}\\
   \vdots &\vdots &&\vdots\\
     A_{k,1} & A_{k,2}&\dots &\boxed{A_{k,k}}
  \end{vmatrix}.
 \end{align*}
 \subsection{Quasi-determinants and Gauss--Borel factorization}

 An important application of quasi-determinants presented in \cite{olver} is the characterization of the factors of the block Gauss--Borel factorization of a partitioned matrix $A$ as in  \eqref{partitioned} (in the case of interest in this paper a Cholesky factorization) in terms of quasi-determinants of $A$. To present this result we need to introduce for two sets of indices $\{i_1,\dots,i_m\}$ and $\{j_1,\dots,j_m\}$ subset, with $m$ elements, of $\{1,\dots,k\}$
\begin{align*}
  A^{i_1\dots i_m}_{j_1\dots j_m}
  =\PARENS{\begin{matrix}
    A_{i_1,j_1} & \dots &A_{i_1,j_m}\\
    \vdots & &\vdots\\
    A_{i_m,j_1}& \dots &A_{i_m,j_m}
  \end{matrix}}.
\end{align*}

\begin{theorem}[Theorem 3 in \cite{olver}]
A regular block matrix as  in \eqref{partitioned} factors as
\begin{align*}
  A= L D V
\end{align*}
  with $L=(L_{i,j})$, $D=\diag(D_1,\dots, D_k)$ and $V=(V_{i,j})$, where
  \begin{align*}
    L_{i,j}=&\ccases{
      0, & i<j,\\
      \Theta_*(A^{12\dots j-1,i}_{12\dots j-1,j})\Theta_*(A^{12\dots j}_{12\dots j})^{-1},& i\geq  j,}\\
      D_j=&\Theta_*(A^{12\dots j}_{12\dots j}),\\
      V_{i,j}=&\ccases{
      0, & i>j,\\
    \Theta_*(A^{12\dots i}_{12\dots i})^{-1}  \Theta_*(A^{12\dots i-1,i}_{12\dots i-1,j}),& i\leq  j.
    }
  \end{align*}
  Regularity of $A$ requires invertibility of $\Theta_*(A^{12\dots j}_{12\dots j})$ for $j=1,\dots,k$.
\end{theorem}
For a symmetric case, $A=A^\top$, we have 
\begin{align*}
  \Theta_*(A^{12\dots i}_{12\dots i})&=  \Theta_*(A^{12\dots i}_{12\dots i})^\top, &
  \Theta_*(A^{12\dots i-1,i}_{12\dots i-1,j}) &=   \Theta_*(A^{12\dots i-1,j}_{12\dots i-1,i})^\top.
\end{align*}

\section{Several complex variables}\label{scv}\settocdepth{section}

In this paper we  discuss multivariate second kind functions in the realm of the block Cholesky factorization and for that aim some facts regarding complex analysis in several variables is needed. Here we just recall them, see for example \cite{begehr,scv0,scv1,scv2} for more information
\begin{enumerate}
\item Given the vector  $\boldsymbol r=(r_1,\dots,r_D)^\top\in\R_+^D$, we consider the polydisk
\begin{align*}
  \Delta(\boldsymbol r)=\{\boldsymbol z=(z_1,\dots,z_D)^\top:|z_i|<r_i, i=1,\dots,D\}\subset\C^D
\end{align*}
centered at the origin of polyradius $\boldsymbol r$. Its distinguished boundary is the $D$-dimensional torus
\begin{align*}
\T^D(\boldsymbol r)=\{\boldsymbol z\in\C^D: |z_i|=r_i, i=1,\dots,D\}.
\end{align*}
Recall that the border of the polydisk $\Gamma=\partial\Delta$ splits in $D$ sets of dimension $2D-1$, being the distinguished border its skeleton; i.e. the intersection of all them. The distinguished border is also known as Shilov border.
\item Given two polyradii $\boldsymbol r$ and $\boldsymbol R$ we construct the associated polyannulus centered at the origin
\begin{align*}
  A^D(\boldsymbol r, \boldsymbol R)\coloneq \{\z\in\C^D: r_i< z_i< R_i, i=1,\dots, D\}.
\end{align*}
\item A set $A\subset \C^n$ is a complete Reinhardt domain if the unit polydisk acts on it by componentwise multiplication.
    \item Any set $A\subset\C^D$ is called Reinhardt ($D$-circled)  if for each $\boldsymbol \lambda\coloneq (\Exp{\operatorname{i}\theta_1},\dots,\Exp{\operatorname{i}\theta_D})\in\T^D$ with $\theta_i\in[0,2\pi)$ for every $\boldsymbol c\in A$ we have that $(\Exp{\operatorname{i}\theta_1}c_1,\dots, \Exp{\operatorname{i}\theta_D}c_D)^\top\in A$; i.e., $\T^D$ acts on $A$ componentwise.
  \item If $\mathscr D\subset\C^D$ is the domain of convergence of a Laurent series $L(\z)$, then for any $\boldsymbol c=(c_1,\dots,c_D)^\top\in\mathscr D$ we have that $\T^D(|c_1|,\dots,|c_D|)\subset\Ds$.
  Thus, the domain of convergence is a Reinhardt ($D$-circled) domain.
  \item The domain of convergence $\Ds_{\ele_a}$ is logarithmically convex; i. e.,  the set
  \begin{align*}
\log\Ds_{\ele_a}\coloneq \{(\log|z_1|,\dots,|z_D|): (z_1,\dots,z_D^\top\in\Ds_{\ele_a}\}
  \end{align*}
 is convex (given any pair of  points $\boldsymbol c_1,\boldsymbol c_2\in\Ds_{\ele_a}$ the segment $[\boldsymbol c_1,\boldsymbol c_2]\coloneq \{(1-t)\boldsymbol c_1+t\boldsymbol c_2: t\in[0,1]\}\subset \Ds_{\ele_a}$).
      \item For all polyradii $\boldsymbol r$ and $\boldsymbol R$ the annulus $A^D(\boldsymbol r,\boldsymbol R)$ is a Reinhardt domain. Any Reinhardt domain is the union of polyannuli and so is the domain of convergence $\Ds$.
               \item The polydisk of convergence of a power series is such that any other polydisk  $\Delta(\boldsymbol r')$ with $r_j<r_j'$ for some $j$ contains points in where the power series diverge.
      \item The Laurent  series is locally normally summable in its domain of convergence and therefore locally absolutely uniformly summable.  We remind the reader that a Laurent series $\sum_{\kk\in\Z^D}a_{\kk}\z^{\kk}$ is locally normally summable if for any compact set $K\subset\Ds$ there exists $C>0$ and $\theta\in(0,1)$ such that $|a_{\kk}\z^{\kk}|\leq C\theta^{|\kk|}$ for $\z\in K$ and $\kk\in\Z^D$.
      \item The function $L(\z)$ is holomorphic (holomorphic in each variable $z_i$, $i=1,\dots, D$) in $\Ds$, which is its domain of holomorphy.

          \item Given a holomorphic function $L(\z)$ in $A^D_{\boldsymbol c}(\boldsymbol r,\boldsymbol R)$ (a polyannullus centered at $\boldsymbol c\in\C^D$),
          and a polyradius $\boldsymbol \rho$ such that $r_i<\rho_i<R_i$, $i=1,\dots,D$ then
              \begin{align*}
                L(\z)&=\sum_{\kk\in\Z^D}c_{\kk}(\z-\boldsymbol c)^{\kk},&
                c_{\kk}&=\frac{1}{(2\pi\operatorname{i})^D}\int_{\T_{\boldsymbol c}^D(\boldsymbol\rho)}\frac{L(\boldsymbol w)}{(\boldsymbol w-\boldsymbol c)^{\kk}} \dd w_1 \dots \dd w_D,
              \end{align*}
              where $\T^D_{\boldsymbol c}(\boldsymbol\rho)$ is the distinguished border of the polycircle centered at $\boldsymbol c$ with polyradius $\boldsymbol \rho$.
\end{enumerate}

\section{Proofs}\label{proofss}
\settocdepth{section}
\subsection{Proof of Proposition \ref{pro:cholesky}}\label{proof1}
\begin{proof}Assuming  $\det A\neq 0$ for any block matrix $M=\left(\begin{smallmatrix}
    A & B\\
    C & D
   \end{smallmatrix}\right)$ we can write in terms of Schur complements
\begin{align*}
M&=\PARENS{\begin{matrix}
    \mathbb{I} & 0\\
    CA^{-1} & \mathbb{I}
   \end{matrix}}
\PARENS{\begin{matrix}
    A & 0\\
    0 & M/ A
   \end{matrix}}
\PARENS{\begin{matrix}
    \mathbb{I} & A^{-1}B\\
    0 & \mathbb{I}
   \end{matrix}}.
\end{align*}
Thus, as  $\det G^{[k]}\neq 0$ $\forall k=0,1,\dots$, we can write
\begin{align*}
 G^{[\ell+1]}&=\PARENS{\begin{array}{c|c}
    \mathbb{I}^{[\ell]}              &    0    \\\hline\bigstrut[t]
    v^{[\ell],[\ell-1]}      &        \mathbb{I}_{[\ell] }
    \end{array}}
\PARENS{\begin{array}{c|c}
  G^{[\ell]}       & 0       \\\hline\bigstrut[t]
       0                 &              G^{[\ell+1]}/ G^{[\ell]}
      \end{array}}
\PARENS{\begin{array}{c|c}
\I^{[\ell]}           &  (v^{[\ell],[\ell-1]})^{\top}      \\\hline\bigstrut[t]
           0                   &  \I_{[\ell] }
      \end{array}},
\end{align*}
 where
 \begin{align*}
v^{[\ell],[\ell-1]}\coloneq \PARENS{\begin{matrix}
               v_{[\ell],[0]} & v_{[\ell],[1]} & \dots & v_{[\ell],[\ell-1]}
              \end{matrix}}
 \end{align*}
Applying the  same factorization   to $G^{[\ell]}$
we get
\begin{multline*}
G^{[\ell+1]}=\PARENS{\begin{array}{c|cc}
     \mathbb{I}^{[\ell-1]}              &     0       &  0  \\\hline\bigstrut[t]
      r^{[\ell-1][\ell-2]}                  & \mathbb{I}_{[\ell-1] }   &  0    \\
      s^{[\ell][\ell-2]}                  &      t_{[\ell][\ell-1]}      &  \mathbb{I}_{[\ell] }
      \end{array}}
\PARENS{\begin{array}{c|cc}
   G^{[\ell-1]}                        &      0     &    0\\\hline\bigstrut[t]
      0              & G^{[\ell]}/ G^{[\ell-1]} &  0    \\
      0              &          0      &  G^{[\ell+1]}/ G^{[\ell]}
      \end{array}}\\\times
\PARENS{\begin{array}{c|cc}
   \mathbb{I}^{[\ell-1]}                & (r^{[\ell-1][\ell-2]})^\top  &  (s^{[\ell][\ell-2]})^\top  \\\hline\bigstrut[t]
             0              & \mathbb{I}_{[\ell-1] }   &    (t_{[\ell],[\ell-1]})^\top \\
              0              &      0      &  \mathbb{I}_{[\ell] }
      \end{array}}.
\end{multline*}
Here the zeroes indicates zero rectangular matrices of different sizes. Finally, the iteration of these factorizations leads to
\begin{align*}
 G^{[\ell+1]}=\PARENS{\begin{matrix}
  \mathbb{I}_{[|0] }&     0     &\dots&0\\
         *   &\mathbb{I}_{[1] }&    \ddots  &\vdots\\
\vdots    &\ddots          &\ddots&0\\
*       &\dots   &*   &\mathbb{I}_{[\ell] }\\
\end{matrix}}\diag(G^{[1]}/ G^{[0]},G^{[2]}/ G^{[1]},\dots, &G^{[\ell+1]}/ G^{[\ell]})
\PARENS{\begin{matrix}
  \mathbb{I}_{[|0] }&     0     &\dots&0\\
         *   &\mathbb{I}_{[1] }&    \ddots  &\vdots\\
\vdots    &\ddots          &\ddots&0\\
*       &\dots   &*   &\mathbb{I}_{[\ell] }\\
\end{matrix}}^\top
\end{align*}
Since this would have been valid for any $\ell$ it would also hold for the direct
limit $\lim\limits_{\longrightarrow}G^{[\ell]}$.
\end{proof}
\subsection{Proof of Proposition \ref{cauchy}}\label{proof2}
\begin{proof}
  \begin{flalign*}
    C_{[\ell]}(\z)=&\sum_{n=0}^\infty (SG)_{[\ell],[n]}\chi^*_{[n]}(\z) & \text{use the Cholesky factorization \eqref{cholesky} }\\
    =&\sum_{n=0}^\infty\sum_{k=0}^\ell S_{[\ell],[k]}G_{[k],[n]}\chi^*_{[n]}(\z)\\
    =&\sum_{n=0}^\infty\int_\Omega\sum_{k=0}^\ell S_{[\ell],[k]}\chi_{[k]}(\y)\dd\mu(\y)\big(\chi_{[n]}(\y)\big)^\top\chi_{[n]}^*(\z) & \text{recall \eqref{eq:Gkl}}\\
    =&\sum_{n=0}^\infty\int_\Omega P_{[\ell]}(\y)\dd\mu(\y)\big(\chi_{[n]}(\y)\big)^\top\chi_{[n]}^*(\z) & \text{because \eqref{eq:polynomials}}\\
    =&\int_\Omega P_{[\ell]}(\y)\dd\mu(\y)\sum_{n=0}^\infty\big(\chi_{[n]}(\y)\big)^\top\chi_{[n]}^*(\z)  &\text{\small interchange of series and integral}\\
    =&\int_\Omega P_{[\ell]}(\y)\dd\mu(\y)\frac{1}{(z_1-y_1)\cdots(z_D-y_D)} & \text{recall \eqref{eq:chi-chi*}}.
      \end{flalign*}
\end{proof}

\subsection{Proof of Proposition \ref{pro3}}\label{proof3}
\begin{proof}From \eqref{lambdaTchi} we deduce that
   \begin{align*}
 \Big[\prod_{i=1}^k (\Lambda_{a_i}^\top-q_{a_i})\Big]\chi^*=&\Big[\prod_{i=1}^k (x_{a_i}\Pi_{a_i}-q_{a_i})\Big]\chi^*
  \end{align*}
but
\begin{align*}
  \prod_{i=1}^k (x_{a_i}\Pi_{a_i}-q_{a_i})=&\prod_{i=1}^k (x_{a_i}-q_{a_i}-x_{a_i}\Pi_{a_i}^\perp)
  \\=&\Big[\prod_{i=1}^k(x_{a_i}-q_{a_i})\Big]+(-1)^k\Big[\prod_{i=1}^kx_{a_i}\Pi^\perp_{a_{i}}\Big]\\
  &\hspace*{3cm}+\sum_{j=1}^{k-1}\frac{(-1)^j}{(k-j)!j!}\sum_{\sigma\in \mathfrak S_k}\Big(\Big[\prod_{i=j+1}^k\big(x_{a_{\sigma i}}-q_{a_{\sigma i}}\big)\Big]\Big[\prod_{i=1}^jx_{a_{\sigma i}}\Pi^\perp_{a_{\sigma i}}\Big]\Big)
\end{align*}
and \eqref{piachi} implies the result.
\end{proof}
\subsection{Proof of Proposition \ref{pro4}}\label{proof4}
\begin{proof}
  Just follow the chain of equalities
      \begin{flalign*}
    \Big[    \prod_{i=1}^k (J_{a_i}-q_{a_i})\Big]C=&S \prod_{i=1}^k (\Lambda_{a_i}-q_{a_i})S^{-1}H(S^{-1})^\top\chi^* & \!\!\!\!\text{use \eqref{def:C} and \eqref{def:J}}\\
    =&S \prod_{i=1}^k (\Lambda_{a_i}-q_{a_i})G\chi^*& \text{use \eqref{cholesky} }\\
    =&SG\prod_{i=1}^k (\Lambda_{a_i}^\top-q_{a_i})\chi^*& \text{from \eqref{eq:symmetry}}\\
    =&H(S^{-1})^\top\prod_{i=1}^k (\Lambda_{a_i}^\top-q_{a_i})\chi^*& \text{follows from \eqref{cholesky}}.
  \end{flalign*}
Finally,  \eqref{lambdaTq} implies the announced result.
\end{proof}

\subsection{Proof of Proposition \ref{pro5}}\label{proof5}
\begin{proof}
   \begin{enumerate}
 \item Is proven as follows
    \begin{align*}
      T_aW_1=&(T_aS)(T_aW_0)\\
      =& (T_aS)( \n_a\cdot\boldsymbol\Lambda-q_a)W_0\\
      =&(T_aS)( \n_a\cdot\boldsymbol\Lambda-q_a)S^{-1}W_1\\
      =&\omega_a W_1,\\
      T_a W_2=&(T_aH)\big(T_aS^{-1}\big)^\top\\
      =& (T_aH)\big(T_aS^{-1}\big)^\top \big(S^\top H^{-1}\big)\big(S^\top H^{-1}\big)^{-1}\\
      =&\omega_a  W_2.
    \end{align*}
    \item For the first equation observe that
    \begin{align*}
      (T_a (\n_b\cdot\boldsymbol J))\omega_a=&(T_aS) (\n_b\cdot\boldsymbol\Lambda)(T_aS)^{-1}(T_aH)\big((T_aS)^{-1}\big)^\top S^\top H^{-1}\\
      =&(T_aS)  (\n_b\cdot\boldsymbol\Lambda) (T_aG) S^\top H^{-1}\\
      =&(T_aS) (T_aG)  (\n_b\cdot\boldsymbol\Lambda)^\top S^\top H^{-1}\\
      =&(T_aH)\big((T_aS)^{-1}\big)^\top ( \n_b\cdot\boldsymbol\Lambda)^\top S^\top H^{-1}\\
      =&(T_aH)\big((T_aS)^{-1}\big)^\top S^\top H^{-1} H \big(S^{-1}\big)^\top ( \n_b\cdot\boldsymbol\Lambda)^\top S^\top H^{-1}\\
      =&\omega_aH( \n_b\cdot\boldsymbol J)^\top H^{-1}\\
      =&\omega_a  (\n_b\cdot\boldsymbol J).
    \end{align*}
    and for the second one
    \begin{align*}
     M_b(T_b M_a)=S(T_bS)^{-1}(T_bS)(\n_a\cdot\boldsymbol\Lambda)(T_bS)^{-1}=(\n_a\cdot\boldsymbol J)M_b.
    \end{align*}
    \item For the first equation from \eqref{eq: linear W} we get $T_b(T_aW)=(T_b\omega_a)(T_bW)=\big[(T_b\omega_a)\omega_b\big]W$ and interchanging $a \leftrightarrow b$
        we get
    $ \big[(T_a\omega_b)\omega_a-(T_b\omega_a)\omega_b\big]SW_0=0$. For the second equation, from the definitions, it is easy to see that
    \begin{align*}
     M_a(T_a M_b)=S (T_a T_b S)^{-1}=M_b (T_bM_a).
    \end{align*}

      \end{enumerate}
\end{proof}
\subsection{Proof of Proposition \ref{pro6}}\label{proof6}
\begin{proof}

From \eqref{lambdaTchi}  we get
\begin{align*}
  (\n_a\cdot\Lambda-q_a)^\top\chi^*=&\Big(\sum_{b=1}^D n_{a,b}x_b\Pi_b-q_a\Big)\chi^*
  \\=&(\n_a\cdot\x-q_a)\chi^*-\Big(\sum_{b=1}^D n_{a,b}x_b\Pi_b^\bot\Big)\chi^*
  \\=&(\n_a\cdot\x-q_a)\chi^*-\Big(\sum_{b=1}^D n_{a,b} \lim_{x_b\to\infty}x_b\chi^*\Big)
  \\=&(\n_a\cdot\x-q_a)\chi^*-\n_a\cdot\widehat{\boldsymbol{\chi^*}}
\end{align*}
where
\begin{align*}
\widehat{ \boldsymbol{\chi^*}}=(\lim_{x_1\to\infty}x_1\chi^*,\dots,\lim_{x_D\to\infty}x_D\chi^*).
\end{align*}
Consequently
  \begin{flalign*}
  M_a (T_aC)&=H\big((T_aS)( \n_a\cdot\boldsymbol\Lambda-q_a)S^{-1}\big)^\top(T_aH)^{-1}(T_aH)
   \Big(\big(T_aS\big)^{-1}\Big)^\top\chi^*    & \text{from \eqref{def:C} and  \eqref{eq:Ma+alternative}}\\
   &=H\big(S^{-1}\big)^\top( \n_a\cdot\boldsymbol\Lambda-q_a)^\top\chi^*
  \\&=(\n_a\cdot\x-q_a)C-\n_{a}\cdot\widehat{\boldsymbol C},
\end{flalign*}
that together with
\begin{align*}
  \omega_a C&=T_aC
\end{align*}
imply the result.
\end{proof}
\subsection{Proof of Theorem \ref{teo1}}\label{proof7}
\begin{proof}
Previously to the proof we need
\begin{lemma}\label{l1}
  The following relation is satisfied by the second kind functions
  \begin{align}\label{eq:Cq}\begin{multlined}
     M(TC)= \Big[\prod_{a=1}^D(x_{a}-q_{a})\Big]C+(-1)^{D}\widehat{C}_{1,\dots,D}
    +\sum_{j=1}^D\frac{(-1)^j}{(D-j)!j!}\sum_{\sigma\in \mathfrak S_D}\Big(\Big[\prod_{a=j+1}^D\big(x_{\sigma a}-q_{\sigma a}\big)\Big]
    \widehat{C}_{\sigma 1,\dots,\sigma j}\Big),
  \end{multlined}
  \end{align}
\end{lemma}
\begin{proof}
  Is a consequence of
\begin{flalign*}
  M(TC)=&(H)(S^{-1})^{\top}\Big[\prod_{a=1}^D(\Lambda_a-q_a)\Big]^{\top}(TS)^{\top}(TH)^{-1}(TH)((TS)^{-1})^\top\chi^* & \text{from \eqref{eq:alternativeMD}}\\
  =&H(S^{-1})^{\top}\Big[\prod_{a=1}^D(\Lambda_a-q_a)\Big]^{\top}\chi^*\\
  =&(HS^{-1})^{\top}\Big(\Big[\prod_{a=1}^D(x_{a}-q_{a})\Big]\chi^* &\text{from \eqref{lambdaTq}}
  \\&+\sum_{j=1}^D\frac{(-1)^j}{(k-j)!j!}\sum_{\sigma\in \mathfrak S_D}\Big(\Big[\prod_{a=j+1}^D\big(x_{\sigma a}-q_{\sigma a}\big)\Big]
    \lim_{x_{\sigma 1}\to\infty}\cdots\lim_{ x_{\sigma j}\to\infty}\Big(\Big[\prod_{i=1}^jx_{\sigma i}\Big]\chi^*\Big)\Big)\Big).
\end{flalign*}
\end{proof}
Now, in Lemma \ref{l1} we put $\x=\boldsymbol q$ into \eqref{eq:Cq}, observe that $\widehat{C}_{[k],1,\dots,D}=\delta_{k,0}H_{[0]}$, and multiply by the  inverse of the lower unitriangular matrix $M$ to get
\begin{align}\label{TC}
(C)_{[k]}(\boldsymbol q)=& (-1)^D  \big(T^{-1}M^{-1}\big)_{[k],[0]}
T^{-1}H_{[0]}.
  \end{align}
According to \eqref{factorM-} with $\sigma=1$ we have
  \begin{align*}
    T^{-1}M^{-1}=\big(T_D^{-1}M_{D}\big)^{-1}\cdots\big(T_{2}^{-1}\cdots T_{D}^{-1}M_{2}\big)^{-1}(T^{-1}M_{1})^{-1}
\end{align*}
and given the particular structure of $M_a$, $a\in\{1,\dots,D\}$, we have the following simple expression
 \begin{align*}
\big(T_a^{-1}\cdots T_{D}^{-1}M_a\big)^{-1}&=\PARENS{\begin{matrix}
\I_{[0]}                          &   0                                          & 0                     & 0&               \cdots        \\
-\rho^{(a)}_{[1]}                                     &   \I_{[1]}                                          & 0                       &    0&          \cdots         \\
\rho^{(a)}_{[2]}\rho^{(a)}_{[1]}                      & -\rho^{(a)}_{[2]} &\I_{[2]}                      & 0                                 & \cdots\\
-\rho^{(a)}_{[3]}\rho^{(a)}_{[2]}\rho^{(a)}_{[1]}    &\rho^{(a)}_{[3]}\rho^{(a)}_{[2]}  & -\rho^{(a)}_{[3]} &\I_{[3]}          &  \ddots       & \\
         \vdots                           &      \quad\quad\quad\quad\ddots                  &         \quad  \quad \quad\ddots       &                        &\ddots
                 \end{matrix}}
 \end{align*}
for $a\in\{1,\dots,D\}$. This allows for explicit computation of the elements of the inverse matrix $M^{-1}$ and in particular leads to products over multisets, see Appendix \ref{symmetric},
\begin{align*}
\big(T^{-1}M^{-1}\big)_{[k],[0]}=(-1)^k \smashoperator{\sum_{1\leq a_1\leq\cdots \leq a_k\leq D }}\rho^{(a_k)}_{[k]}\cdots \rho^{(a_1)}_{[1]}
\end{align*}
so that \eqref{TC} reads
\begin{align*}
  C_{[k]}(\boldsymbol q)=(-1)^{k+D} \smashoperator{\sum_{1\leq a_1\leq\cdots \leq a_k\leq D }}\rho^{(a_k)}_{[k]}\cdots \rho^{(a_1)}_{[1]}T^{-1}H_{[0]}
\end{align*}
and recalling \eqref{rhoij2} we get the desired result.
\end{proof}
\subsection{Proof of Theorem \ref{translatedCD}}\label{prooftranslatedCD}
\begin{proof}
  To obtain the  result we consider the expressions of
$ P^{[\ell]}(\x)^\top  (H^{[\ell]})^{-1}M_a^{[\ell]}(T_aP^{[\ell]})(\y)$  when letting the operator between square brackets
act to the right or to the left. Acting on its right gives the Christoffel--Darboux kernel
\begin{flalign*}
  P^{[\ell]}(\x)^\top  (H^{[\ell]})^{-1}M_a^{[\ell]}(T_aP^{[\ell]})(\y)=&  P^{[\ell]}(\x)^\top  (H^{[\ell]})^{-1}S^{[\ell]}(T_aS^{[\ell]})^{-1}(T_aP^{[\ell]})(\y), & \text{consequence of \eqref{def:Ma}}\\
    =&P^{[\ell]}(\x)^\top  (H^{[\ell]})^{-1}P^{[\ell]}(\y)& \text{see \eqref{eq:polynomials}}\\
    =&K^{(\ell)}(\x,\y),& \text{see Definition \ref{DEFCD}.}
    \end{flalign*}
     If we act on the left, recalling \eqref{eq:Ma+alternative} we get
       \begin{align*}
     P^{[\ell]}(\x)^\top  (H^{[\ell]})^{-1}M_a^{[\ell]}(T_aP^{[\ell]})(\y)=&   P^{[\ell]}(\x)^\top  (H^{[\ell]})^{-1}H^{[\ell]}\Big(\big((T_aS)( \n_a\cdot\boldsymbol\Lambda-q_a)S^{-1}\big)^{[\ell]}\Big)^\top(T_aH^{[\ell]})^{-1}(T_aP^{[\ell]})(\y) \\
    =&\Big(\big((T_aS)( \n_a\cdot\boldsymbol\Lambda-q_a)S^{-1}\big)^{[\ell]}P^{[\ell]}(\x)\Big)^\top (T_aH^{[\ell]})^{-1}(T_aP^{[\ell]})(\y)
    \end{align*}
     Now, with the help of the block decomposition of any block semi-infinite matrix $M=\begin{psmallmatrix}
       M^{[\ell]}&M^{[\ell],[\geq \ell]}\\
        M^{[\geq \ell],[\ell]}&M^{[\geq \ell]}
     \end{psmallmatrix}$ we write
     \begin{align*}
       \big((T_aS)( \n_a\cdot\boldsymbol\Lambda-q_a)S^{-1}\big)^{[\ell]}P^{[\ell]}(\x)= \big((T_aS)( \n_a\cdot\boldsymbol\Lambda-q_a)S^{-1}P(\x)\big)^{[\ell]}-\big((T_aS)( \n_a\cdot\boldsymbol\Lambda-q_a)S^{-1}\big)^{[\ell],[\geq \ell]}P^{[\geq \ell]}(\x)
     \end{align*}
     In the one hand, if we take into account  \eqref{eq:polynomials} and \eqref{eigen} the first term in the LHS reads
     \begin{align*}
        \big((T_aS)( \n_a\cdot\boldsymbol\Lambda-q_a)S^{-1}P(\x)\big)^{[\ell]}&= (\n_a\cdot\x-q_a)T_aP^{[k]}(\x)
     \end{align*}
     and in the other hand, given the lower unitriangular form of $T_aS$ and $S$ and that $\n_a\cdot\boldsymbol\Lambda$ is zero but for the first superdiagonal
     \begin{align*}
       \big((T_aS)( \n_a\cdot\boldsymbol\Lambda-q_a)S^{-1}\big)^{[\ell],[\geq \ell]}=\PARENS{\begin{matrix}
         0_{[0],[\ell]} &0_{[0],[\ell+1]}&\dots\\
           0_{[1],[\ell]} &0_{[1],[\ell+1]}&\dots\\
           \vdots &\vdots\\
             0_{[\ell-2],[\ell]} &0_{[\ell-2],[\ell+1]}&\dots\\
           (\n\cdot\boldsymbol\Lambda)_{[\ell-1],[\ell]} &0_{[\ell-1],[\ell+1]}&\dots
       \end{matrix}}
     \end{align*}
     and therefore
     \begin{align*}
     \big((T_aS)( \n_a\cdot\boldsymbol\Lambda-q_a)S^{-1}\big)^{[\ell],[\geq \ell]}P^{[\geq \ell]}(\x)=\PARENS{
     \begin{matrix}
       0_{[0]}\\
       0_{[1]}\\
       \vdots\\
       0_{[\ell-2]}\\
        (\n\cdot\boldsymbol\Lambda)_{[\ell-1],[\ell]}P_{[\ell]}(\x)
            \end{matrix}}.
     \end{align*}
 Hence,
 \begin{multline*}
   \Big(\big((T_aS)( \n_a\cdot\boldsymbol\Lambda-q_a)S^{-1}\big)^{[\ell]}P^{[\ell]}(\x)\Big)^\top= (\n_a\cdot\x-q_a)(T_aP^{[\ell]}(\x))^\top\\-(0_{[0]},0_{[1]},\dots,0_{[\ell-2]},P_{[\ell]}(\x)^\top((\n_a\cdot\boldsymbol  \Lambda)_{[\ell-1],[\ell]})^\top)
 \end{multline*}
   so that
       \begin{multline*}
     P^{[\ell]}(\x)^\top  (H^{[\ell]})^{-1}M_a^{[\ell]}(T_aP^{[\ell]})(\y)=   (\n_a\cdot\x-q_a)T_a\big(P^{[\ell]}(\x)^\top   (H^{[\ell]})^{-1}P^{[\ell]}(\y)\big)\\-P_{[\ell]}(\x)^\top\big((\n_a\cdot\boldsymbol  \Lambda)_{[\ell-1],[\ell]}\big)^\top
   (T_aH_{[\ell-1]})^{-1}(T_aP_{[\ell-1]})(\y).
    \end{multline*}
   Consequently, equating both results we conclude
   \begin{align}\label{CDprevio}
   K^{(\ell)}(\x,\y)  = (\n_a\cdot\x-q_a)T_aK^{(\ell)}(\x,\y)-\big( (T_aH_{[\ell-1]})^{-1}(\n_a\cdot\boldsymbol  \Lambda)_{[\ell-1],[\ell]}P_{[\ell]}(\x)\big)^\top
(T_aP_{[\ell-1]})(\y).
   \end{align}
   Now we recall \eqref{SnS} in the following form
   \begin{align*}
 (T_aH_{[\ell-1]})^{-1}(\n_a\cdot\boldsymbol \Lambda)_{[\ell-1],[k]} P_{[\ell]}&=  (\n_a\cdot\x-q_a) (T_aH_{[\ell-1]})^{-1} (T_aP)_{[\ell-1]}-H_{[\ell-1]}^{-1}P_{[\ell-1]},
   \end{align*}
   and introduce it into \eqref{CDprevio} to get
     \begin{align*}
   K^{(\ell)}(\x,\y)  =& (\n_a\cdot\x-q_a)T_aK^{(\ell)}(\x,\y)-\big( (\n_a\cdot\x-q_a) (T_aH_{[\ell-1]})^{-1} (T_aP_{[\ell-1]}(\x))-H_{[\ell-1]}^{-1}P_{[\ell-1]}(\x))\big)^\top
(T_aP_{[\ell-1]})(\y)\\
   =& (\n_a\cdot\x-q_a)T_aK^{(\ell-1)}(\x,\y)+P_{[\ell-1]}(\x)^\top)H_{[\ell-1]}^{-1}(T_aP_{[\ell-1]})(\y).
   \end{align*}
\end{proof}
\subsection{Proof of Proposition \ref{betaH2step}}\label{proof beta H 2 step}
\begin{proof}
  Observe that  Proposition \ref{pro:chistoffel2} implies
\begin{align*}
  \omega_{[k],[k]} =&-\omega_{[k],[k+2]}\PARENS{\begin{matrix}
     \Sigma^{(1),k}_{[k+2]} & \Sigma^{(2),{k+1}}_{[k+2]}
    \end{matrix}}\PARENS{\begin{matrix}
    \Sigma^{(1),k}_{[k]} &\Sigma^{(2),k+1}_{[k]}\\
    \Sigma^{(1),k}_{[k+1]} &\Sigma^{(2),k+1}_{[k+1]}
    \end{matrix}}^{-1}\PARENS{\begin{matrix}
      \I_{[k]}\\0_{[k+1],[k]}
    \end{matrix}},\\
\omega_{[k],[k+1]}
=&-\omega_{[k],[k+2]}
\PARENS{\begin{matrix}
     \Sigma^{(1),k}_{[k+2]} & \Sigma^{(2),{k+1}}_{[k+2]}
    \end{matrix}}
    \PARENS{\begin{matrix}
    \Sigma^{(1),k}_{[k]} &\Sigma^{(2),k+1}_{[k]}\\
    \Sigma^{(1),k}_{[k+1]} &\Sigma^{(2),k+1}_{[k+1]}
    \end{matrix}}^{-1}
    \PARENS{\begin{matrix}
  0_{[k],[k+1]}\\    \I_{[k+1]}
    \end{matrix}}.
  \end{align*}
Now, from Proposition \ref{pro:omega2c} we get
  \begin{align*}
(T^{(1)}T^{(2)}H_{[k]})H_{[k]}^{-1} =&-\big((\n^{(1)}\cdot\boldsymbol\Lambda)(\n^{(2)}\cdot\boldsymbol\Lambda)\big)_{[k],[k+2]}\PARENS{\begin{matrix}
     \Sigma^{(1),k}_{[k+2]} & \Sigma^{(2),{k+1}}_{[k+2]}
    \end{matrix}}\PARENS{\begin{matrix}
    \Sigma^{(1),k}_{[k]} &\Sigma^{(2),k+1}_{[k]}\\
    \Sigma^{(1),k}_{[k+1]} &\Sigma^{(2),k+1}_{[k+1]}
    \end{matrix}}^{-1}\PARENS{\begin{matrix}
      \I_{[k]}\\0_{[k+1],[k]}
    \end{matrix}},
    \end{align*}
    and
    \begin{multline*}
 (T^{(1)}
 T^{(2)}\beta_{[k]})
 \big((\n^{(1)}\cdot\boldsymbol\Lambda)(\n^{(2)}
 \cdot\boldsymbol\Lambda)\big)_{[k-1],[k+1]}
=(\n\cdot\boldsymbol\Lambda)_{[k],[k+1]}+\big((\n^{(1)}\cdot\boldsymbol\Lambda)(\n^{(2)}\cdot\boldsymbol\Lambda)\big)_{[k],[k+2]}\Big(\beta_{[k+2]}\\-
\PARENS{\begin{matrix}
     \Sigma^{(1),k}_{[k+2]} & \Sigma^{(2),{k+1}}_{[k+2]}
    \end{matrix}}
    \PARENS{\begin{matrix}
    \Sigma^{(1),k}_{[k]} &\Sigma^{(2),k+1}_{[k]}\\
    \Sigma^{(1),k}_{[k+1]} &\Sigma^{(2),k+1}_{[k+1]}
    \end{matrix}}^{-1}
    \PARENS{\begin{matrix}
  0_{[k],[k+1]}\\    \I_{[k+1]}
    \end{matrix}}\Big).
  \end{multline*}
  From here the stated result follows easily by recalling the expressions of the quasi-determinant.
\end{proof}
\subsection{Proof of Theorem \ref{teo2}}\label{proof8}
\begin{proof}
  From \eqref{mdarbouxm} we have
\begin{align*}
\prod_{i=1}^m(\n^{(i)}\cdot\x-q^{(i)})\Big(\prod_{i=1}^m T_{\n^{(i)}}P\Big)_{[k]}(\x)=\omega_{[k],[k+m]} P_{[k+m]}(\x)+\omega_{[k],[k+m-1]} P_{[k+m-1]}(\x)+\cdots+\omega_{[k],[k]}P(\x)\\
  =\omega_{[k],[k+m]} \Big(P_{[k+m]}(\x)-\PARENS{\begin{matrix}
     \Sigma^{(1),k}_{[k+m]} & \dots&\Sigma^{(m),{k+m-1}}_{[k+m]}
    \end{matrix}}\PARENS{\begin{matrix}
        \Sigma^{(1),k}_{[k]} &\dots &\Sigma^{(m),k+m-1}_{[k]}\\ \vdots&&\vdots\\
    \Sigma^{(1),k}_{[k+m-1]} &\dots&\Sigma^{(m),k+m-1}_{[k+m-1]}
    \end{matrix}}^{-1}\PARENS{\begin{matrix}
      P_{[k]}(\x)\\
      \vdots\\
        P_{[k+m-1]}(\x)
    \end{matrix}}\Big)
\end{align*}
from where the result follows. For the second kind functions we proceed similarly
\begin{multline*}
\Big(\prod_{i=1}^m T_{\n^{(i)}}C\Big)_{[k]}(\x)=\omega_{[k],[k+m]} C_{[k+m]}(\x)+\omega_{[k],[k+m-1]} C_{[k+m-1]}(\x)+\cdots+\omega_{[k],[k]}C(\x)\\
  =\omega_{[k],[k+m]} \Big(C_{[k+m]}(\x)-\PARENS{\begin{matrix}
     \Sigma^{(1),k}_{[k+m]} & \dots&\Sigma^{(m),{k+m-1}}_{[k+m]}
    \end{matrix}}\PARENS{\begin{matrix}
        \Sigma^{(1),k}_{[k]} &\dots &\Sigma^{(m),k+m-1}_{[k]}\\ \vdots&&\vdots\\
    \Sigma^{(1),k}_{[k+m-1]} &\dots&\Sigma^{(m),k+m-1}_{[k+m-1]}
    \end{matrix}}^{-1}\PARENS{\begin{matrix}
      C_{[k]}(\x)\\
      \vdots\\
        C_{[k+m-1]}(\x)
    \end{matrix}}\Big)
\end{multline*}
\end{proof}
\subsection{Proof of Theorem \ref{superostia}}\label{mCD mas}
\begin{proof}
  Let us consider a similar matrix to that discussed in Definition \ref{T}, $M=S(TS)^{-1}$ which factors out as $M=M^{(1)}(T^{(1)}M^{(2)})\cdots (T^{(1)}\cdots T^{(m-1}M^{(m)})$. From the symmetry of the moment matrix,
 $G\prod_{i=1}^m\big(n^{(i)}\cdot\boldsymbol \Lambda-q^{(i)}\big)^\top=TG$, we conclude that
$M=H\omega^\top(TH)^{-1}$. Notice that $MTP=P$ and $\omega P=\mathcal Q TP$.
 Now, we proceed as in the proof of Theorem \ref{translatedCD} and evaluate $P^{[\ell+m]}(\x)^\top\big(H^{[\ell+m]}\big)^{-1}M^{[\ell+m]}TP^{[\ell+m]}(\y)$; a sandwich constructed in terms of $(\ell+m)$-th block truncations of semi-infinite matrices.
 We do it in two ways; first by acting on the right and, as $S$ is a block lower unitriangular matrix and therefore $M^{[\ell+m]}=S^{[\ell+m]}\big(TS^{[\ell+m]}\big)^{-1}$, we get
 \begin{align*}
   P^{[\ell+m]}(\x)^\top\big(H^{[\ell+m]}\big)^{-1}M^{[\ell+m]}TP^{[\ell+m]}(\y)=
   &P^{[\ell+m]}(\x)^\top\big(H^{[\ell+m]}\big)^{-1}P^{[\ell+m]}(\y)\\
   =& K^{(\ell+m)}(\x,\y).
 \end{align*}
 To act on the left we first evaluate the truncation $M^{[\ell+m]}=H^{[\ell+m]}\big(\omega^{[\ell+m]}\big)^\top \big(TH^{[\ell+m]}\big)^{-1}$ in terms of the corresponding truncated resolvent. Then,
 \begin{align*}
   P^{[\ell+m]}(\x)^\top\big(H^{[\ell+m]}\big)^{-1}M^{[\ell+m]}TP^{[\ell+m]}(\y)=
   &\big(\omega^{[\ell+m]}P^{[\ell+m]}(\x)\big)^\top\big(TH^{[\ell+m]}\big)^{-1}P^{[\ell+m]}(\y).
 \end{align*}
As we know the resolvent $\omega$ is a block upper triangular semi-infinite matrix with all its superdiagonals equal to zero but for the first $m$. Thus, the $(\ell+m)$-th truncation gives a matrix of the following form
\begin{align*}\omega^{[\ell+m]}=\PARENS{
\begin{array}{c}
\omega^{[l],[\ell+m]}\\
\hline
\begin{array}{c|c}\bigstrut[t]
0&\omega^{[\ell,m]}
  \end{array}
  \end{array}}
\end{align*}
where $\omega^{[\ell],[\ell+m]}$ is a truncation built up with the first $\ell$  block rows and the first $\ell+m$ block columns of the resolvent $\omega$, that for $\ell$ big enough looks like
{\small
\begin{align*}
 \omega^{[\ell],[\ell+m]}\coloneq \PARENS{
 \begin{matrix}
   \omega_{[0],[0]} &  \omega_{[0],[1]} &\dots&\omega_{[0],[m]}&0_{[0],[m+1]}&0_{[0],[m+2]}& \dots &0_{[0],[\ell-1]}& \dots& 0_{[0],[m+\ell-1]}\\
 0_{[1],[0]} &  \omega_{[1],[1]} &\dots&\omega_{[1],[m]}&\omega_{[1],[m+1]}& 0_{[1],[m+2]}&\dots &0_{[1],[\ell-1]}& \dots& 0_{[1],[m+\ell-1]}\\
 \vdots &&&&&&\ddots&\vdots\\
  0_{[\ell-1],[0]} &  0_{[\ell-1],[1]} &\dots&0_{[\ell-1],[m]}&0_{[\ell-1],[m+1]}& 0_{[\ell-1],[m+2]}&\dots& \omega_{[\ell-1],[\ell-1]} & \dots& \omega_{[\ell-1],[m+\ell-1]}
 \end{matrix}
 }.
\end{align*}}
 Then,
 \begin{align*}
  \omega^{[\ell+m]}P^{[\ell+m]}=\PARENS{
\begin{array}{c}
\omega^{[l],[\ell+m]}\\
\hline
\begin{array}{c|c}\bigstrut[t]
0&\omega^{[\ell,m]}
  \end{array}
  \end{array}}P^{[\ell+m]}(\x)=\PARENS{\begin{array}{c}\omega^{[\ell],[\ell+m]}P^{[\ell+m]}(\x)\\\hline
  \omega^{[\ell,m]}\PARENS{\begin{matrix}\bigstrut[t]
    P_{[\ell]}(\x)\\\vdots\\P_{\ell+m}(\x)
  \end{matrix}}
      \end{array}}
 \end{align*}

Is important to notice that each row of the truncation  $\omega^{[\ell],[\ell+m]}$ contains the complete non trivial part of the corresponding row of the resolvent; i.e.
\begin{align*}
  \omega^{[\ell],[\ell+m]}P^{[\ell+m]}(\x)=&(\omega P(\x))^{[\ell]},\\
  =&\mathcal Q(\x)TP^{[\ell]}(\x).
\end{align*}
Therefore
\begin{align*}
  P^{[\ell+m]}(\x)^\top\big(H^{[\ell+m]}\big)^{-1}M^{[\ell+m]}TP^{[\ell+m]}(\y)=&\PARENS{\begin{array}{c}
\mathcal Q(\x)TP^{[\ell]}(\x)\\\hline
  \omega^{[\ell,m]}\PARENS{\begin{matrix}\bigstrut[t]\bigstrut[t]
    P_{[\ell]}(\x)\\\vdots\\P_{\ell+m}(\x)
  \end{matrix}}
      \end{array}}^\top\PARENS{
      \begin{array}{c}\bigstrut[t]
        (TH^{[\ell]})^{-1}TP^{[\ell]}(\y)\\
          (TH_{[\ell]})^{-1}TP_{[\ell]}(\y)\\\vdots\\   (TH_{[\ell+m]})^{-1}TP_{\ell+m}(\y)
      \end{array}}\\
      =&\begin{multlined}[t]
        \mathcal Q(\x)(TP^{[\ell]}(\x))^\top (TH^{[\ell]})^{-1}TP^{[\ell]}(\y)\\+\PARENS{\begin{array}{c}
    P_{[\ell]}(\x)\\\vdots\\P_{\ell+m}(\x)
  \end{array}}^\top\big(\omega^{[\ell,m]}\big)^\top\PARENS{\begin{matrix}
    P_{[\ell]}(\x)\\\vdots\\P_{\ell+m}(\x)
  \end{matrix}}.
      \end{multlined}
\end{align*}
\end{proof}

\subsection{Proof of Proposition \ref{otro mas}}\label{yo que se}
\begin{proof}
  We first notice that
 \begin{align*}
  \frac{\partial S}{\partial t_a}=&\frac{\partial\beta^{(1)}}{\partial t_a}+\frac{\partial\beta^{(2)}}{\partial t_a}+\cdots\\
  S^{-1}=&\I-\beta^{(1)}\underbracket{-
  \beta^{(2)}+
  \big(
  \beta^{(1)}
  \big)^2}_{\text{second subdiagonal}}+\cdots
\end{align*}
so that we can split the right $t_a$ derivative of $S$ into subdiagonals as follows
\begin{align*}
  \frac{\partial S}{\partial t_a}S^{-1}=&  \Big(\frac{\partial\beta^{(1)}}{\partial t_a}+\frac{\partial\beta^{(2)}}{\partial t_a}+\cdots\Big)\Big(\I-\beta^{(1)}-
  \beta^{(2)}+
  \big(
  \beta^{(1)}
  \big)^2\Big)\\=&\frac{\partial\beta^{(1)}}{\partial t_a}+\underbracket{\frac{\partial\beta^{(2)}}{\partial t_a}-\frac{\partial\beta^{(1)}}{\partial t_a}\beta^{(1)}}_{\text{second subdiagonal}}+
 \underbracket{ \frac{\partial\beta^{(3)}}{\partial t_a}- \frac{\partial\beta^{(2)}}{\partial t_a}\beta^{(1)}
  -\frac{\partial\beta^{(1)}}{\partial t_a}\beta^{(2)}+\frac{\partial\beta^{(1)}}{\partial t_a}
  \big( \beta^{(1)}  \big)^2}_{\text{third subdiagonal}}+\cdots.
\end{align*}
Now, recalling that the basic Jacobi operators have only a non vanishing subdiagonal from \eqref{diff jacobi} we get
\begin{align*}
  \frac{\partial\beta_{[k]}}{\partial t_a}=&J_{[k],[k-1]},\\
  \frac{\partial\beta^{(2)}}{\partial t_a}=&\frac{\partial\beta^{(1)}}{\partial t_a}\beta^{(1)},\\
  \frac{\partial\beta^{(3)}}{\partial t_a}= &\frac{\partial\beta^{(2)}}{\partial t_a}\beta^{(1)}
  +\frac{\partial\beta^{(1)}}{\partial t_a}\beta^{(2)}-\frac{\partial\beta^{(1)}}{\partial t_a}
  \big( \beta^{(1)}  \big)^2.
\end{align*}
\end{proof}
\subsection{Proof of Proposition \ref{pro.baker.expressions}}\label{proof9}
\begin{proof}
  We first consider the expressions for the Baker functions $\Psi_1$, $\Psi_2$ and the adjoint Baker function $\Psi_2^*$:
\begin{align*}
  \Psi_1(\z,t,\m)=&
  S(t,\m)\Exp{t(\z)}\Big[\prod_{a=1}^D\big((\n_a\cdot\x)-q_a\big)^{m_a}\Big]\chi(\z)\\=&\Exp{t(\z)}\Big[\prod_{a=1}^D\big((\n_a\cdot\x)-q_a\big)^{m_a}\Big]P(\z,t,\m),\\
  \Psi_2^*(\z,t,\m)=&H(t,\m)^{-1}S(t,\m)\chi(\z)=H(t,\m)^{-1}P(\z,t,\m),\\
  \Psi_2(\z,t)=&H(t,\m)(S(t,\m)^{-1})^\top\chi^*(\z)=C(\z,t,\m).
\end{align*}
For the remaining Baker function  $\Psi_1^*$ we proceed as follows
\begin{align*}
    \Psi_1^*(\z,t,m)=&  \big[(W_1(t,\m))^{-1}\big]^\top\chi=    (W_2(\z,t,\m)^{-1})^\top G^\top\chi^*(\z)\\
    =&H(t,\m)^{-1} S(t,\m)G\chi^*(\z),
\end{align*}
and recall the proof of Proposition \ref{cauchy} in where we replace $S\to S(t,\m)$ but keeping $G$ (not replacing it by $G(t,\m)$) to get
\begin{align*}
  \Psi^*_1(\z,t,\m)=H(t,\m)^{-1}\int_\Omega P(\y,t,\m)\dd\mu(\y)\frac{1}{(z_1-y_1)\cdots(z_D-y_D)},
\end{align*}
and we get the desired result.
\end{proof}
\subsection{Proof of Proposition \ref{pro10}}\label{proof10}
\begin{proof}
  From \eqref{totoW} we get
  \begin{align*}
S(t)W_0(t)G=H(t)(S(t)^{-1})^\top,
 \end{align*}
that, by differentiation, leads to
\begin{align*}
  \frac{\partial S}{\partial t_a}S^{-1}+ S\Lambda_aS^{-1}=\frac{\partial H}{\partial t_a}H^{-1}- H \Big( \frac{\partial S}{\partial t_a}S^{-1}\Big)^\top H^{-1}.
\end{align*}
 Then, we have
\begin{align*}
  (J_a)_{[k],[k]}=\beta_{[k]}(\Lambda_a)_{[k-1],[k]}-(\Lambda_a)_{[k],[k+1]}\beta_{[k+1]}&=\frac{\partial H_{[k]}}{\partial t_a}H_{[k]}^{-1},\\
   (J_a)_{[k],[k+1]}=(\Lambda_a)_{[k],[k+1]}&=-H_{[k]}\frac{\partial(\beta_{[k+1]})^\top}{\partial t_a}H_{[k+1]}^{-1}.
\end{align*}
\end{proof}
\subsection{Proof of Theorem \ref{teo3}}\label{proof11}
\begin{proof}
  We obviously have
  \begin{align*}
\Big[    \frac{\partial}{\partial \n_b}-T_b,   \frac{\partial}{\partial \n_a}-T_a\Big](W_1)=&0,&a,b&\in\{1,\dots,D\}.
  \end{align*}
  Recalling the proof of Proposition \ref{pro: yo que se} we can write
  \begin{multline*}
 \Big[    \frac{\partial}{\partial \n_b}-T_b,   \frac{\partial}{\partial \n_a}-T_a\Big](W_1)=\Big(\frac{\partial\Delta_a\beta(\n_a\cdot\boldsymbol\Lambda)}{\partial\n_b}-\frac{\partial\Delta_{b}\
 \beta(\n_b\cdot\boldsymbol\Lambda)}{\partial\n_a}\Big)(W_1)\\+
\big(-q_a+( \Delta_a\beta)(\n_a\cdot\boldsymbol\Lambda)\big)\frac{\partial W_1}{\partial\n_b}-\big(-q_b+ (\Delta_{b}\beta)(\n_b\cdot\boldsymbol\Lambda)\big)\frac{\partial W_1}{\partial\n_a}\\
 -(-q_a+T_b (\Delta_a\beta)(\n_a\cdot\boldsymbol\Lambda))(T_bW_1)+(-q_b+T_a (\Delta_{b}\beta)(\n_b\cdot\boldsymbol\Lambda))(T_aW_1).
  \end{multline*}
  To evaluate this expression we recall Proposition \ref{pro:wave} that splits it by diagonals, in fact we are dealing with a Hessenberg matrix with the first non vanishing diagonal the first superdiagonal where we find
\begin{multline*}
  (\Delta_a\beta)(\n_a\cdot\boldsymbol\Lambda)(\n_b\cdot\boldsymbol\Lambda)-
 (\Delta_{b} \beta)(\n_b\cdot\boldsymbol\Lambda)(\n_a\cdot\boldsymbol\Lambda)=
   (T_b\Delta_a\beta)(\n_a\cdot\boldsymbol\Lambda)(\n_b\cdot\boldsymbol\Lambda)
 -(T_a\Delta_{b}\beta)((\n_b\cdot\boldsymbol\Lambda)(\n_a\cdot\boldsymbol\Lambda)
\end{multline*}
or
\begin{align*}
  (\Delta_a\Delta_b\beta)(\n_a\cdot\boldsymbol\Lambda)(\n_b\cdot\boldsymbol\Lambda)
 =(\Delta_{b}\Delta_{a}\beta)((\n_b\cdot\boldsymbol\Lambda)(\n_a\cdot\boldsymbol\Lambda)
\end{align*}
which happens to be an identity. Next we look at the main diagonal where we have
\begin{align*}
 \Delta_b\left[\frac{\partial \beta}{\partial \boldsymbol{n}_a}+(\Delta_a \beta)\big(q_a+(\boldsymbol{n_a}\cdot \boldsymbol{\Lambda}) \beta \big) \right]\boldsymbol{n_b}\cdot \boldsymbol{\Lambda}&=
  \Delta_a\left[\frac{\partial \beta}{\partial \boldsymbol{n}_b}+(\Delta_b \beta)\big(q_b+(\boldsymbol{n_b}\cdot \boldsymbol{\Lambda}) \beta \big) \right]\boldsymbol{n_a}\cdot \boldsymbol{\Lambda}
\end{align*}
\end{proof}
\subsection{Proof of Theorem \ref{teo4}}\label{proof12}
\begin{proof}
   If we denote
\begin{align}\label{L2}
L_{a,b}&\coloneq \partial_a\partial_b +U_{a,b}, & a,b&\in\{1,\dots,D\},
\end{align}
 \eqref{eq: linear.wave} reads $ \partial_{a,b}(W_i)=L_{a,b}(W_i)$.
The compatibility conditions for this linear system are
\begin{align*}
\Big(  \partial_{(a,b)}(L_{c,d})-\partial_{(c,d)}(L_{a,b})+[L_{c,d},L_{a,b}]\Big)(W_i)&=0, & i&=1,2,
\end{align*}
and consequently
    \begin{align}\label{curvature}
R_{a,b,c,d}(W_i)&=0, & a,b,c,d\in\{1,\dots,D\},
    \end{align}
where
  \begin{multline*}
    R_{a,b,c,d}\coloneq (\partial_bU_{c,d})\partial_a+(\partial_aU_{c,d})\partial_b-
    (\partial_dU_{a,b})\partial_c-(\partial_cU_{a,b})\partial_d-
    \partial_{(a,b)}(U_{c,d})+\partial_{(c,d)}(U_{a,b})\\+[U_{c,d},U_{a,b}]-\partial_c\partial_d U_{a,b}+
\partial_a\partial_bU_{c,d}.
  \end{multline*}
Notice that we have that $W_1$ satisfies an equation of the form
      \begin{align*}
\big(    \sum_{j=1}^DB_j\partial_j+ A\big)(W_1)&=0.
  \end{align*}
Recalling \eqref{eq:partialW1}, we find that
  \begin{align*}
  0&= \big (\sum_{j=1}^DB_j\partial_j+ A\big)(W_1)=\Big(  \underbracket{\sum_{j=1}^DB_j\Lambda_j}_{\text{first superdiagonal}}+ \underbracket{\sum_{j=1}^D B_j\beta\Lambda_j+A}_{\text{ diagonal}}+\mathfrak
l\Big)W_0
  \end{align*}
  and, decoupling by diagonals,  we get
     \begin{align*}
    \sum_{j=1}^D B_j\Lambda_j&=0,&
    \sum_{j=1}^D B_j\beta\Lambda_j+A&=0.
  \end{align*}

      From \eqref{curvature} we find, in the first place, that
   \begin{align*}
      (\partial_bU_{c,d})\Lambda_a+(\partial_aU_{c,d})\Lambda_b-
    (\partial_dU_{a,b})\Lambda_c-(\partial_cU_{a,b})\Lambda_d&=0,
    \end{align*}
    which is identically satisfied because of \eqref{eta}. In the second place, we get the
    following  nonlinear equation
    \begin{multline*}
          (\partial_bU_{c,d})\beta \Lambda_a+(\partial_aU_{c,d})\beta \Lambda_b-
    (\partial_dU_{a,b})\beta \Lambda_c-(\partial_cU_{a,b})\beta \Lambda_d+    \partial_{(a,b)}U_{c,d}-\partial_{(c,d)}U_{a,b}\\+[U_{c,d},U_{a,b}]-\partial_c\partial_d U_{a,b}+
\partial_a\partial_bU_{c,d}=0,
    \end{multline*}
   and recalling \eqref{eta} we get the desired result.
\end{proof}
\subsection{Proof of Proposition \ref{moorepenroseright}}\label{proof moorepenroseright}
\begin{proof}
Observe that for $\n=R\ee_a$ we have \eqref{a invertir} that
\begin{align*}
M=(\n\cdot\boldsymbol\Lambda)\mathcal M^{-1/2}=\eta_R\Lambda_a\eta_R^{-1} \mathcal M^{-1/2},
\end{align*}
from where we deduce that
\begin{flalign*}
  MM^\top=&\eta_R\Lambda_a\eta_R^{-1} \mathcal M^{-1}\big(\eta_R^{-1}\big)^\top\Lambda_a^\top\eta_R^\top
  \\
  =&\eta_R\Lambda_a[\mathscr R]_{B_c}^\top \mathcal M^{-1}[\mathscr R]_{B_c}\Lambda_a^\top\eta_R^\top &
  \text{because \eqref{etarelations}}\\
  =& \eta_R\Lambda_a\mathcal M^{-1}\Lambda_a^\top\eta_R^\top &\text{from \eqref{ortocano}}
\end{flalign*}

We now introduce the matrix built up of multinomial coefficients involving $\ee_a$,
  \begin{align*}
   \mathcal M_{[k+1]_a}=\diag\bigg({k+1\choose \q^{(k+1)}_{j_1}},\dots,{k+1\choose \q^{(k+1)}_{j_{|[k]|}}}\bigg) \in\R^{|[k]|\times|[k]|},
  \end{align*}
  where $[k+1]_a\coloneq \big\{\q^{(k+1)}_{j_m}\big\}_{m=1}^{|[k]|}\subset [k+1]$ is the set containing only the multi-indices such that $\ee_a\cdot\q^{(k+1)}_{j_m}\neq 0$, assuming the reverse lexicographical order; notice that $|[k+1]_a|=|[k]|$.
  Then, we can write
  \begin{align*}
    (\Lambda_a)_{[k],[k+1]}\mathcal M_{[k+1]}^{-1}((\Lambda_a)_{[k],[k+1]})^\top=\mathcal M_{[k+1]_a}^{-1},
  \end{align*}
  which is clearly invertible and, consequently
  \begin{align*}
 M_{[k],[k+1]}(M_{[k],[k+1]})^\top=\eta_{R,[k]}\mathcal M_{[k+1]_a}^{-1}\eta_{R,[k]}^\top
\end{align*}
is invertible.

Thus, following Appendix \ref{appendix pseudoinverse}, we get the
the Moore--Penrose pseudo-inverse of the matrix $(\n\cdot\boldsymbol\Lambda)_{[k-1],[k]}\mathcal M_{[k]}^{-1/2}$ is
\begin{align*}
 \big( (\n\cdot\boldsymbol\Lambda)_{[k-1],[k]}\mathcal M_{[k]}^{-1/2}\big)^+=\mathcal M_{[k]}^{-1/2}((\n\cdot\boldsymbol\Lambda)_{[k-1],[k]}) ^\top\big(
 (\n\cdot\boldsymbol\Lambda)_{[k-1],[k]}\mathcal M_{[k]}^{-1} ( (\n\cdot\boldsymbol\Lambda)_{[k-1],[k]})^\top)
 \big)^{-1}
\end{align*}
which is the right inverse of the matrix, and therefore we get the result.
\end{proof}
\subsection{Proof of Proposition \ref{pro13}}\label{proof13}
\begin{proof}
  From Definition \ref{moment} we have
\begin{align*}
  G=&\int_{\R^D}\chi(\x)\dd\mu(\x)\chi(\x)^\top \\
  =&\int_{\R^D}\chi(R \x)\dd\mu(R\x)\chi(R\x)^\top \\
  =&\int_{\R^D} \eta_R\chi(\x)\dd\mu(\x)\chi(\x)^\top    \eta_R^\top\\
  =& \eta_R G    \eta_R^\top.
\end{align*}
From this formula and the Cholesky factorization we get
\begin{align*}
  SH(S^{-1})^\top =& \eta_R SH(S^{-1})^\top     \eta_R^\top\\
  =& \eta_R  S  \eta_R^{-1} \eta_R H  \eta_R^\top (\eta_R^\top)^{-1}(S^{-1})^\top     \eta_R^\top\\
  =& \big( \eta_R  S  \eta_R^{-1}\big) \big(\eta_R H  \eta_R^\top\big)\Big(\big( \eta_R  S  \eta_R^{-1}\big)^\top\Big)^{-1},
\end{align*}
and given the uniqueness of the Cholesky factorization and that $ \eta_R$ is block diagonal we get the stated result for $S$ and $H$. The equation for $\beta$ follows from the equation for $S$.
\end{proof}
\subsection{Proof Proposition \ref{latriple}}\label{prooftriple}
\begin{proof}
  For convenience we write here the next couple of equations
  \begin{align*}
  \partial_aW_1=&(\partial_aS+S\Lambda_a)W_0,\\
        ( \partial_{(a,b,c)}-    \partial_a\partial_b\partial_c)(W_1)=&
        \begin{multlined}[t]\big(\partial_{a,b,c}S-\partial_a\partial_b\partial_c S-\partial_a\partial_bS\Lambda_c-\partial_b\partial_cS\Lambda_a-\partial_c\partial_aS\Lambda_b\\-
    \partial_aS\Lambda_b\Lambda_c-\partial_bS\Lambda_c\Lambda_a-\partial_cS\Lambda_a\Lambda_b\big)W_0,
        \end{multlined}
  \end{align*}
Taking into account the form of $S=\I+\beta^{(1)}+\beta^{(2)}+\cdots$, being $\beta^{(k)}$ the $k$-th subdiagonal of $S$, we can write
  \begin{align}
  \partial_aW_1=&(\Lambda_a+\beta^{(1)}\Lambda_a)W_0+\mathfrak l W_0,\label{congruence 3.1}
  \\
        ( \partial_{(a,b,c)}-    \partial_a\partial_b\partial_c)(W_1)=&
        \begin{multlined}[t]
          -\big( \overbracket{ \partial_a\beta^{(1)}\Lambda_b\Lambda_c+\partial_b\beta^{(1)}\Lambda_c\Lambda_a+\partial_c\beta^{(1)}\Lambda_a\Lambda_b}^{\text{first superdiagonal}}
        \\ +\underbracket{\partial_a\beta^{(2)}\Lambda_b\Lambda_c+\partial_b\beta^{(2)}\Lambda_c\Lambda_a
        +\partial_c\beta^{(2)}\Lambda_a\Lambda_b}_{\text{diagonal}}\\+\underbracket{
        \partial_a\partial_b\beta^{(1)}\Lambda_c
        +\partial_b\partial_c\beta^{(1)}\Lambda_a+\partial_c\partial_a\beta^{(1)}\Lambda_b}_{\text{diagonal}}
  \big)W_0\\+\mathfrak l W_0,
        \end{multlined}\label{congruence 3.2}
  \end{align}
  We can use \eqref{congruence 3.1} to  move to the RHS of \eqref{congruence 3.2} the contribution on the first superdiagonal so that
  \begin{multline*}
      ( \partial_{(a,b,c)}-    \partial_a\partial_b\partial_c+\partial_a\beta^{(1)}\Lambda_b\partial_c+\partial_b\beta^{(1)}\Lambda_c\partial_a+\partial_c\beta^{(1)}\Lambda_a\partial_b)(W_1)\\=
          -\big(  \overbracket{\partial_a\beta^{(2)}\Lambda_b\Lambda_c+\partial_b\beta^{(2)}\Lambda_c\Lambda_a+\partial_c\beta^{(2)}\Lambda_a\Lambda_b+
        \partial_a\partial_b\beta^{(1)}\Lambda_c+\partial_b\partial_c\beta^{(1)}\Lambda_a+\partial_c\partial_a\beta^{(1)}\Lambda_b}^{\text{diagonal}}
        \\\underbracket{-\partial_a\beta^{(1)}\Lambda_b\beta^{(1)}\Lambda_c-\partial_b\beta^{(1)}\Lambda_c\beta^{(1)}\Lambda_a-\partial_c\beta^{(1)}\Lambda_a\beta^{(1)}\Lambda_b}_{\text{diagonal}}
  \big)W_0+\mathfrak l W_0,
  \end{multline*}
Which using Proposition \ref{otro mas} can be written as\footnote{This could be avoided, depending on whether or not we desired to use $\beta^{(2)}$ in the expressions  for $ V_{a,b,c}$.}
  \begin{multline*}
      ( \partial_{(a,b,c)}-    \partial_a\partial_b\partial_c+\partial_a\beta^{(1)}\Lambda_b\partial_c+\partial_b\beta^{(1)}\Lambda_c\partial_a+\partial_c\beta^{(1)}\Lambda_a\partial_b)(W_1)\\=
          -\big(  \overbracket{\partial_a\partial_b\beta^{(1)}\Lambda_c+\partial_b\partial_c\beta^{(1)}\Lambda_a
          +\partial_c\partial_a\beta^{(1)}\Lambda_b}^{\text{diagonal}}
        \\\underbracket{+(\partial_a\beta^{(1)})\beta^{(1)}\Lambda_b\Lambda_c+(\partial_b\beta^{(1)})\beta^{(1)}\Lambda_c\Lambda_a
        +(\partial_c\beta^{(1)})\beta^{(1)}\Lambda_a\Lambda_b}_{\text{diagonal}}\\\underbracket{
        -\partial_a\beta^{(1)}\Lambda_b\beta^{(1)}\Lambda_c-\partial_b\beta^{(1)}\Lambda_c\beta^{(1)}\Lambda_c
        -\partial_c\beta^{(1)}\Lambda_a\beta^{(1)}\Lambda_b}_{\text{diagonal}}
  \big)W_0+\mathfrak l W_0,
  \end{multline*}
  which after simplifying and writing $\beta^{(1)}=\beta$
    \begin{multline*}
      ( \partial_{(a,b,c)}-    \partial_a\partial_b\partial_c+\partial_a\beta\Lambda_b\partial_c+\partial_b\beta\Lambda_c\partial_a+\partial_c\beta\Lambda_a\partial_b)(W_1)\\=
          -\big( \partial_a\partial_b\beta\Lambda_c+\partial_b\partial_c\beta\Lambda_a+\partial_c\partial_a\beta\Lambda_b
    +(\partial_a\beta)[\beta,\Lambda_b]\Lambda_c+(\partial_b\beta)[\beta,\Lambda_c]\Lambda_a+(\partial_c\beta)[\beta,\Lambda_a]\Lambda_b
  \big)W_0+\mathfrak l W_0.
  \end{multline*}
Hence, from  \eqref{eta}
  \begin{multline*}
     R_1\coloneq  ( \partial_{(a,b,c)}-    \partial_a\partial_b\partial_c+ V_{a,b}\partial_c+ V_{b,c}\partial_a+ V_{c,a}\partial_b\\+\partial_c( V_{a,b})+\partial_a( V_{b,c})+\partial_b( V_{c,a})+ V_{a,b,c}+ V_{b,c,a}+ V_{c,b,a})(W_1)\in\mathfrak l W_0,
  \end{multline*}
and trivially we know that
 \begin{multline*}
     R_2\coloneq  ( \partial_{(a,b,c)}-    \partial_a\partial_b\partial_c+ V_{a,b}\partial_c+ V_{b,c}\partial_a+ V_{c,a}\partial_b\\+\partial_c( V_{a,b})+\partial_a( V_{b,c})+\partial_b( V_{c,a})+ V_{a,b,c}+ V_{b,c,a}+ V_{c,b,a})(W_2)\in\mathfrak u,
  \end{multline*}
  where $R_1G=R_2$. Consequently, from the \emph{asymptotic module}  Proposition \ref{pro:asymptotic-module} we deduce that $R_1=R_2=0$.
  Therefore,
   \begin{multline*}
      \frac{\partial\Psi_i}{\partial t_{(a,b,c)}}=\frac{\partial^3\Psi_i}{\partial t_a\partial t_b\partial t_c}
      - V_{a,b}\frac{\partial\Psi_i}{\partial t_c}- V_{b,c}\frac{\partial\Psi_i}{\partial t_a}- V_{c,a}\frac{\partial\Psi_i}{\partial t_b}\\-\Big(\partial_c( V_{a,b})+\partial_a( V_{b,c})+\partial_b( V_{c,a})+ V_{a,b,c}+ V_{b,c,a}+ V_{c,b,a}\Big)\Psi_i.
    \end{multline*}
\end{proof}
\subsection{Proof of Proposition \ref{pro14}}\label{proof14}
\begin{proof}
  To prove \eqref{symmetri-MOVPR} for the MOVPR observe that
  \begin{align*}
    P(R\x)=&S \eta_R\chi(\x)\\=& \eta_R S\chi(\x)\\=& \eta_RP(\x)
  \end{align*}
  which together with  \eqref{varsigma} leads to the result.
  To check  \eqref{symmetry-jacobi} just follow the next equalities
 \begin{align*}
   \eta_R(\n\cdot \boldsymbol J)\eta_R^{-1}=&\eta_R S(\n\cdot\boldsymbol\Lambda)S^{-1}\eta_R^{-1}
   \\
   =&S\eta_R(\n\cdot\boldsymbol\Lambda)\eta_R^{-1}S^{-1}
   \\
   =& S(R\n\cdot \boldsymbol \Lambda) S^{-1}
   \\
   =& R\n\cdot\boldsymbol J.
 \end{align*}

 Equations \eqref{RCD} are a direct consequence of \eqref{DEFCD} and \eqref{symmetry-S}.
\end{proof}
\end{appendices}

\end{document}